\documentclass[preprint,12pt]{elsarticle}

\usepackage{graphicx}
\usepackage{amssymb}

\usepackage{lineno}

\usepackage{amsmath}
\usepackage{xcolor}
\usepackage{subfigure}
\usepackage{float}
\usepackage{multirow}
\usepackage{makecell}
\usepackage{mathtools}
\usepackage{amssymb}
\usepackage{amsthm}
\usepackage{bm}
\usepackage{mathrsfs}
\usepackage{amscd}

\DeclareMathOperator\supp{supp}

\DeclareMathOperator*{\argmin}{\arg\!\min}
\DeclareMathOperator{\sgn}{sgn}

\newcommand{\mcL}{\mathcal{L}}

\newcommand{\mbR}{\mathbb{R}}

\newcommand{\mbRn}{{\mathbb{R}^n}}

\def \xb{{\bf x}}
\def \yb{{\bf y}}

\def \xb{\mathbf{x}}
\def \xbj{{\xb_j}}

\def \yb{\mathbf{y}}

\newtheorem{thm}{Theorem}[section]
\newtheorem{lem}[thm]{Lemma}

\newtheorem{rem}[thm]{Remark}

\journal{...}

\makeatletter
\def\ps@pprintTitle{%
   \let\@oddhead\@empty
   \let\@evenhead\@empty
   \def\@oddfoot{\reset@font\hfil\thepage\hfil}
   \let\@evenfoot\@oddfoot
}
\makeatother

\newcommand{\vertiii}[1]{{\left\vert\kern-0.25ex\left\vert\kern-0.25ex\left\vert #1 
    \right\vert\kern-0.25ex\right\vert\kern-0.25ex\right\vert}}

\begin{document}

\begin{frontmatter}


\title{Efficient optimization-based quadrature for variational discretization of nonlocal problems}



\author[ucsd-mae]{Marco Pasetto\corref{cor1}}\ead{mpasetto@ucsd.edu}
\author[unilu]{Zhaoxiang Shen}
\author[sandia-ca]{Marta D’Elia}
\author[ucsd-math]{Xiaochuan Tian}
\author[sandia-nm]{Nathaniel Trask}
\author[ucsd-mae]{David Kamensky}
\cortext[cor1]{Corresponding author}

\address[ucsd-mae]{Department of Mechanical and Aerospace Engineering, University of California, San Diego, La Jolla, CA 92093, USA}
\address[unilu]{Department of Engineering, Faculty of Science, Technology and Medicine, University of Luxembourg, Esch-sur-Alzette 4365, Luxembourg}
\address[sandia-ca]{Computational Science and Analysis, Sandia National Laboratories, Livermore, CA 94550, USA}
\address[ucsd-math]{Department of Mathematics, University of California, San Diego, La Jolla, CA 92093, USA}
\address[sandia-nm]{Center for Computing Research, Sandia National Laboratories, Albuquerque, NM 87185, USA}

\begin{abstract}
Casting nonlocal problems in variational form and discretizing them with the finite element (FE) method facilitates the use of nonlocal vector calculus to prove well-posedeness, convergence, and stability of such schemes. Employing an FE method also facilitates meshing of complicated domain geometries and coupling with FE methods for local problems. However, nonlocal weak problems involve the computation of a double-integral, which is computationally expensive and presents several challenges.  In particular, the inner integral of the variational form associated with the stiffness matrix is defined over the intersections of FE mesh elements with a ball of radius $\delta$, where $\delta$ is the range of nonlocal interaction. Identifying and parameterizing these intersections is a nontrivial computational geometry problem. In this work, we propose a quadrature technique where the inner integration is performed using quadrature points distributed over the full ball, without regard for how it intersects elements, and weights are computed based on the generalized moving least squares method. Thus, as opposed to all previously employed methods, our technique does not require element-by-element integration and fully circumvents the computation of element--ball intersections. This paper considers one- and two-dimensional implementations of piecewise linear continuous FE approximations, focusing on the case where the element size $h$ and the nonlocal radius $\delta$ are proportional, as is typical of practical computations. When boundary conditions are treated carefully and the outer integral of the variational form is computed accurately, the proposed method is asymptotically compatible in the limit of $h\sim\delta\to 0$, featuring at least first-order convergence in $L^2$ for all dimensions, using both uniform and nonuniform grids. Moreover, in the case of uniform grids, the proposed method passes a patch test and, according to numerical evidence, exhibits an optimal, second-order convergence rate. Our numerical tests also indicate that, even for nonuniform grids, second-order convergence can be observed over a substantial pre-asymptotic regime.
\end{abstract}

\begin{keyword}
Nonlocal models \sep finite element discretizations \sep generalized moving least squares \sep numerical quadrature \sep peridynamics


\end{keyword}

\end{frontmatter}



\section{Introduction}
\label{sec:Introduction}
Nonlocal models have become viable alternatives to partial differential equations (PDEs) for applications where small-scale effects affect the global behavior of a system or when discontinuities in the quantity of interest make it impractical to use differential operators. In fact, nonlocal operators embed length scales in their definitions and allow for irregular functions.  For these reasons, nonlocal models are currently employed in several scientific and engineering applications including surface or subsurface transport 
\cite{Benson2000,Benson2001,d2021analysis,Deng2004,Schumer2003,Schumer2001},
fracture mechanics
\cite{Ha2011,Littlewood2010,Silling2000},
turbulence
\cite{Scalar_FSGS,DiLeoni-2020,Pang2020},
image processing
\cite{Buades2010,DElia2019imaging,Gilboa2007}
and stochastic processes
\cite{Burch2014,DElia2017,Meerschaert2012,Metzler2000,Metzler2004}.

Nonlocal operators are integral operators that embed length scales in the domain of integration; as such, they allow one to model long-range forces within the length scale and to reduce the regularity requirements on the solutions. The most general form of nonlocal Laplace operator is given by \cite{DElia2020Unified}

\begin{equation*}
\mcL_\delta u(\xb) = 2\int_\mbRn 
(u(\yb)-u(\xb)) \gamma(\xb,\yb)\,d\yb,
\end{equation*}
where $u:\mbRn\to\mbR$ is a scalar function and $\gamma$ is a symmetric {\it kernel} function whose support is $\mathscr{H}{(\mathbf{x},\delta)}$, the ball centered at $\xb$ of radius $\delta$, the so-called horizon or interaction radius. In most cases, the ball is understood in the Euclidean sense (to maintain rotational invariance), but recent works also employ more general balls, including $\ell^\infty$ balls, see, e.g. \cite{capodaglio2019,xu2021feti,xu2021machine}. The function $\gamma$ determines the function space that the nonlocal solution belongs to. Its choice is nontrivial and non-intuitive; in fact, the selection of the optimal kernel is a widely studied research question \cite{Pang2020,burkovska2020,DElia2014DistControl,DElia2016ParamControl,Gulian2019,Pang2019fPINNs,Pang2017discovery,Xu2020learning,You2020Regression,You2020aaai,you2021data}. 

Because of the integral nature of nonlocal operators, the discretization and numerical solution of nonlocal equations raises several unresolved challenges. These include the design of accurate and efficient discretization schemes and the development of efficient numerical solvers \cite{AinsworthGlusa2018,Capodaglio2020DD,acta20,DEliaFEM2020,Pasetto2019,Pasetto2018,silling2005meshfree,Wang2010}. With the ultimate goal of easily handling nontrivial domains and possibly using mesh adaptivity, this work focuses on variational discretizations and, specifically, the finite element method. However, we point out that the nonlocal literature offers a broad class of meshfree techniques, widely used at the engineering level. We refer the interested reader to, e.g., \cite{Pasetto2018,silling2005meshfree,Chen2006meshless,parks2012peridigm,parks2010lammps,trask2019asymptotically}. One advantage of the FE method is that the nonlocal vector calculus facilitates its numerical analysis. This theoretical framework, first introduced in \cite{Gunzburger2010}, further developed in \cite{Du2013}, and generalized in \cite{DElia2020Unified}, allows us to cast nonlocal equations in a variational setting and analyze them in the same way as PDEs. Using this framework, one can prove well-posedness, convergence, and stability of nonlocal FE schemes. Nonetheless, variational discretizations introduce further computational challenges due to the presence of an additional integration over the domain of the problem. In fact, the nonlocal weak problem associated with the operator $\mcL_\delta$ involves the computation of a double integral. Specifically, the core computation required by standard codes to assemble the FE stiffness matrix is given by an integral of the form

\begin{equation}\label{eq:double-int}
\int_{\Omega^h_i}\int_{\Omega^h_j\cap \mathscr{H}{(\mathbf{x},\delta)}}
[\psi_k(\yb)-\psi_k(\xb)]\gamma(\xb,\yb)[\psi_l(\yb)-\psi_l(\xb)]\,d\yb\,d\xb,
\end{equation}
where $\Omega^h_i$ is the $i$-th element of the partition and $\psi_k$ is the $k$-th FE basis function.  (See Section \ref{sec:discrete-form} for a complete formulation.) In the formula above, we have purposely written the inner domain of integration explicitly, to highlight the fact that, prior to numerically evaluating the integral, we must identify the region of the $j$-th element that overlaps with the support of the kernel function, since a na\"ive, global integration over the whole element would not guarantee numerical convergence of the overall scheme. Identifying this region is a nontrivial, time-consuming task. Furthermore, the presence of a double integral inevitably adds computational cost and it is often the case that the integrand function is singular, requiring the use of sophisticated, possibly adaptive quadrature rules. 

A thorough description of the computational challenges that arise in the computation of the double integral \eqref{eq:double-int} can be found in \cite{DEliaFEM2020}. For the case of finite horizon, the authors of \cite{DEliaFEM2020} propose efficient ways to circumvent the problem of finding intersections between FEs and nonlocal neighborhoods by introducing the concept of ``approximate balls'' given by FE patches that roughly approximate $\mathscr{H}{(\mathbf{x},\delta)}$; their results indicate that in the case of piecewise linear FE spaces, optimal numerical convergence can be preserved. Alternatively, in \cite{aulisa2021efficient}, the authors propose a technique that allows one to compute the inner integral over the whole element $\Omega^h_j$ by introducing a smoothing of the kernel function. The smoothed kernel is still compactly supported, but it continuously decays to zero, allowing for simple Gaussian quadrature rules over each FE. 

In this work, under the assumption that the discretization parameter $h$ (the size of the FE) and the nonlocal radius $\delta$ are proportional, we propose a change of perspective and introduce a technique where the inner integration is performed over the ball $\mathscr{H}{(\mathbf{x},\delta)}$ rather than on a single element, i.e., the core computation in the stiffness matrix assembly now becomes
\begin{equation}\label{eq:double-int-B}
\int_{\Omega^h_i}\int_{\mathscr{H}{(\mathbf{x},\delta)}}
[\psi_k(\yb)-\psi_k(\xb)]\gamma(\xb,\yb)[\psi_l(\yb)-\psi_l(\xb)]\,d\yb\,d\xb,
\end{equation}
where we utilize special quadrature rules for the numerical computation of the inner integral. Specifically, we consider quadrature rules based on the generalized moving least squares (GMLS) method, successfully used for strong-form meshfree discretizations of nonlocal problems in \cite{trask2019asymptotically,Leng2021_AsymptoticallyCR,Gross2020}. 
The main idea behind this approach is to determine the quadrature weights associated with quadrature points (the meshfree discretization nodes in a meshfree setting) by solving an equality constrained optimization problem (see Section \ref{sec:GMLS_construction} for a thorough discussion).
The introduction of a technique that fully circumvents the computation of element--ball intersections and that allows for the use of global quadrature rules over the support of the kernel function is the major contribution of this work. Additionally, the technique we propose requires minimal implementation effort, as the GMLS subroutine can be embedded into an existing FE code. As such, we envision this technique as a key component of agile FE engineering codes. In this work, we consider one- and two-dimensional implementations of piecewise linear continuous FE approximations. When boundary conditions are carefully treated and when the outer integral in \eqref{eq:double-int-B} is accurately computed, this method is asymptotically compatible in the limit of $h$ and $\delta$ vanishing and features first-order convergence in the $L^2$ norm for all dimensions and for both uniform and nonuniform grids. Furthermore, in the case of uniform grids, the proposed method is patch-test consistent (i.e., it is machine-precision accurate for linear solutions) and, according to numerical evidence, features an optimal, second-order convergence rate. Our numerical tests also indicate that, even for nonuniform grids, second-order convergence may be observed in the pre-asymptotic regime. Another contribution of the current work is a preliminary theoretical study, where, in a simplified, uniformly-discretized, one-dimensional setting, we show that the proposed method features optimal first-order numerical convergence in the $H^1$ norm. These results set the groundwork for more rigorous studies that we will pursue in future works. 

\paragraph{Paper outline} Section \ref{sec:Nonlocal_diff_model} introduces the nonlocal Laplace operator and the corresponding volume-constrained nonlocal problem in its strong and weak form. Section \ref{sec:GMLS_construction} introduces the GMLS technique for the numerical evaluation of general integrals. In Section \ref{sec:discrete-form}, we formulate the discrete variational problem for a FE discretization and we provide a detailed description of the quadrature rules and the resulting, fully-discrete problem. We also introduce a technique for the treatment of nonlocal boundary conditions that guarantees an improved convergence behavior. In Section \ref{sec:convergence}, we introduce the concept of asymptotic compatibility and prove that, under certain assumptions, the proposed method features linear numerical convergence in the $H^1$ norm with respect to the mesh size (and, as a consequence, with respect to the interaction radius). Section \ref{sec:numerics} illustrates the accuracy of the proposed method with several one- and two-dimensional numerical tests on uniform and non-uniform grids using piecewise linear continuous FE discretizations. 
We first show the improved convergence behavior induced by the special treatment of the nonlocal boundary condition and then show that, in the $L^2$ norm, our scheme is second-order accurate for uniform discretizations and at least first-order accurate for non-uniform ones, with respect to $h$ and $\delta$ and is, hence, asymptotically compatible. Moreover, we also show that convergence in the $H^1$ norm is consistent with the theoretical predictions discussed in Section \ref{sec:convergence}. Lastly, we make some concluding remarks in Section \ref{sec:conclusion}. 

\section{Nonlocal Laplace operator and model problem}\label{sec:Nonlocal_diff_model}

In this section we set the notation that will be used throughout the paper and introduce relevant definitions and results. In particular, we formulate the strong and weak forms of the nonlocal Poisson problem used to describe the technique proposed in this work. 

Let $\gamma(\mathbf{x},\mathbf{y}):\mathbb{R}^d\times\mathbb{R}^d\rightarrow\mathbb{R}^{+}_0$ be a symmetric, i.e., $\gamma(\mathbf{x},\mathbf{y})=\gamma(\mathbf{y},\mathbf{x})$, non-negative kernel\footnote{Examples and analysis of nonsymmetric and sign-changing kernels can be found in \cite{d2017nonlocal,felsinger2015dirichlet} and \cite{mengesha2013analysis}, respectively.} with bounded support in the norm-induced ball of radius $\delta$, i.e.,

\begin{equation}\label{ball}
\mathscr{H}{(\mathbf{x},\delta)}\coloneqq\supp(\gamma(\mathbf{x},\cdot))=
\left \{\mathbf{y}\in {\mathbb{R} ^{d}}:
\lvert \mathbf{y}-\mathbf{x} \rvert_{\ell^{\tilde{p}}} \leq \delta
\right \},    
\end{equation}
where $\delta>0$ is referred to as the \textit{horizon} and $\tilde{p}\in\left[1,\infty\right]$. In this work, without loss of generality, we consider Euclidean balls, i.e., we take $\tilde{p}=2$. Furthermore, we restrict ourselves to kernels of the form

\begin{equation}\label{kernel_form1}
\gamma{(\mathbf{x},\mathbf{y})}=\left\{\begin{aligned}
         \ \frac{\zeta}{\delta^{d+2}} \quad \ &\rm{for}\ \lvert\mathbf{y}-\mathbf{x}\rvert_{\ell^{\tilde{p}}}\leq\delta,\\\
        \ 0 \ \ \quad\ &\rm{for}\ \lvert\mathbf{y}-\mathbf{x}\rvert_{\ell^{\tilde{p}}}>\delta,\\
\end{aligned}\right.  
\end{equation}
and

\begin{equation}\label{kernel_form2}
\gamma{(\mathbf{x},\mathbf{y})}=\left\{\begin{aligned}
         \ \frac{\zeta}{\delta^{d+1}\lvert \mathbf{y}-\mathbf{x} \rvert_{\ell^{\tilde{p}}}} \quad \ &\rm{for}\ \lvert\mathbf{y}-\mathbf{x}\rvert_{\ell^{\tilde{p}}}\leq\delta,\\\
        \ 0 \ \ \quad\ &\rm{for}\ \lvert\mathbf{y}-\mathbf{x}\rvert_{\ell^{\tilde{p}}}>\delta,\\
\end{aligned}\right.  
\end{equation}

\noindent
with $\zeta\in\mathbb{R}^+$. 


Let $\Omega\subset\mathbb{R}^d$ be a bounded open domain. Its associated \textit{interaction domain} is defined as the set of points outside of $\Omega$ that interact with points inside of it (see Figure \ref{Interaction_domain}), i.e.,

\begin{equation}\label{interaction domain}
\mathscr{B}\Omega\coloneqq
\left \{\mathbf{y}\in {\mathbb{R} ^{d}}\setminus\Omega:
\exists \mathbf{x}\in\Omega \; \text{such that} \;\lvert \mathbf{y}-\mathbf{x} \rvert_{\ell^{\tilde{p}}} \leq \delta
\right \}.    
\end{equation}
Note that $\mathscr{B}\Omega\cap\partial\Omega=\partial\Omega$, where $\partial\Omega$ is the boundary of $\Omega$ \cite{delia2020cookbook}. 

\begin{figure}[H] 
\begin{center}
\scalebox{0.75}{\includegraphics[trim = 50mm 375mm 50mm 375mm, clip=true,width=1\textwidth]{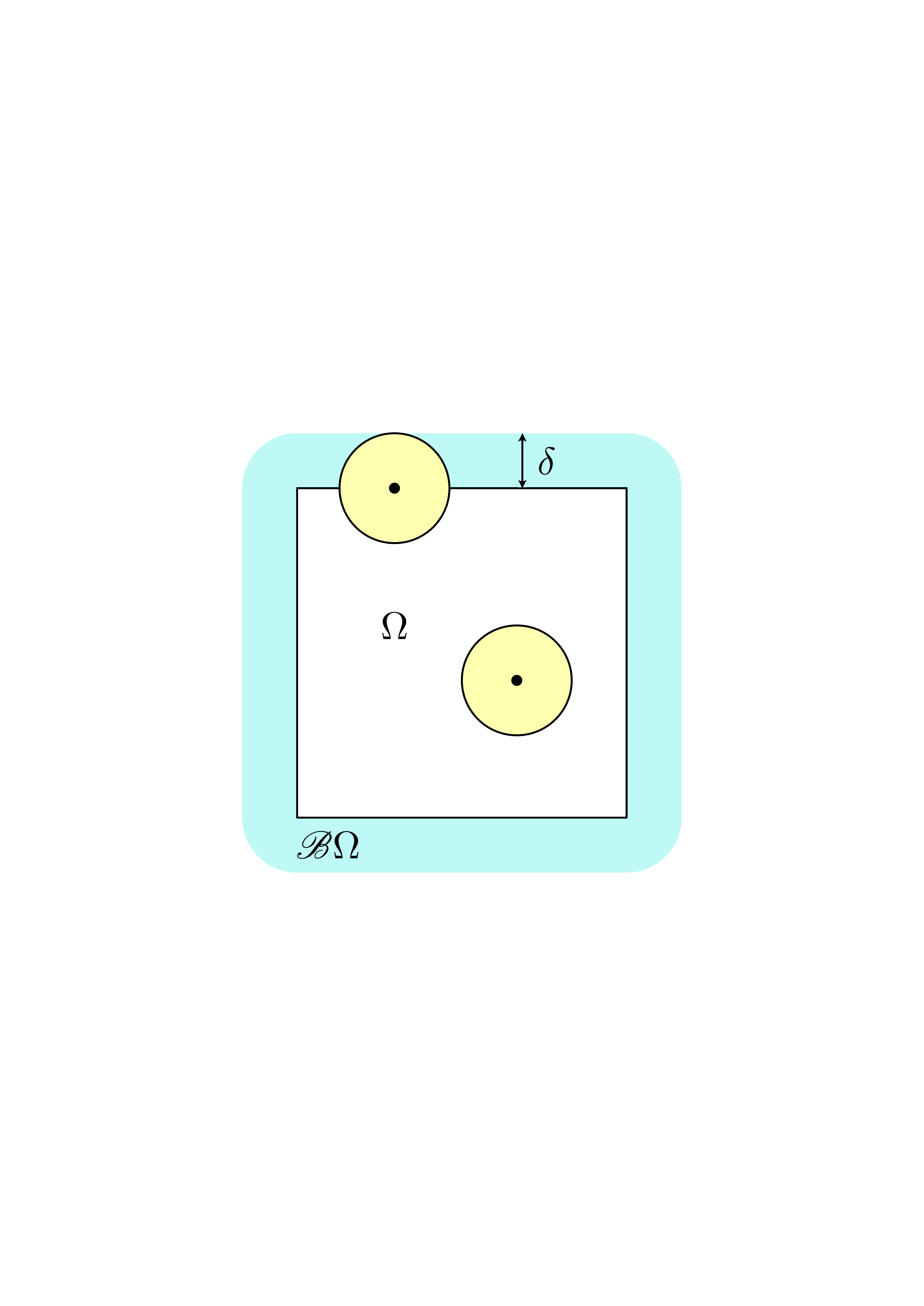}} 
\caption{A square domain $\Omega$ (in white) with its corresponding interaction domain of thickness $\delta$ (in light-blue). In yellow, two balls of radius $\delta$, centred at two points in $\Omega\cup\mathscr{B}\Omega$, depicted by black dots, one of which is in $\Omega$, while the other is located on the boundary $\partial\Omega$ between $\Omega$ and $\mathscr{B}\Omega$.}
\label{Interaction_domain}
\end{center}
\end{figure}

We introduce the strong form of a nonlocal volume-constrained Poisson problem \cite{Leng2021_AsymptoticallyCR,Du2012_NonlocalVolumes,Gunzburger2010_NonlocalVolumes2,DElia2020_Acta}: given $b:{\Omega}\rightarrow{\mathbb{R}}$ and $g:\mathscr{B}\Omega\rightarrow{\mathbb{R}}$, find $u:\Omega\cup\mathscr{B}\Omega\rightarrow{\mathbb{R}}$, such that

\begin{equation}
\left\{\begin{aligned}
    -\mathcal{L}_{\delta}u(\mathbf{x})
    &=b(\mathbf{x}), \quad \ \mathbf{x}\in \Omega,\\
    u(\mathbf{x})&=g(\mathbf{x}), \quad\  \mathbf{x}\in\mathscr{B}\Omega,
\end{aligned}\right.
    \label{strong form_NLD.}
\end{equation}
where $\mathcal{L}_{\delta}u(\mathbf{x})$ is the nonlocal Laplacian

\begin{equation}
\begin{aligned}
    \mathcal{L}_{\delta}u(\mathbf{x})
    &=2\int_{\Omega\cup\mathscr{B}\Omega}\gamma(\mathbf{x},\mathbf{y})(u(\mathbf{y})-u(\mathbf{x}))d\mathbf{y},
\end{aligned}
    \label{nonlocal_Laplacian.}
\end{equation}
and where the second equation is a \textit{Dirichlet volume constraint}. In this work, we only consider Dirichlet constraints\footnote{Examples of the numerical treatment of Neumann constraints can be found in, e.g., \cite{d2020physically,d2021prescription}}.

\subsection{Weak formulation}\label{weak_form}
To derive the weak formulation associated with the problem in Eq.~(\ref{strong form_NLD.}), we multiply the first equation in (\ref{strong form_NLD.}) by a test function $v(\mathbf{x}):\Omega\cup\mathscr{B}\Omega\rightarrow\mathbb{R}$ and then integrate over $\Omega$, i.e.

\begin{equation}
\begin{aligned}
    0&=\int_{\Omega}v(\mathbf{x})\left[-\mathcal{L}_{\delta}u(\mathbf{x})-b(\mathbf{x})\right]d\mathbf{x}\\
    &=-2\int_{\Omega}v(\mathbf{x})\int_{\Omega\cup\mathscr{B}\Omega}\gamma(\mathbf{x},\mathbf{y})(u(\mathbf{y})-u(\mathbf{x}))d\mathbf{y}d\mathbf{x}-\int_{\Omega}v(\mathbf{x})b(\mathbf{x})d\mathbf{x}.
\end{aligned}
    \label{weak_form_s1}
\end{equation}
Now we recast $-2\int_{\Omega}v(\mathbf{x})\int_{\Omega\cup\mathscr{B}\Omega}\gamma(\mathbf{x},\mathbf{y})(u(\mathbf{y})-u(\mathbf{x}))d\mathbf{y}d\mathbf{x}$ as

\begin{equation}
\begin{aligned}
    &-2\int_{\Omega}v(\mathbf{x})\int_{\Omega\cup\mathscr{B}\Omega}\gamma(\mathbf{x},\mathbf{y})(u(\mathbf{y})-u(\mathbf{x}))d\mathbf{y}d\mathbf{x}\\
    =&-2\int_{\Omega}v(\mathbf{x})\int_{\Omega\cup\mathscr{B}\Omega}\frac{1}{2}\left[\gamma(\mathbf{x},\mathbf{y})(u(\mathbf{y})-u(\mathbf{x}))-\gamma(\mathbf{x},\mathbf{y})(u(\mathbf{x})-u(\mathbf{y}))\right]d\mathbf{y}d\mathbf{x}\\
    =&-\int_{\Omega}v(\mathbf{x})\int_{\Omega\cup\mathscr{B}\Omega}\left[\gamma(\mathbf{x},\mathbf{y})(u(\mathbf{y})-u(\mathbf{x}))-\gamma(\mathbf{y},\mathbf{x})(u(\mathbf{x})-u(\mathbf{y}))\right]d\mathbf{y}d\mathbf{x},
\end{aligned}
    \label{weak_form_s2}
\end{equation}
where we employed the symmetry of $\gamma$.
As is standard in the presence of Dirichlet conditions, we require $v(\mathbf{x})$ to be zero on $\mathscr{B}\Omega$. We then apply \textit{Green's first identity of nonlocal vector calculus} \cite{Gunzburger2010_NonlocalVolumes2} to the term in Eq.~(\ref{weak_form_s2}), which gives us, with $v(\mathbf{x})=0$ for $\mathbf{x}\in\mathscr{B}\Omega$,

\begin{equation}
\begin{aligned}
    &-\int_{\Omega}v(\mathbf{x})\int_{\Omega\cup\mathscr{B}\Omega}\left[\gamma(\mathbf{x},\mathbf{y})(u(\mathbf{y})-u(\mathbf{x}))-\gamma(\mathbf{y},\mathbf{x})(u(\mathbf{x})-u(\mathbf{y}))\right]d\mathbf{y}d\mathbf{x}\\
    =&\int_{\Omega\cup\mathscr{B}\Omega}\int_{\Omega\cup\mathscr{B}\Omega}\left[v(\mathbf{y})-v(\mathbf{x})\right]\gamma(\mathbf{x},\mathbf{y})\left[u(\mathbf{y})-u(\mathbf{x})\right]d\mathbf{y}d\mathbf{x}.
\end{aligned}
    \label{weak_form_s3}
\end{equation}
Therefore, by combining Eq.~(\ref{weak_form_s2}) and Eq.~(\ref{weak_form_s3}) we get

\begin{equation}
\begin{aligned}
    &-2\int_{\Omega}v(\mathbf{x})\int_{\Omega\cup\mathscr{B}\Omega}\gamma(\mathbf{x},\mathbf{y})(u(\mathbf{y})-u(\mathbf{x}))d\mathbf{y}d\mathbf{x}\\
    =&\int_{\Omega\cup\mathscr{B}\Omega}\int_{\Omega\cup\mathscr{B}\Omega}\left[v(\mathbf{y})-v(\mathbf{x})\right]\gamma(\mathbf{x},\mathbf{y})\left[u(\mathbf{y})-u(\mathbf{x})\right]d\mathbf{y}d\mathbf{x}.
\end{aligned}
    \label{weak_form_s4}
\end{equation}
By substituting the latter in Eq.~(\ref{weak_form_s1}), we obtain

\begin{equation}
\begin{aligned}
    &\int_{\Omega\cup\mathscr{B}\Omega}\int_{\Omega\cup\mathscr{B}\Omega}\left[v(\mathbf{y})-v(\mathbf{x})\right]\gamma(\mathbf{x},\mathbf{y})\left[u(\mathbf{y})-u(\mathbf{x})\right]d\mathbf{y}d\mathbf{x}
    &=\int_{\Omega}v(\mathbf{x})b(\mathbf{x})d\mathbf{x}.
\end{aligned}
    \label{weak_form_s5}
\end{equation}
By defining the bilinear form $D(\cdot,\cdot)$ and the linear functional $G(\cdot)$ as

\begin{equation}
\begin{aligned}
    D(u,v)\coloneqq\int_{\Omega\cup\mathscr{B}\Omega}\int_{\Omega\cup\mathscr{B}\Omega}\left[v(\mathbf{y})-v(\mathbf{x})\right]\gamma(\mathbf{x},\mathbf{y})\left[u(\mathbf{y})-u(\mathbf{x})\right]d\mathbf{y}d\mathbf{x},
\end{aligned}
    \label{weak_form_s6}
\end{equation}
and

\begin{equation}
\begin{aligned}
    G(v)\coloneqq\int_{\Omega}v(\mathbf{x})b(\mathbf{x})d\mathbf{x},
\end{aligned}
    \label{weak_form_s7}
\end{equation}
we can rewrite Eq.~(\ref{weak_form_s5}) as

\begin{equation}
\begin{aligned}
    D(u,v)=G(v).
\end{aligned}
    \label{weak_form_s8}
\end{equation}

In double integral operators of the form $\int\left(\int\left(...\right)d\mathbf{y}\right)d\mathbf{x}$, we refer to $\int\left(...\right)d\mathbf{y}$ as the \textit{inner integral}, and to $\int\left(...\right)d\mathbf{x}$ as the \textit{outer integral}.

We define the following function spaces for functions $w(\mathbf{x})$ defined for $\mathbf{x}\in\Omega\cup\mathscr{B}\Omega$:

\begin{equation}
\begin{aligned}
    \mathcal{V}(\Omega\cup\mathscr{B}\Omega)\coloneqq\left\{w\in L^2(\Omega\cup\mathscr{B}\Omega):\vertiii{w}<\infty\right\},
\end{aligned}
    \label{func_spaces1}
\end{equation}
where we define the norm

\begin{equation}
\begin{aligned}
    \vertiii{w}^2& =\int_{\Omega\cup\mathscr{B}\Omega}\int_{\Omega\cup\mathscr{B}\Omega}\lvert w(\mathbf{y})-w(\mathbf{x})\rvert^2\gamma(\mathbf{x},\mathbf{y})d\mathbf{y}d\mathbf{x}+\lVert w \rVert_{L^2(\Omega\cup\mathscr{B}\Omega)}^2\\[2mm]
    & = D(w,w) + \lVert w \rVert_{L^2(\Omega\cup\mathscr{B}\Omega)}^2.
\end{aligned}
    \label{NL_energy_norm1}
\end{equation}
We also introduce the constrained energy space

\begin{equation}
\begin{aligned}
    \mathcal{V}_0(\Omega\cup\mathscr{B}\Omega)\coloneqq\left\{w\in \mathcal{V}(\Omega\cup\mathscr{B}\Omega):\left.w\right\vert_{\mathscr{B}\Omega}=0\right\},
\end{aligned}
    \label{func_spaces2}
\end{equation}
for which 

\begin{equation}\label{eq:V0norm}
\vertiii{w}_0^2 = D(w,w),
\end{equation}
defines a norm. Finally, we define the nonlocal trace space as $\mathcal{V}_t(\Omega\cup\mathscr{B}\Omega)=\left\{\left.w\right\vert_{\mathscr{B}\Omega}:w\in\mathcal{V}(\Omega\cup\mathscr{B}\Omega)\right\}$. Let $\mathcal{V}':\mathcal{V}_0\rightarrow\mathbb{R}$ denote the dual space of bounded linear functionals on $\mathcal{V}_0$ via $L^2$ duality pairing, i.e. the space of functionals $\varphi:\mathcal{V}_0\times\mathcal{W}\rightarrow\mathbb{R}$ of the type  

\begin{equation}
\begin{aligned}
    \varphi(\cdot,\cdot)=\int_{\Omega}(\cdot)(\cdot)d\mathbf{x},
\end{aligned}
    \label{func_spaces_att1}
\end{equation}
where $\mathcal{W}:\Omega\rightarrow\mathbb{R}$. Thus, $\forall w\in\mathcal{W}(\Omega)$, we can write $\varphi(\cdot,w):\mathcal{V}_0\rightarrow\mathbb{R}\in\mathcal{V}'$ as

\begin{equation}
\begin{aligned}
    \varphi(\cdot,w)=\int_{\Omega}(\cdot)w(\mathbf{x})d\mathbf{x}.
\end{aligned}
    \label{func_spaces_att2}
\end{equation}
By comparing Eqs.~(\ref{weak_form_s7}) and (\ref{func_spaces_att2}) we see that $G(\cdot)=\varphi(\cdot,w)$, $\forall w\in\mathcal{W}(\Omega)$.
Then, the weak form of (\ref{strong form_NLD.}) is defined as follows: given $g(\mathbf{x})\in\mathcal{V}_t(\Omega\cup\mathscr{B}\Omega)$, and $b(\mathbf{x})\in\mathcal{W}(\Omega)$, find $u(\mathbf{x})\in\mathcal{V}(\Omega\cup\mathscr{B}\Omega)$ such that $\forall v(\mathbf{x})\in\mathcal{V}_0(\Omega\cup\mathscr{B}\Omega)$

\begin{equation}
\begin{aligned}
    D(u,v)=G(v),
\end{aligned}
    \label{weak_form_s9}
\end{equation}
subject to $u(\mathbf{x})=g(\mathbf{x})$ for $\mathbf{x}\in\mathscr{B}\Omega$. Discussions on the well-posedness of (\ref{weak_form_s9}) can be found in \cite{Du2012_NonlocalVolumes,DElia2020_Acta}.

\section{Quadrature weights using generalized moving least squares}\label{sec:GMLS_construction}

In this section we review the quadrature approach based on the generalized moving least squares (GMLS) \cite{Mirzaei2012,Salehi2013,Mirzaei2013} method, proposed in \cite{trask2019asymptotically}. For given positions of quadrature points, this method determines their associated quadrature weights by solving an equality constrained optimization problem. In \cite{trask2019asymptotically}, the GMLS-based quadrature was employed within the framework of collocation-based meshfree discretizations of strong-form nonlocal problems.

Consider a collection of points $\mathbf{X}_p=\{\mathbf{x}_j\}_{j=1,...,N_p}\subset\mathscr{H}(\mathbf{x},\delta)$, with $N_p\in\mathbb{N}$, and a quadrature rule for functions $f(\mathbf{x},\mathbf{y})\in\mathbf{V}$, given by
\begin{equation}\label{GMLSpurposeFunc1}
    \int_{\mathscr{H}(\mathbf{x},\delta)}f(\mathbf{x},\mathbf{y})d\mathbf{y}\approx\sum_{\substack{j=1\\j:\mathbf{x}_j\neq\mathbf{x}}}^{N_p}f_{j}\omega_{j},
\end{equation}
where $\mathbf{V}$ denotes a Banach space, $f_{j}=f(\mathbf{x},\mathbf{x}_j)$, and $\{{\omega}_j\}_{j=1,...,N_p}\in\mathbb{R}^{N_p}$ is a collection of quadrature weights to be determined. Notice that in Eq.~(\ref{GMLSpurposeFunc1}) we are excluding $\mathbf{x}_j=\mathbf{x}$ to account for the possibility of $f(\mathbf{x},\mathbf{y})$ having a singularity when $\mathbf{x}=\mathbf{x}_j=\mathbf{y}$. If the function does not exhibit such singularity, then $\mathbf{x}_j=\mathbf{x}$ could also be included in the summation. In order to find the quadrature weights we define the following equality constrained optimization problem: find

\begin{equation}
    \argmin_{\{\omega_j\}\in\mathbb{R}^{N_p}}\sum_{\substack{j=1\\j:\mathbf{x}_j\neq\mathbf{x}}}^{N_p}\omega_j^2
    \label{GMLS_opt1}
\end{equation}
\begin{equation*}
    \text{subject to }\sum_{\substack{j=1\\j:\mathbf{x}_j\neq\mathbf{x}}}^{N_p}f_{j}\omega_{j}=\int_{\mathscr{H}(\mathbf{x},\delta)}f(\mathbf{x},\mathbf{y})d\mathbf{y}\;\;\forall f\in\mathbf{V}_h\subset\mathbf{V},
    \label{GMLS_opt2}
\end{equation*}
where $\mathbf{V}_h$ is a finite dimensional subspace of $\bf V$, consisting of functions to be integrated exactly. The problem in (\ref{GMLS_opt1}) leads to the following saddle-point problem:

\begin{equation}
\setlength{\arraycolsep}{2.2pt}
\renewcommand\arraystretch{0.6}
 \begin{bmatrix} 
    \mathbf{I} & \mathbf{B}^T \\
    \mathbf{B} & \mathbf{0}
    \end{bmatrix}
    \begin{bmatrix}
    \boldsymbol{\omega} \\ \boldsymbol{\lambda}
    \end{bmatrix}
    =
    \begin{bmatrix}
    \mathbf{0}\\\mathbf{g}
    \end{bmatrix},
    \label{saddle-point}
\end{equation}
where $\mathbf{I}\in\mathbb{R}^{N_p\times N_p}$ is the identity matrix; $\boldsymbol{\omega}=\{{\omega}_j\}_{j=1,...,N_p}\in\mathbb{R}^{N_p}$ is the vector containing the set of quadrature weights; and $\boldsymbol{\lambda}\in\mathbb{R}^{ \operatorname{dim}(\mathbf{V}_h)}$ is the vector of Lagrange multipliers enforcing the constraint. $\mathbf{B}\in\mathbb{R}^{N_p\times \operatorname{dim}(\mathbf{V}_h)}$ is defined by $B_{aj}=f^{\alpha}(\mathbf{x},\mathbf{x}_j),\,\forall f^{\alpha}\in\mathbf{V}_h$, where $\{f^{\alpha}\}_{\alpha=1,...,\operatorname{dim}(\mathbf{V}_h)}$ is a basis of $\mathbf{V}_h$. The vector $\mathbf{g}\in\mathbb{R}^{\operatorname{dim}(\mathbf{V}_h)}$ contains the exact integrals of each function in $\{f^{\alpha}\}_{\alpha=1,...,\operatorname{dim}(\mathbf{V}_h)}$, i.e., $g_{\alpha}=\int_{\mathscr{H}(\mathbf{x},\delta)}f^\alpha(\mathbf{x},\mathbf{y})d\mathbf{y}$. Based on Eq.~(\ref{saddle-point}), the quadrature weights can be obtained as

\begin{equation}
    \boldsymbol{\omega}=\mathbf{B}^\mathrm{T}\mathbf{S}^{-1}\mathbf{g},
    \label{gmls_final_equation}
\end{equation}
where $\mathbf{S}=\mathbf{BB}^\mathrm{T}$. It has to be noted that the set of integration weights for a given constraint might not be unique \cite{Pasetto2019,Leng2021_AsymptoticallyCR} and that redundant (linearly dependent) conditions might be present in the constraints. This results in the singularity of the matrix $\mathbf{S}$. In this work, as in \cite{trask2019asymptotically}, a pseudoinverse is used to compute $\mathbf{S}^{-1}$, whenever $\mathbf{S}$ is singular. It should also be noted that, as discussed in \cite{Pasetto2019,Leng2021_AsymptoticallyCR}, this set of integration weights can be constructed equivalently by using the reproducing kernel particle method (RKPM) \cite{MP:progress20years}, due to the equivalence of RKPM and GMLS. 

\section{Discrete variational form}\label{sec:discrete-form}

In this section we introduce the discrete form of the variational problem in Eq.~(\ref{weak_form_s9}); specifically, for piecewise linear, finite element discretizations, we describe the computational domain, the discrete representation of the unknown field $u(\mathbf{x})$ and the trial function $v(\mathbf{x})$, and the quadrature rules utilized for the numerical evaluation of the integrals. 

\subsection{Finite element discretization of the weak formulation}\label{FEM_discretization}

Let $\mathcal{M}^h_\Omega\coloneqq\{\Omega^h_e\}_{e=1,\ldots,n_{el,\Omega}}$, $n_{el,\Omega}\in\mathbb{N}$, be a collection of non-overlapping elements, which are open, simply connected subsets of $\mathbb{R}^d$, and let $\partial\Omega^h_e$ be their corresponding boundary. Therefore, $\Omega^h_i\cap\Omega^h_j=\varnothing$ and { $\overline{\Omega^h_i}\cap\overline{\Omega^h_j}=\partial\Omega^h_i\cap\partial\Omega^h_j$} with $i\neq j$, $i,j=1,\ldots,n_{el}${, where the bar denotes the closure of the set}.
We assume that the domain $\Omega$, introduced in Section~\ref{sec:Nonlocal_diff_model}, is a polyhedral so that it can be exactly covered by the mesh $\mathcal M_\Omega^h$, i.e. $\Omega =\cup_{e=1}^{n_{el,\Omega}}\Omega^h_e$. Note that when $\Omega$ is not polyhedral, one can introduce a polyhedral approximation $\Omega^h\approx\Omega$ for which a covering exists. When the nonlocal interaction region is a Euclidean ball, the interaction domain $\mathcal B\Omega$ is generally not a polyhedral domain since vertices of $\Omega$ create rounded corners in $\mathscr{B}\Omega$ (see, for example, Figure \ref{Interaction_domain}). Therefore, following \cite{delia2020cookbook}, we approximate $\mathscr{B}\Omega$ by a polyhedral domain by replacing rounded corners by vertices (see Figure \ref{Interaction_domain_poly} for illustration). From now on, we will refer to this approximate polyhedral domain also as $\mathscr{B}\Omega$. Note that there is no need to extend the boundary data $g(\mathbf{x})$ to added regions between the original curved corners and the new corners of the polyhedral approximation since these portions of the domain are never accessed during the numerical solution process.

\begin{figure}[H]
\begin{center}
\scalebox{0.75}{\includegraphics[trim = 0mm 225mm 0mm 145mm, clip=true,width=1\textwidth]{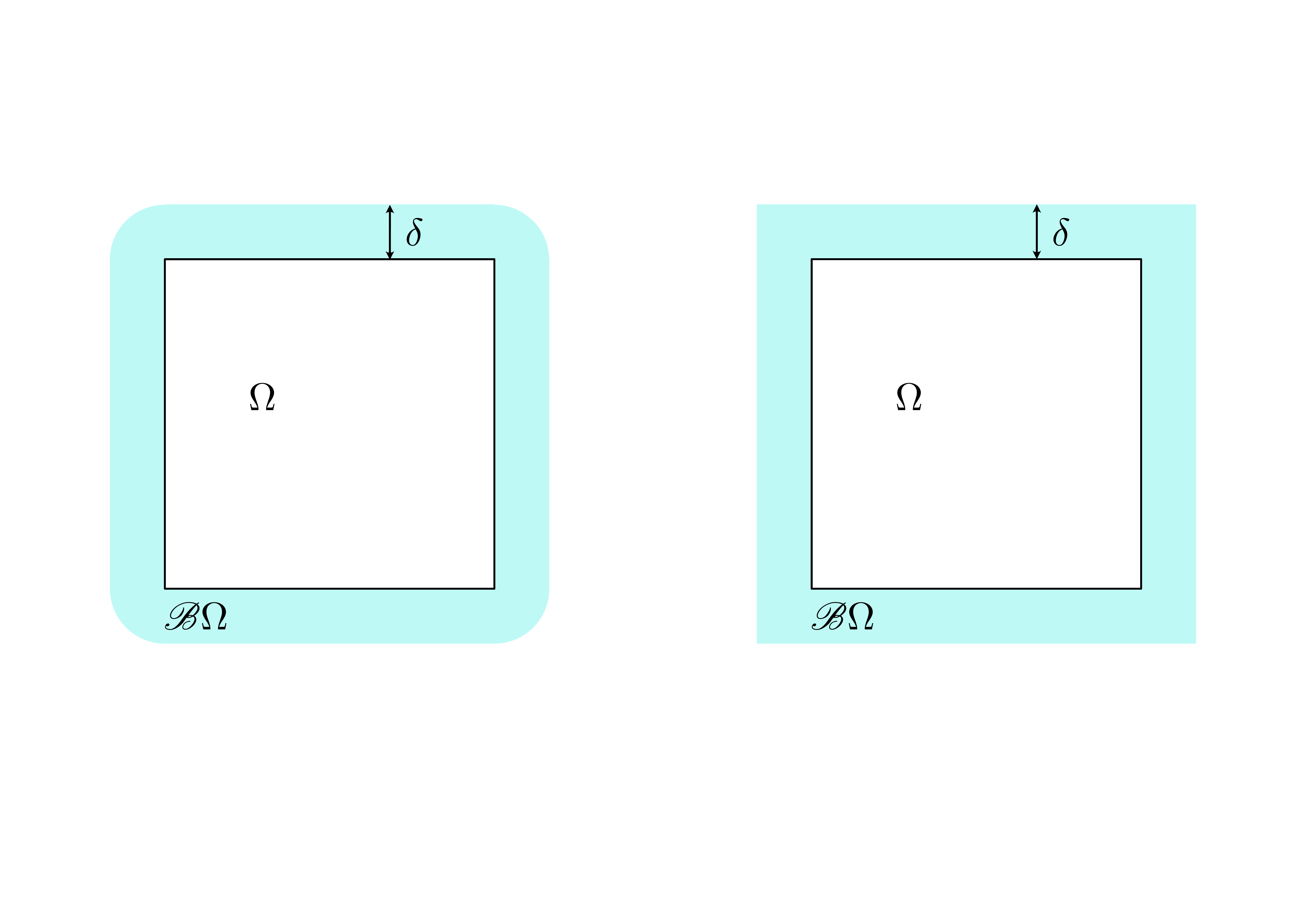}} 
\caption{Left: a square domain $\Omega$ (in white) with its corresponding interaction domain $\mathscr{B}\Omega$ (in light-blue). Right: the same square domain and a polygonal approximate interaction domain, still referred to as $\mathscr{B}\Omega$.}
\label{Interaction_domain_poly}
\end{center}
\end{figure}

Since we consider a polyhedral $\mathscr{B}\Omega$, we can construct another exact mesh $\mathcal{M}^h_{\mathscr{B}\Omega}\coloneqq\{\Omega^h_e\}_{e=n_{el,\Omega}+1,\ldots,n_{el}}$ containing $n_{el}-n_{el,\Omega}$ elements, i.e., $\mathscr{B}\Omega^h\coloneqq\cup_{e=n_{el,\Omega}+1}^{n_{el}-n_{el,\Omega}}\Omega^h_e=\mathscr{B}\Omega$, with $n_{el}\in\mathbb{N}$. Meshing $\Omega$ and $\mathscr{B}\Omega$ separately guarantees that elements do not straddle the shared boundary between $\Omega$ and $\mathscr{B}\Omega$, i.e., $\partial\Omega=\overline{\Omega}\cap\mathscr{B}\Omega$. Moreover, we require that the vertices of the elements of $\mathcal{M}^h_{\mathscr{B}\Omega}$ and $\mathcal{M}^h_{\Omega}$ coincide along the boundary $\partial\Omega$, so that $\mathcal{M}^h=\mathcal{M}^h_{\Omega}\cup\mathcal{M}^h_{\mathscr{B}\Omega}=\{\Omega^h_e\}_{e=1}^{n_{el}}$ is a regular mesh for $\Omega\cup\mathscr{B}\Omega$.\newline

\smallskip
We consider continuous finite element spaces with Lagrange-type compactly supported linear polynomial bases defined with respect to the nodes of $\mathcal{M}^h$.
With $J\in\mathbb{N}$ and $J_\Omega\in\mathbb{N}$, let $\{\tilde{\mathbf{x}}_j\}_{j=1}^{J}$ be the set of all the nodes in $\mathcal{M}^h$, with $\{\tilde{\mathbf{x}}_j\}_{j=1}^{J_\Omega}$ and $\{\tilde{\mathbf{x}}_j\}_{j=J_\Omega+1}^{J}$ being the subset of nodes located in the open domain $\Omega$ and in the closed domain $\mathscr{B}\Omega$, respectively. Notice that in this way, the nodes located on $\partial\Omega=\overline{\Omega}\cap\mathscr{B}\Omega$ are assigned to $\mathscr{B}\Omega$. Then, for $j=1,\ldots,J$, let $\psi_j(\mathbf{x})$ denote a piecewise linear polynomial function such that $\psi_j(\tilde{\mathbf{x}}_k)=\delta_{jk}$ for $k=1,\ldots,J$, where $\delta_{jk}$ is the Kronecker delta function. Then, we define the finite element spaces as

\begin{equation}
    \mathcal{V}^h=\operatorname{span}\{\psi_j\}_{j=1}^{J}\subset\mathcal{V}(\Omega\cup\mathscr{B}\Omega),
\end{equation}
and

\begin{equation}\label{Vh0}
    \mathcal{V}^h_0=\operatorname{span}\{\psi_j\}_{j=1}^{J_\Omega}\subset\mathcal{V}_0(\Omega\cup\mathscr{B}\Omega),
\end{equation}
of dimensions $J$ and $J_\Omega$, respectively. Note that all functions belonging to $\mathcal{V}^h$ and $\mathcal{V}^h_0$ are continuous by construction.

The finite element approximation $u^h(\mathbf{x})\in\mathcal{V}^h$ of the solution $u(\mathbf{x})$ of the nonlocal problem is defined as

\begin{equation}
    u^h(\mathbf{x})=\sum_{j=1}^{J}\psi_j(\mathbf{x})u_j=\sum_{j=1}^{J_\Omega}\psi_j(\mathbf{x})u_j+\sum_{j=J_\Omega+1}^{J}\psi_j(\mathbf{x})g(\tilde{\mathbf{x}}_j)=w^h+g^h,
    \label{uh_fem}
\end{equation}
for a set of coefficients $\{u_j\}_{j=1}^{J}$. Here, the volume constraint in (\ref{strong form_NLD.}) has been applied to a subset associated with the nodes in $\mathscr{B}\Omega$

\begin{equation}
 u_j=g(\tilde{\mathbf{x}}_j)\;\;\;\text{for}\;j=J_\Omega+1,\ldots,J,
\end{equation}
so that

\begin{equation}
 w^h\coloneqq\sum_{j=1}^{J_\Omega}\psi_j(\mathbf{x})u_j\;\;\;\text{and}\;\;\;g^h\coloneqq\sum_{j=J_\Omega+1}^{J}\psi_j(\mathbf{x})g(\tilde{\mathbf{x}}_j).
 \label{uh_fem2}
\end{equation}
The finite element approximation $u^h$ associated with the nonlocal problem in (\ref{weak_form_s9}) is then found by solving the following discrete weak formulation: given $g(\mathbf{x})\in\mathcal{V}_t(\Omega\cup\mathscr{B}\Omega)$, and {$b(\mathbf{x})\in\mathcal{W}(\Omega)$} (see Section~\ref{weak_form}), find $u^h\in\mathcal{V}^h$ such that $\forall v^h\in\mathcal{V}^h_0$ 

\begin{equation}
\begin{aligned}
    D(u^h,v^h)=G(v^h).
\end{aligned}
    \label{weak_form_discrete1}
\end{equation}
By substituting Eq.~(\ref{uh_fem}) in Eq.~(\ref{weak_form_discrete1}) and by choosing $v^h(\mathbf{x})$ from the set of basis functions $\{\psi_i\}_{i=1}^{J_\Omega}$ we get

\begin{equation}
\begin{aligned}
    D(w^h,v^h)=G(v^h)-D(g^h,v^h),
\end{aligned}
    \label{weak_form_discrete2}
\end{equation}
which results in the linear system

\begin{equation}
\begin{aligned}
    \sum_{j=1}^{J_\Omega}D(\psi_j,\psi_i)u_j=G(\psi_i)-D(g^h,\psi_i),
\end{aligned}
    \label{weak_form_discrete3}
\end{equation}
for $i=1,\ldots,J_\Omega$.
Eq.~(\ref{weak_form_discrete3}) can be expressed in matrix form as

\begin{equation}
\begin{aligned}
    \mathbf{A}\mathbf{u}=\mathbf{f},
\end{aligned}
    \label{weak_form_discrete4}
\end{equation}
where $\mathbf{A}$ is a $J_\Omega\times J_\Omega$ matrix with components 

\begin{equation}
\begin{aligned}
    A_{ij}&=D(\psi_j,\psi_i)\\
    &=\int_{\Omega\cup\mathscr{B}\Omega}\int_{\Omega\cup\mathscr{B}\Omega}\left[\psi_i(\mathbf{y})-\psi_i(\mathbf{x})\right]\gamma(\mathbf{x},\mathbf{y})\left[\psi_j(\mathbf{y})-\psi_j(\mathbf{x})\right]d\mathbf{y}d\mathbf{x},
\end{aligned}
    \label{weak_form_discrete5}
\end{equation}
$\mathbf{f}$ is a $J_\Omega\times 1$ vector with components

\begin{equation}
\begin{aligned}
    f_{i}&=G(\psi_i)-D(g^h,\psi_i)=\int_{\Omega}\psi_i(\mathbf{x})b(\mathbf{x})d\mathbf{x}\\
    &\phantom{=}-\int_{\Omega\cup\mathscr{B}\Omega}\int_{\Omega\cup\mathscr{B}\Omega}\left[\psi_i(\mathbf{y})-\psi_i(\mathbf{x})\right]\gamma(\mathbf{x},\mathbf{y})\left[g^h(\mathbf{y})-g^h(\mathbf{x})\right]d\mathbf{y}d\mathbf{x},
\end{aligned}
    \label{weak_form_discrete6}
\end{equation}
and $\mathbf{u}$ is a vector of size $J_\Omega\times 1$ containing the set of unknown coefficients $\{u_j\}_{j=1}^{J_\Omega}$ to be determined. 

\subsection{Discrete quadrature}\label{quadrature_discretization}

We introduce the numerical quadrature used to solve Eq.~(\ref{weak_form_discrete2}). As described in Section \ref{FEM_discretization} we discretize $\Omega\cup\mathscr{B}\Omega$ using the mesh $\mathcal{M}^h$, and $\Omega$ with $M^h_\Omega\subset\mathcal{M}^h$. Therefore, we can express the left-hand side (LHS) and the right-hand side (RHS) of Eq.~(\ref{weak_form_discrete2}) as

\begin{equation}
\begin{aligned}
    &D(w^h,v^h)\\=&\int_{\Omega\cup\mathscr{B}\Omega}\int_{\Omega\cup\mathscr{B}\Omega}\left[v^h(\mathbf{y})-v^h(\mathbf{x})\right]\gamma(\mathbf{x},\mathbf{y})\left[w^h(\mathbf{y})-w^h(\mathbf{x})\right]d\mathbf{y}d\mathbf{x}\\
    =&\sum_{\Omega_e^h\in\mathcal{M}^h}\int_{\Omega_e^h}\int_{\Omega\cup\mathscr{B}\Omega}\left[v^h(\mathbf{y})-v^h(\mathbf{x})\right]\gamma(\mathbf{x},\mathbf{y})\left[w^h(\mathbf{y})-w^h(\mathbf{x})\right]d\mathbf{y}d\mathbf{x}\\
    =&\sum_{\Omega_e^h\in\mathcal{M}^h}\int_{\Omega_e^h}\int_{\left(\Omega\cup\mathscr{B}\Omega\right)\cap\mathscr{H}(\mathbf{x},\delta)}\left[v^h(\mathbf{y})-v^h(\mathbf{x})\right]\gamma(\mathbf{x},\mathbf{y})\left[w^h(\mathbf{y})-w^h(\mathbf{x})\right]d\mathbf{y}d\mathbf{x},
\end{aligned}
    \label{weak_form_discretequad1}
\end{equation}
and

\begin{equation}
\begin{aligned}
    &G(v^h)-D(g^h,v^h)\\=&\int_{\Omega}v^h(\mathbf{x})b(\mathbf{x})d\mathbf{x}\\
    &-\int_{\Omega\cup\mathscr{B}\Omega}\int_{\Omega\cup\mathscr{B}\Omega}\left[v^h(\mathbf{y})-v^h(\mathbf{x})\right]\gamma(\mathbf{x},\mathbf{y})\left[g^h(\mathbf{y})-g^h(\mathbf{x})\right]d\mathbf{y}d\mathbf{x}\\
    =&\sum_{\Omega_e^h\in\mathcal{M}^h_\Omega}\int_{\Omega_e^h}v^h(\mathbf{x})b(\mathbf{x})d\mathbf{x}\\
    &-\sum_{\Omega_e^h\in\mathcal{M}^h}\int_{\Omega_e^h}\int_{\left(\Omega\cup\mathscr{B}\Omega\right)\cap\mathscr{H}(\mathbf{x},\delta)}\left[v^h(\mathbf{y})-v^h(\mathbf{x})\right]\gamma(\mathbf{x},\mathbf{y})\left[g^h(\mathbf{y})-g^h(\mathbf{x})\right]d\mathbf{y}d\mathbf{x},
\end{aligned}
    \label{weak_form_discretequad2}
\end{equation}
where we have restricted the inner integral to $\left(\Omega\cup\mathscr{B}\Omega\right)\cap\mathscr{H}(\mathbf{x},\delta)$ (see Eq.~(\ref{ball})). 

We now describe how to discretize the integrals over the mesh elements $\Omega_e^h$ and over $\left(\Omega\cup\mathscr{B}\Omega\right)\cap\mathscr{H}(\mathbf{x},\delta)$. 
For the outer integral (over the elements) we consider a high-order Gauss quadrature; for further discussion on outer quadrature schemes we refer the reader to \cite{delia2020cookbook}. For $N_q\in\mathbb{N}$, we denote the set of element Gauss quadrature points and weights to be used for the element integrals present in $D(w^h,v^h)$ as $\{\mathbf{x}^e_q\}_{q=1}^{N_q}\in\Omega^h_e$ and $\{\omega^e_q\}_{q=1}^{N_q}$, respectively. We also define, for $N_{b}\in\mathbb{N}$, the set of element Gauss quadrature points and weights to be employed for the integration over the elements in $G(v^h)$ as $\{\mathbf{x}^e_b\}_{b=1}^{N_b}\in\Omega^h_e$ and $\{\omega^e_b\}_{b=1}^{N_b}$, respectively. Therefore, from Eqs.~(\ref{weak_form_discretequad1}) and (\ref{weak_form_discretequad2}), we get

\begin{equation}
\begin{aligned}
    & D(w^h,v^h)\\=&\sum_{\Omega_e^h\in\mathcal{M}^h}\int_{\Omega_e^h}\int_{\left(\Omega\cup\mathscr{B}\Omega\right)\cap\mathscr{H}(\mathbf{x},\delta)}\left[v^h(\mathbf{y})-v^h(\mathbf{x})\right]\gamma(\mathbf{x},\mathbf{y})\\&\left[w^h(\mathbf{y})-w^h(\mathbf{x})\right]d\mathbf{y}d\mathbf{x}\\
    \approx &\sum_{\Omega_e^h\in\mathcal{M}^h}\sum_{q=1}^{N_q}\int_{\left(\Omega\cup\mathscr{B}\Omega\right)\cap\mathscr{H}(\mathbf{x}^e_q,\delta)}\left[v^h(\mathbf{y})-v^h(\mathbf{x}^e_q)\right]\gamma(\mathbf{x}^e_q,\mathbf{y})\\&\left[w^h(\mathbf{y})-w^h(\mathbf{x}^e_q)\right]d\mathbf{y}\omega^e_q,
\end{aligned}
    \label{weak_form_discretequad3}
\end{equation}
and

\begin{equation}
\begin{aligned}
    & G(v^h)-D(g^h,v^h)\\=&\sum_{\Omega_e^h\in\mathcal{M}^h_\Omega}\int_{\Omega_e^h}v^h(\mathbf{x})b(\mathbf{x})d\mathbf{x}\\
    &-\sum_{\Omega_e^h\in\mathcal{M}^h}\int_{\Omega_e^h}\int_{\left(\Omega\cup\mathscr{B}\Omega\right)\cap\mathscr{H}(\mathbf{x},\delta)}\left[v^h(\mathbf{y})-v^h(\mathbf{x})\right]\gamma(\mathbf{x},\mathbf{y})\left[g^h(\mathbf{y})-g^h(\mathbf{x})\right]d\mathbf{y}d\mathbf{x}\\
    \approx&\sum_{\Omega_e^h\in\mathcal{M}^h_\Omega}\sum_{b=1}^{N_b}v^h(\mathbf{x}^e_b)b(\mathbf{x}^e_b)\omega^e_b\\
    &-\sum_{\Omega_e^h\in\mathcal{M}^h}\sum_{q=1}^{N_q}\int_{\left(\Omega\cup\mathscr{B}\Omega\right)\cap\mathscr{H}(\mathbf{x}^e_q,\delta)}\left[v^h(\mathbf{y})-v^h(\mathbf{x}^e_q)\right]\gamma(\mathbf{x}^e_q,\mathbf{y})\left[g^h(\mathbf{y})-g^h(\mathbf{x}^e_q)\right]d\mathbf{y}\omega^e_q.
\end{aligned}
    \label{weak_form_discretequad4}
\end{equation}

To discretize the remaining inner integrals over $\left(\Omega\cup\mathscr{B}\Omega\right)\cap\mathscr{H}(\mathbf{x}^e_q,\delta)$ in Eqs.~(\ref{weak_form_discretequad3}) and (\ref{weak_form_discretequad4}) we use the GMLS quadrature introduced in Section \ref{sec:GMLS_construction}. We start with the case in which  $\left(\Omega\cup\mathscr{B}\Omega\right)\cap\mathscr{H}(\mathbf{x}^e_q,\delta)=\mathscr{H}(\mathbf{x}^e_q,\delta)$, i.e., the integration domain is the full ball of radius $\delta$ around $\mathbf{x}^e_q$. Note that this is the case for all $\mathbf{x}^e_q\in\left(\Omega\cup\partial\Omega\right)$. We then consider the following set of points placed in a regular uniform grid, symmetric around $\mathbf{x}^e_q$:

\begin{equation}
\begin{aligned}\label{quad_points_ball_grid}
\left\{{\mathbf{x}}^e_{qp}\right\}_{p=1}^{\overline{N}_{qp}}\coloneqq\bigg\{&{\mathbf{x}}^e_{qp}\in\mathbb{R}^d;k_1,k_2,\ldots,k_d\in\mathbb{Z}\setminus\left\{0\right\}:
\\&{\mathbf{x}}^e_{qp}=\left({x}^e_{qp1},{x}^e_{qp2},\ldots,{x}^e_{qpd}\right)\\
&=\left(x^e_{q1}+(2k_1-\sgn(k_1))\frac{\overline{h}}{2},x^e_{q2}+(2k_2-\sgn(k_2))\frac{\overline{h}}{2},\ldots,\right.\\& \left. x^e_{qd}+(2k_d-\sgn(k_d))\frac{\overline{h}}{2}\right),\\
&-\overline{N}_{qp,\delta}\leq k_1,k_2,\ldots,k_d\leq \overline{N}_{qp,\delta}
\bigg\},
\end{aligned}
\end{equation}
where $\overline{N}_{qp,\delta}\in\mathbb{N}$, 

\begin{equation}
\overline{h}=\frac{\delta}{\overline{N}_{qp,\delta}}
\end{equation}
is the spacing between grid points, and

\begin{equation}
\overline{N}_{qp}=\left(2\overline{N}_{qp,\delta}\right)^d,
\end{equation}
is the overall number of points. In this work, we take $\overline{N}_{qp,\delta}$ to be a constant independent of $q$, i.e., $\overline{N}_{{q_i}p,\delta}=\overline{N}_{{q_j}p,\delta}$ $\forall q_i,q_j$ such that $\mathbf{x}^e_{q_i},\mathbf{x}^e_{q_j}\in\Omega\cup\mathscr{B}\Omega$. 
The subset of $N_{qp}$ points of $\left\{{\mathbf{x}}^e_{qp}\right\}_{p=1}^{\overline{N}_{qp}}$ contained in $\mathscr{H}(\mathbf{x}^e_q,\delta)\cap\left(\Omega\cup\mathscr{B}\Omega\right)$ is given by  

\begin{equation}\label{Inner_quadrature_set}
\begin{aligned}
\left\{{\mathbf{x}}^e_{qp}\right\}_{p=1}^{{N}_{qp}}&\coloneqq\left\{{\mathbf{x}}^e_{qp}\right\}_{p=1}^{\overline{N}_{qp}}\cap\mathscr{H}(\mathbf{x}^e_q,\delta)\cap\left(\Omega\cup\mathscr{B}\Omega\right)\\
&=\left \{\mathbf{x}^e_{qp}\in \left\{{\mathbf{x}}^e_{qp}\right\}_{p=1}^{\overline{N}_{qp}}\cap\left(\Omega\cup\mathscr{B}\Omega\right):
\lvert \mathbf{x}^e_{qp}-\mathbf{x}^e_q \rvert_{\ell^{\tilde{p}}} \leq \delta
\right \}.    
\end{aligned}
\end{equation}
This is the set of quadrature points used to discretize the integrals over $\left(\Omega\cup\mathscr{B}\Omega\right)\cap\mathscr{H}(\mathbf{x}^e_q,\delta)$ in Eqs.~(\ref{weak_form_discretequad3}) and (\ref{weak_form_discretequad4}). When $\left(\Omega\cup\mathscr{B}\Omega\right)\cap\mathscr{H}(\mathbf{x}^e_q,\delta)=\mathscr{H}(\mathbf{x}^e_q,\delta)$, this set reduces to 

\begin{equation}\label{Inner_quadrature_set2}
\begin{aligned}
\left\{{\mathbf{x}}^e_{qp}\right\}_{p=1}^{{N}_{qp}}=\left \{\mathbf{x}^e_{qp}\in \left\{{\mathbf{x}}^e_{qp}\right\}_{p=1}^{\overline{N}_{qp}}:
\lvert \mathbf{x}^e_{qp}-\mathbf{x}^e_q \rvert_{\ell^{\tilde{p}}} \leq \delta
\right \}.
\end{aligned}
\end{equation}
Figures \ref{1D_ball_quadrature} and \ref{2D_ball_quadrature} show the distribution of quadrature points for one-dimensional and two-dimensional Euclidean balls.

\begin{figure}[H]
\begin{center}
\scalebox{1.0}{\includegraphics[trim = 50mm 430mm 50mm 320mm, clip=true,width=1\textwidth]{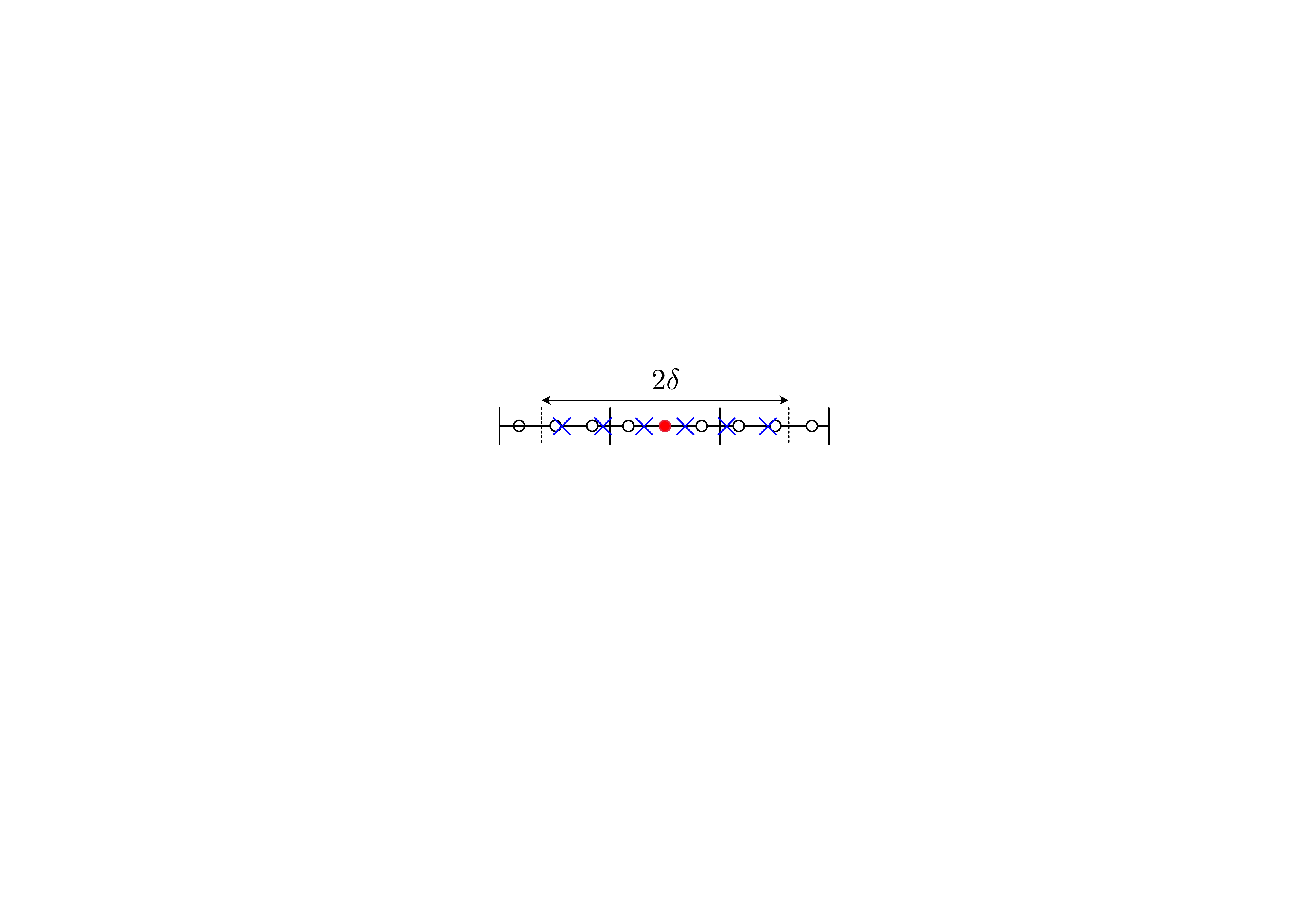}} 
\caption{One-dimensional Euclidean ball quadrature points. The filled red dot represents $x^e_q$ while the blue crosses are the associated quadrature points $x^e_{qp}$.}
\label{1D_ball_quadrature}
\end{center}
\end{figure}

\begin{figure}[H]
\begin{center}
\scalebox{1.0}{\includegraphics[trim = 50mm 300mm 50mm 175mm, clip=true,width=1\textwidth]{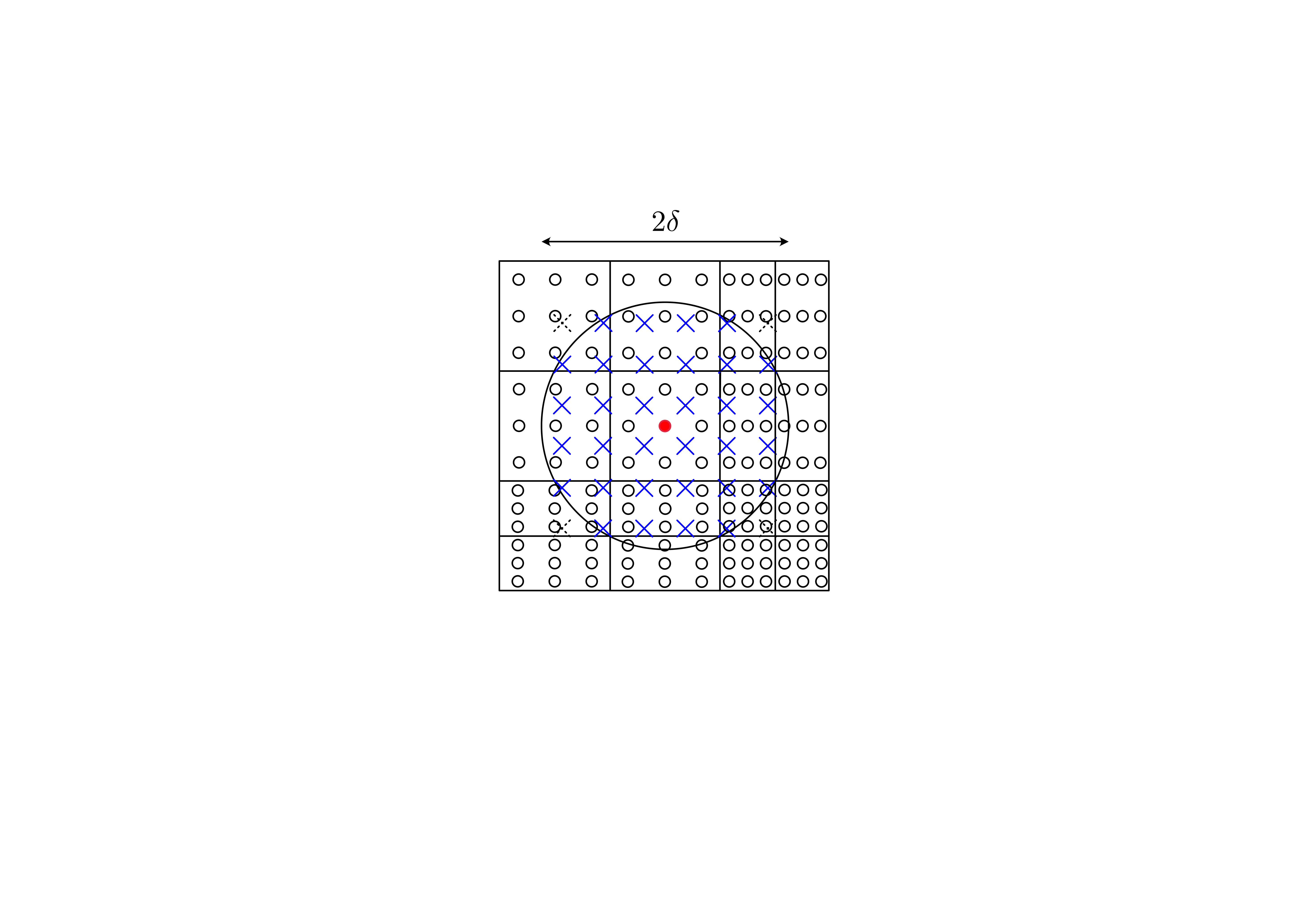}} 
\caption{Two-dimensional Euclidean ball quadrature points. The filled red dot represents $\mathbf{x}^e_q$ while the blue crosses are the associated quadrature points $\mathbf{x}^e_{qp}$.}
\label{2D_ball_quadrature}
\end{center}
\end{figure}

To determine the set of quadrature weights $\left\{\omega^e_{qp}\right\}_{p=1}^{N_{qp}}$ associated with $\left\{{\mathbf{x}}^e_{qp}\right\}_{p=1}^{{N}_{qp}}$, we employ the approach presented in Section \ref{sec:GMLS_construction}, with $\mathbf{x}=\mathbf{x}^e_q$. As our finite dimensional space $\mathbf{V}_h$, i.e., as the space of functions for which we impose exactness of integration, we select

\begin{equation}\label{Vh_exact_constraints}
\begin{aligned}
\mathbf{V}_h=\{&f(\mathbf{x},\mathbf{y}):\Omega\cup\mathscr{B}\Omega\times\Omega\cup\mathscr{B}\Omega\rightarrow\mathbb{R}, \\
&f(\mathbf{x},\mathbf{y})=\gamma(\mathbf{x},\mathbf{y})\left(\mathbf{y}-\mathbf{x}\right)^{\boldsymbol{\beta}} \text{ with }\lvert\boldsymbol{\beta}\rvert=2\},
\end{aligned}
\end{equation}
where we are using multi-index notation. Here, $\boldsymbol{\beta}$ is a collection of $d$ non-negative integers, $\boldsymbol{\beta}=(\beta_1,\ldots,\beta_d)$ with length $\lvert\boldsymbol{\beta}\rvert=\sum_{i=1}^{d}\beta_i$. For a given $\boldsymbol{\beta}$, $(\mathbf{y}-\mathbf{x})^{\boldsymbol{\beta}}=(y_1-x_1)^{\beta_1}\ldots(y_d-x_d)^{\beta_d}$. Eq.~(\ref{Vh_exact_constraints}) can be related to assuming the trial and test functions $v(\mathbf{x})$ and $u(\mathbf{x})$ to be linear functions in Eqs.~(\ref{weak_form_discretequad3}) and (\ref{weak_form_discretequad4}), consistently with our choice to approximate them with linear finite element approximations (see Section \ref{FEM_discretization}). In fact, in a one-dimensional case, Eq.~(\ref{Vh_exact_constraints}) corresponds to imposing exact integration of $\int_{\left(\Omega\cup\mathscr{B}\Omega\right)\cap\mathscr{H}(x,\delta)}(y-x)\gamma(x,y)(y-x)dy$, while in a two-dimensional case, to imposing exact integration of $\int_{\left(\Omega\cup\mathscr{B}\Omega\right)\cap\mathscr{H}(\mathbf{x},\delta)}(y_1-x_1)\gamma(\mathbf{x},\mathbf{y})(y_1-x_1)d\mathbf{y}$, $\int_{\left(\Omega\cup\mathscr{B}\Omega\right)\cap\mathscr{H}(\mathbf{x},\delta)}(y_2-x_2)\gamma(\mathbf{x},\mathbf{y})(y_2-x_2)d\mathbf{y}$, and $\int_{\left(\Omega\cup\mathscr{B}\Omega\right)\cap\mathscr{H}(\mathbf{x},\delta)}(y_1-x_1)\gamma(\mathbf{x},\mathbf{y})(y_2-x_2)d\mathbf{y}$. Furthermore, for kernels $\gamma(\mathbf{x},\mathbf{y})$ of the types expressed in Eqs.~(\ref{kernel_form1}) and (\ref{kernel_form2}), the functions in $\mathbf{V}_h$ only depend on $\mathbf{y}-\mathbf{x}$, meaning that the quadrature weights depend only on the relative position between the quadrature points in $\left\{{\mathbf{x}}^e_{qp}\right\}_{p=1}^{{N}_{qp}}$ and the center of the ball $\mathbf{x}^e_{q}$, i.e., $\mathbf{x}^e_{qp}-\mathbf{x}^e_{q}$, for arbitrary $p,q$. Since the positions of the points $\mathbf{x}^e_{qp}$ are defined relative to $\mathbf{x}^e_{q}$ (see Eq. \ref{quad_points_ball_grid}), their relative positions with respect to the centers of the balls is always the same for all full balls. Therefore, the quadrature weights can be evaluated once for a representative full ball and used for all full balls $\mathscr{H}(\mathbf{x}^e_q,\delta)$, $\forall\mathbf{x}^e_q\in\Omega\cup\partial\Omega$.

Note that in \cite{Leng2021_AsymptoticallyCR} a similar placement of quadrature points within the full ball $\mathscr{H}(\mathbf{x}^e_q,\delta)$, i.e., quadrature points in a regular uniform grid,  symmetrically distributed around $\mathbf{x}^e_q$, was employed for the numerical quadrature of strong-form nonlocal diffusion. Furthermore, conditions for obtaining positive quadrature weights, as well as expressions for them, are also provided in \cite{Leng2021_AsymptoticallyCR}. While in this work we do not explicitly impose any restriction on the positivity of the weights, in all our tests the quadrature weights $\left\{\omega^e_{qp}\right\}_{p=1}^{N_{qp}}$ are verified to be positive.

Next, we consider the case in which $\mathbf{x}^e_q\in\mathscr{B}\Omega\setminus\partial\Omega$. In this case, $\left(\Omega\cup\mathscr{B}\Omega\right)\cap\mathscr{H}(\mathbf{x}^e_q,\delta)\subset\mathscr{H}(\mathbf{x}^e_q,\delta)$, i.e. the integration over $\left(\Omega\cup\mathscr{B}\Omega\right)\cap\mathscr{H}(\mathbf{x}^e_q,\delta)$ is over a \textit{partial} or \textit{truncated} ball (see Figure  \ref{2D_ball_quadrature_boundary} for an example of a two-dimensional partial Euclidean ball). Therefore, the set of points defined by Eq. (\ref{Inner_quadrature_set}) will not be symmetrical with respect to its associated $\mathbf{x}^e_q$. Moreover, its dimension $N_{qp}$ will also be different depending on the position of $\mathbf{x}^e_q$. For this reason, a different set of weights $\left\{\omega^e_{qp}\right\}_{p=1}^{N_{qp}}$ needs to be computed $\forall q$ such that $\mathbf{x}^e_q\in\left(\mathscr{B}\Omega\setminus\partial\Omega\right)$. 
Recall that the expressions in Section \ref{sec:GMLS_construction} were presented for integrals over full balls. For partial balls, the constraints in the optimization problem in Eq.~(\ref{GMLS_opt1}) can be stated as

\begin{equation}
    \sum_{\substack{p=1\\p:\mathbf{x}^e_{qp}\neq\mathbf{x}^e_q}}^{N_{qp}}f_{j}\omega_{j}=\int_{\mathscr{H}(\mathbf{x}^e_q,\delta)\cap\left(\Omega\cup\mathscr{B}\Omega\right)}f(\mathbf{x}^e_q,\mathbf{y})d\mathbf{y}\;\;\forall f\in\mathbf{V}_h.
    \label{GMLS_opt2_partial_balls}
\end{equation}
Due to the complex geometry of $\mathscr{H}(\mathbf{x}^e_q,\delta)\cap\left(\Omega\cup\mathscr{B}\Omega\right)$ , the analytical integral on the right-hand side of Eq.~(\ref{GMLS_opt2_partial_balls}) is particularly cumbersome. Therefore, in this work, we follow \cite{You2020Regression,trask2019asymptotically,Leng2021_AsymptoticallyCR,Gross2020} and approximate the right-hand side of Eq.~(\ref{GMLS_opt2_partial_balls}) with the integral over the full ball, as in Eq. (\ref{GMLS_opt1}). 

\subsubsection{Special treatment of the nonlocal boundary}
Let us now consider $q_1\in e_1$, and $q_2\in e_2$ such that $\mathbf{x}^{e_1}_{q_1}\in\Omega\cup\partial\Omega$, and $\mathbf{x}^{e_2}_{q_2}\in\mathscr{B}\Omega\setminus\partial\Omega$, and two points $\mathbf{x}^{e_1}_{q_1p_i}\in\left\{\mathbf{x}^{e_1}_{q_1p}\right\}_{p=1}^{N_{{q_1}p}}$, and $\mathbf{x}^{e_2}_{q_2p_j}\in\left\{\mathbf{x}^{e_2}_{q_2p}\right\}_{p=1}^{N_{{q_2}p}}$, such that 

\begin{equation}
    \mathbf{x}^{e_1}_{q_1p_i}-\mathbf{x}^{e_1}_{q_1}=\mathbf{x}^{e_2}_{q_2p_j}-\mathbf{x}^{e_2}_{q_2}.
\end{equation}
By employing the optimization-based procedure described in Section~\ref{sec:GMLS_construction} with the finite-dimensional space in (\ref{Vh_exact_constraints}), we can determine the sets of weights $\{\omega^{e_1}_{q_1p}\}_{p=1}^{N_{q_1p}}$ and $\{\omega^{e_2}_{q_2p}\}_{p=1}^{N_{q_2p}}$ associated with $\left\{\mathbf{x}^{e_1}_{q_1p}\right\}_{p=1}^{N_{{q_1}p}}$ and $\mathbf{x}^{e_2}_{q_2p_j}\in\left\{\mathbf{x}^{e_2}_{q_2p}\right\}_{p=1}^{N_{{q_2}p}}$, respectively. In general, $\omega^{e_1}_{{q_1}p_i}\neq\omega^{e_2}_{{q_2}p_j}$, meaning that two quadrature points with the same relative position with respect to the center $\mathbf{x}^e_{q}$ of the corresponding ball will have different weights. As illustrated numerically in Section \ref{numerical_1D_uniform}, this fact may cause the discretization error to increase near the boundary $\mathscr{B}\Omega$. To circumvent this issue, we consider an extension of the interaction domain of size $t_e$, with $0\leq t_e\leq\delta$, for the computation of the inner quadrature weights. To this end, we define 

\begin{equation}\label{interaction domain_extended}
\mathscr{B}\Omega^{t_e}\coloneqq
\left \{\mathbf{y}\in {\mathbb{R} ^{d}}\setminus\Omega:
\exists \mathbf{x}\in\Omega \; \text{such that} \;\lvert \mathbf{y}-\mathbf{x} \rvert_{\ell^{\tilde{p}}} \leq \left(\delta+t_e\right)
\right \}\setminus\mathscr{B}\Omega.    
\end{equation}
As discussed above, for the interaction domain $\mathscr{B}\Omega$, in the case of Euclidean balls, i.e., $\tilde{p}=2$, $\mathscr{B}\Omega^{t_e}$ will have rounded corners, which we replace with vertices to make $\mathscr{B}\Omega^{t_e}$ a polyhedral domain that can be easily meshed (see Figure \ref{Interaction_domain_poly_te} for a two-dimensional illustration). 

\begin{figure}[H]
\begin{center}
\scalebox{0.75}{\includegraphics[trim = 0mm 200mm 0mm 100mm, clip=true,width=1\textwidth]{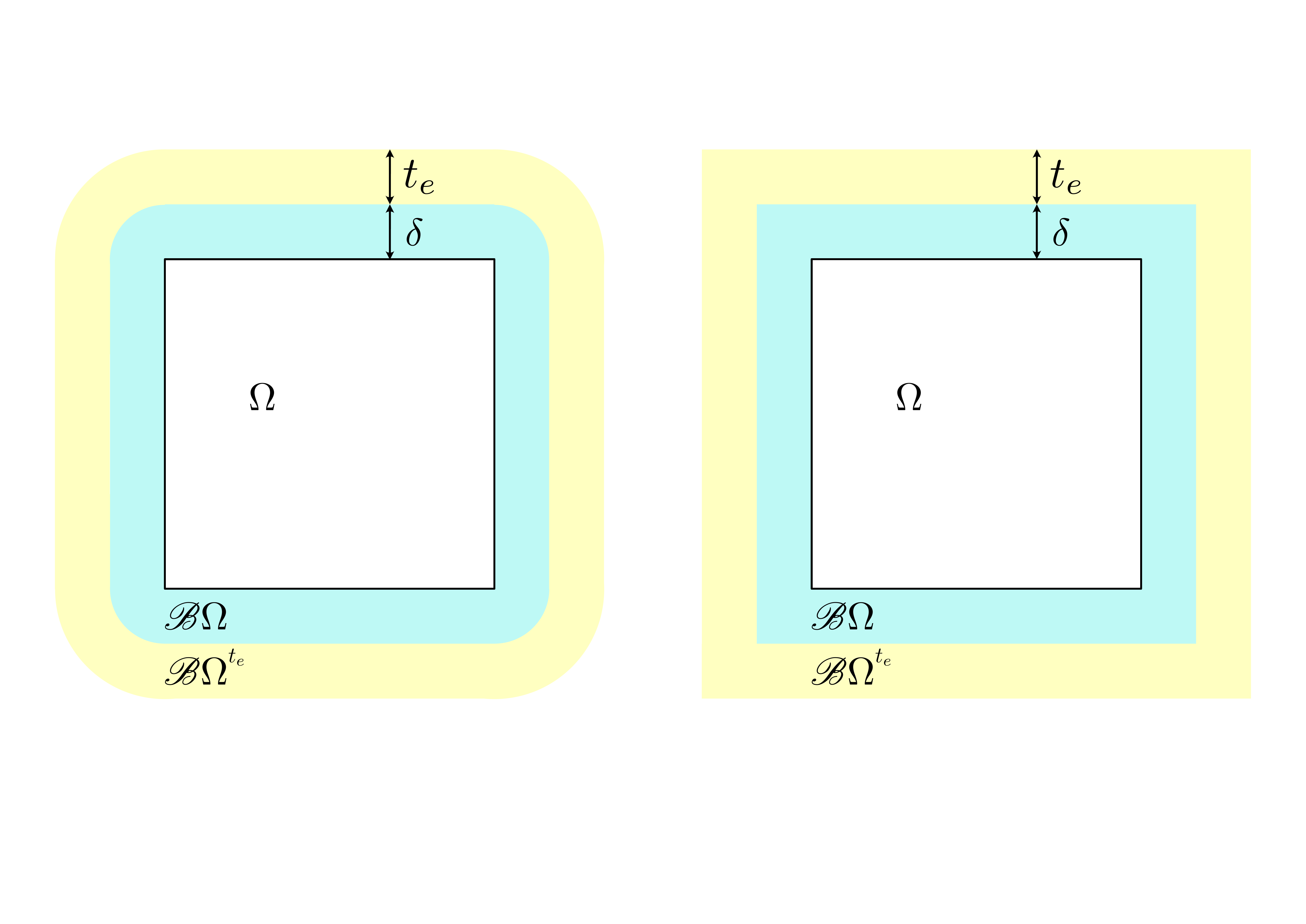}} 
\caption{Left: a square domain $\Omega$ (in white) with its corresponding interaction domain $\mathscr{B}\Omega$ (in light-blue) and its interaction domain extension $\mathscr{B}\Omega^{t_e}$ (in yellow) . Right: the same square domain and polygonal approximate interaction domain and interaction domain extension, still referred to as $\mathscr{B}\Omega$ and $\mathscr{B}\Omega^{t_e}$, respectively.}
\label{Interaction_domain_poly_te}
\end{center}
\end{figure}

Now, $\forall \mathbf{x}^e_q\in\left(\Omega\cup\mathscr{B}\Omega\right)$, we define the following set of points 

\begin{equation}\label{Inner_quadrature_set_extended}
\begin{aligned}
\left\{{\mathbf{x}}^e_{qp}\right\}_{p=1}^{\tilde{N}_{qp}}&\coloneqq\left\{{\mathbf{x}}^e_{qp}\right\}_{p=1}^{\overline{N}_{qp}}\cap\mathscr{H}(\mathbf{x}^e_q,\delta)\cap\left(\Omega\cup\mathscr{B}\Omega\cup\mathscr{B}\Omega^{t_e}\right)\\
&=\left \{\mathbf{x}^e_{qp}\in \left\{{\mathbf{x}}^e_{qp}\right\}_{p=1}^{\overline{N}_{qp}}\cap\left(\Omega\cup\mathscr{B}\Omega\cup\mathscr{B}\Omega^{t_e}\right):
\lvert \mathbf{x}^e_{qp}-\mathbf{x}^e_q \rvert_{\ell^{\tilde{p}}} \leq \delta
\right \},    
\end{aligned}
\end{equation}
which coincides to the one defined in (\ref{Inner_quadrature_set}) for $t_e=0$ and, regardless of $t_e$, $\forall\mathbf{x}^e_q\in(\Omega\cup\partial\Omega).$
For $t_e=\delta$, we have $\left(\Omega\cup\mathscr{B}\Omega\cup\mathscr{B}\Omega^{t_e}\right)\cap\mathscr{H}(\mathbf{x}^e_q,\delta)=\mathscr{H}(\mathbf{x}^e_q,\delta)$, meaning that for every $\mathbf{x}^e_q\in\left(\Omega\cup\mathscr{B}\Omega\right)$ the set of points defined in (\ref{Inner_quadrature_set_extended}) is distributed across each full ball $\mathscr{H}(\mathbf{x}^e_q,\delta)$, as illustrated in Figures \ref{1D_ball_quadrature_boundary} and \ref{2D_ball_quadrature_boundary} for a one-dimensional and a two-dimensional case, respectively.
\begin{figure}[H]
\begin{center}
\scalebox{0.5}{\includegraphics[trim = 0mm 0mm 0mm 0mm, clip=true,width=1\textwidth]{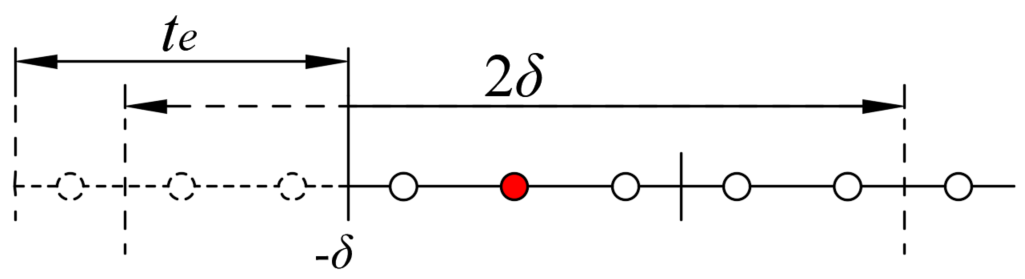}} 
\caption{One-dimensional partial integration ball for the filled red point and mesh extension of size $t_e$. The shown region is the left region of a domain $\mathscr{B}\Omega\cup\Omega=\left[-\delta,1+\delta\right]$, with $\Omega=\left(0,1\right)$.}
\label{1D_ball_quadrature_boundary}
\end{center}
\end{figure}

\begin{figure}[H]
\begin{center}
\scalebox{0.5}{\includegraphics[trim = 23mm 0mm 23mm 0mm, clip=true,width=1\textwidth]{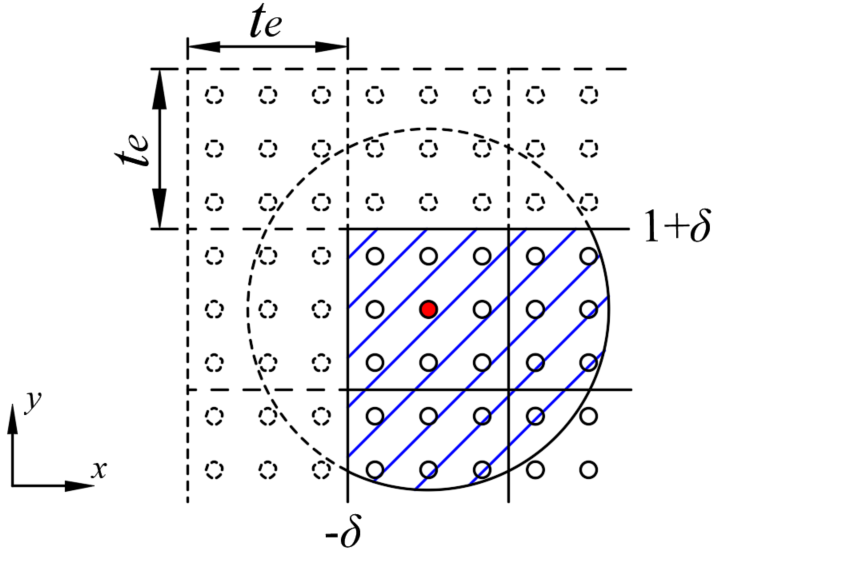}} 
\caption{Two-dimensional partial integration Euclidean ball for the filled red point (shaded area) and mesh extension of size $t_e$. The shown region is the top-left region of a domain $\mathscr{B}\Omega\cup\Omega=\left[-\delta,1+\delta\right]\times\left[-\delta,1+\delta\right]$, with $\Omega=\left(0,1\right)\times\left(0,1\right)$.}
\label{2D_ball_quadrature_boundary}
\end{center}
\end{figure}

We can now employ the optimization procedure from Section \ref{sec:GMLS_construction}, for the finite dimensional space defined in (\ref{Vh_exact_constraints}), to determine the set of weights $\{\omega^e_{qp}\}_{p=1}^{\tilde{N}_{qp}}$.
We can then select $\left\{{\mathbf{x}}^e_{qp}\right\}_{p=1}^{{N}_{qp}}\subseteq\left\{{\mathbf{x}}^e_{qp}\right\}_{p=1}^{\tilde{N}_{qp}}$ as

\begin{equation}\label{Inner_quadrature_set_extended_reduced}
\begin{aligned}
\left\{{\mathbf{x}}^e_{qp}\right\}_{p=1}^{{N}_{qp}}&\coloneqq\left \{\mathbf{x}^e_{qp}\in \left\{{\mathbf{x}}^e_{qp}\right\}_{p=1}^{\tilde{N}_{qp}}\cap\left(\Omega\cup\mathscr{B}\Omega\right)
\right \},    
\end{aligned}
\end{equation}
and their associated weights $\left\{\omega^e_{qp}\right\}_{p=1}^{{N}_{qp}}\subseteq\left\{{\omega}^e_{qp}\right\}_{p=1}^{\tilde{N}_{qp}}$. Therefore, from Eqs. (\ref{weak_form_discretequad3}) and (\ref{weak_form_discretequad4}), we can obtain

\begin{equation}
\begin{aligned}
    & D(w^h,v^h)\\
    \approx&\sum_{\Omega_e^h\in\mathcal{M}^h}\sum_{q=1}^{N_q}\int_{\left(\Omega\cup\mathscr{B}\Omega\right)\cap\mathscr{H}(\mathbf{x}^e_q,\delta)}\left[v^h(\mathbf{y})-v^h(\mathbf{x}^e_q)\right]\gamma(\mathbf{x}^e_q,\mathbf{y})\left[w^h(\mathbf{y})-w^h(\mathbf{x}^e_q)\right]d\mathbf{y}\omega^e_q\\
    \approx&\sum_{\Omega_e^h\in\mathcal{M}^h}\sum_{q=1}^{N_q}\sum_{p=1}^{N_{qp}}\left[v^h(\mathbf{x}^e_{qp})-v^h(\mathbf{x}^e_q)\right]\gamma(\mathbf{x}^e_q,\mathbf{x}^e_{qp})\left[w^h(\mathbf{x}^e_{qp})-w^h(\mathbf{x}^e_q)\right]\omega^e_{qp}\omega^e_q\\
    =& D^h(w^h,v^h),
\end{aligned}
    \label{weak_form_discretequad_full1}
\end{equation}
and

\begin{equation}
\begin{aligned}
    &G(v^h)-D(g^h,v^h)\\
    \approx&\sum_{\Omega_e^h\in\mathcal{M}^h_\Omega}\sum_{b=1}^{N_b}v^h(\mathbf{x}^e_b)b(\mathbf{x}^e_b)\omega^e_b\\
    &-\sum_{\Omega_e^h\in\mathcal{M}^h}\sum_{q=1}^{N_q}\int_{\left(\Omega\cup\mathscr{B}\Omega\right)\cap\mathscr{H}(\mathbf{x}^e_q,\delta)}\left[v^h(\mathbf{y})-v^h(\mathbf{x}^e_q)\right]\gamma(\mathbf{x}^e_q,\mathbf{y})\left[g^h(\mathbf{y})-g^h(\mathbf{x}^e_q)\right]d\mathbf{y}\omega^e_q\\
    \approx&\sum_{\Omega_e^h\in\mathcal{M}^h_\Omega}\sum_{b=1}^{N_b}v^h(\mathbf{x}^e_b)b(\mathbf{x}^e_b)\omega^e_b\\
    &-\sum_{\Omega_e^h\in\mathcal{M}^h}\sum_{q=1}^{N_q}\sum_{p=1}^{N_{qp}}\left[v^h(\mathbf{x}^e_{qp})-v^h(\mathbf{x}^e_q)\right]\gamma(\mathbf{x}^e_q,\mathbf{x}^e_{qp})\left[g^h(\mathbf{x}^e_{qp})-g^h(\mathbf{x}^e_q)\right]\omega^e_{qp}\omega^e_q\\
    =&G^h(v^h)-D^h(g^h,v^h),
\end{aligned}
    \label{weak_form_discretequad_full2}
\end{equation}
where we defined

\begin{equation}
\begin{aligned}
    D^h(\cdot,v^h)\coloneqq\sum_{\Omega_e^h\in\mathcal{M}^h}\sum_{q=1}^{N_q}\sum_{p=1}^{N_{qp}}&\left[v^h(\mathbf{x}^e_{qp})-v^h(\mathbf{x}^e_q)\right]\gamma(\mathbf{x}^e_q,\mathbf{x}^e_{qp})\\&\left[(\cdot)(\mathbf{x}^e_{qp})-(\cdot)(\mathbf{x}^e_q)\right]\omega^e_{qp}\omega^e_q,
\end{aligned}
    \label{weak_form_Dh}
\end{equation}
and

\begin{equation}
\begin{aligned}
    &G^h(v^h)\coloneqq\sum_{\Omega_e^h\in\mathcal{M}^h_\Omega}\sum_{b=1}^{N_b}v^h(\mathbf{x}^e_b)b(\mathbf{x}^e_b)\omega^e_b.
\end{aligned}
    \label{weak_form_Gh}
\end{equation}

\subsection{Fully discrete variational form for nonlocal diffusion}

We combine the finite element discretization from Section \ref{FEM_discretization} with the discrete quadrature approach discussed in Section \ref{quadrature_discretization}. By substituting the continuous operators in (\ref{weak_form_discrete2}) with the discrete ones defined in 
Eqs. (\ref{weak_form_Dh}) and (\ref{weak_form_Gh}), we get

\begin{equation}
\begin{aligned}
    D^h(w^h,v^h)=G^h(v^h)-D^h(g^h,v^h),
\end{aligned}
    \label{weak_form_fullydiscrete1}
\end{equation}
which, by employing Eqs. (\ref{uh_fem}) and (\ref{uh_fem2}), results in the following linear system

\begin{equation}
\begin{aligned}
    \sum_{j=1}^{J_\Omega}D^h(\psi_j,\psi_i)u_j=G^h(\psi_i)-D^h(g^h,\psi_i),
\end{aligned}
    \label{weak_form_fullydiscrete2}
\end{equation}
for $i=1,\ldots,J_\Omega$.
Eq.~(\ref{weak_form_fullydiscrete2}) can be expressed in matrix form as

\begin{equation}
\begin{aligned}
    \mathbf{A}^h\mathbf{u}=\mathbf{f}^h,
\end{aligned}
    \label{weak_form_fullydiscrete3}
\end{equation}
where $\mathbf{A}^h$ is a $J_\Omega\times J_\Omega$ matrix with components 

\begin{equation}
\begin{aligned}
    A^h_{ij}&=D^h(\psi_j,\psi_i)\\
    &=\sum_{\Omega_e^h\in\mathcal{M}^h}\sum_{q=1}^{N_q}\sum_{p=1}^{N_{qp}}\left[\psi_i(\mathbf{x}^e_{qp})-\psi_i(\mathbf{x}^e_q)\right]\gamma(\mathbf{x}^e_q,\mathbf{x}^e_{qp})\left[\psi_j(\mathbf{x}^e_{qp})-\psi_j(\mathbf{x}^e_q)\right]\omega^e_{qp}\omega^e_q,
\end{aligned}
    \label{weak_form_fullydiscrete4}
\end{equation}
$\mathbf{f}^h$ is a $J_\Omega\times 1$ vector with components

\begin{equation}
\begin{aligned}
    f^h_{i}&=G^h(\psi_i)-D^h(g^h,\psi_i)=\sum_{\Omega_e^h\in\mathcal{M}^h_\Omega}\sum_{b=1}^{N_b}\psi_i(\mathbf{x}^e_b)b(\mathbf{x}^e_b)\omega^e_b\\
    &-\sum_{\Omega_e^h\in\mathcal{M}^h}\sum_{q=1}^{N_q}\sum_{p=1}^{N_{qp}}\left[\psi_i(\mathbf{x}^e_{qp})-\psi_i(\mathbf{x}^e_q)\right]\gamma(\mathbf{x}^e_q,\mathbf{x}^e_{qp})\left[g^h(\mathbf{x}^e_{qp})-g^h(\mathbf{x}^e_q)\right]\omega^e_{qp}\omega^e_q,
\end{aligned}
    \label{weak_form_fullydiscrete5}
\end{equation}
and $\mathbf{u}$ is a vector of size $J_\Omega\times 1$ containing the set of unknown coefficients $\{u_j\}_{j=1}^{J_\Omega}$ to be determined.

\section{Properties of the numerical scheme}\label{sec:convergence}

In this section, we investigate the numerical properties of the proposed scheme. We first describe different types of convergence in the context of nonlocal models (see Figure \ref{Convergence_types_PD}) and then provide a convergence analysis in the $H^1$ norm in a simplified, one-dimensional setting.

\begin{figure}[H]
\begin{center}
\scalebox{1.0}{\includegraphics[trim = 23mm 40mm 23mm 23mm, clip=true,width=1\textwidth]{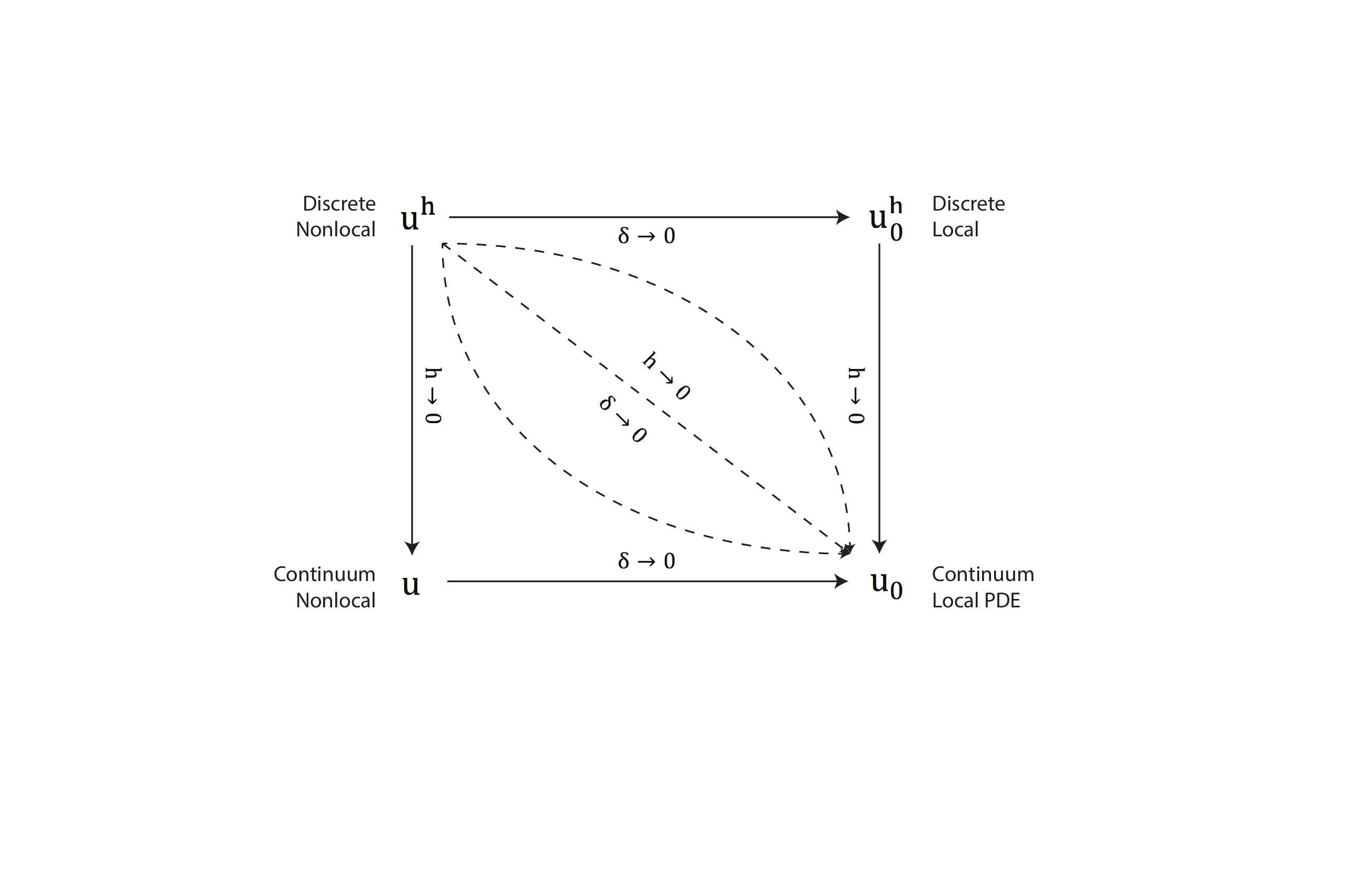}} 
\caption{Different convergence paths in finite-length nonlocal model. $u$ and $u_0$ represent the continuum nonlocal and local solutions, respectively, while $u^h$ and $u_0^h$ are their discrete counterparts.}
\label{Convergence_types_PD}
\end{center}
\end{figure}

\subsection{Brief review of asymptotically compatible schemes}
As described in Section \ref{sec:Nonlocal_diff_model}, continuum nonlocal models are characterized by the length scale $\delta$. Under proper regularity assumptions, as $\delta\to 0$, nonlocal solutions converge to their local, PDE counterparts \cite{du12}; we refer to this type of convergence as $\delta$-convergence. When a discretization scheme is employed, its size $h$ introduces a second length scale. For a fixed horizon $\delta$, a discretization scheme is $h$-convergent if the nonlocal discrete solution converges to the continuum nonlocal solution as $h\rightarrow0$.
Lastly, a discretization scheme is called  \textit{asymptotically compatible} if, in addition to the $\delta$- and $h$-convergence above, the discrete solution to the nonlocal problem also converges to the analytical solution of its local PDE counterpart as $\delta\rightarrow0$ and $h\rightarrow0$ \cite{XTian2014_AC,peryhandbook}. For numerical schemes where $h$ and $\delta$ are tied together by the relationship $m=\delta/h$, the term \textit{asymptotic compatibility} refers only to the last type of convergence described above, i.e., meaning that as $\delta\rightarrow0$ and $h\rightarrow0$, the discrete solution to the nonlocal problem converges to the analytical solution of the associated local problem \cite{trask2019asymptotically,Leng2021_AsymptoticallyCR}.
As the proposed scheme is such that $m=\delta/h\in\mathbb{N}$, we will only focus on the latter type of convergence (i.e.,  $\delta\rightarrow0$ and $h\rightarrow0$ simultaneously).

\subsection{Convergence analysis}

In this section we derive a preliminary estimate for the convergence of discrete solutions obtained via inexact quadrature of the inner integral. We assume that the outer quadrature is performed with a high-accuracy scheme whose contribution to the overall error can be considered negligible. In the following analysis we hence assume the outer integration to be exact. Therefore,

\begin{equation}
\begin{aligned}
    D^h(u^h,u^h)&=\int_{\Omega\cup\mathscr{B}\Omega}\sum_{j=1}^{N_j}\gamma(\mathbf{x},\mathbf{x}_j)\left[u^h(\mathbf{x}_j)-u^h(\mathbf{x})\right]^2\omega_j d\mathbf{x}\\
    &=\sum_{\Omega_e^h\in\mathcal{M}^h}\int_{\Omega_e^h}\sum_{j=1}^{N_j}\gamma(\mathbf{x},\mathbf{x}_j)\left[u^h(\mathbf{x}_j)-u^h(\mathbf{x})\right]^2\omega_j d\mathbf{x},
\end{aligned}
    \label{Proof1_Dh}
\end{equation}
where $\mathbf{x}_j$ and $\omega_j$ are the $j$-th inner quadrature point and associated weight in the ball $\mathscr{H}(\mathbf{x},\delta)$, respectively, and $N_j$ is the total number of inner quadrature points. Furthermore, we also restrict ourselves to kernels of the form \eqref{kernel_form1}. 

\subsubsection{Uniform $\mathcal{V}^h$-coercivity}\label{proof_coercivity}

We show, under certain conditions, that the approximate bilinear form $D^h(\cdot,\cdot):\mathcal{V}^h\times\mathcal{V}^h\rightarrow\mathbb{R}$ is uniformly $\mathcal{V}^h$-coercive. Recall that we are considering a $C^0$ linear FE approximation. Therefore, $\forall \xb$

\begin{equation}\label{Proof2_uh}
    u^h(\xbj)-u^h(\xb) =\left\{\begin{aligned}
         \left.\nabla u^h(\xb)\right|_e\cdot(\xbj-\xb), &\ \text{for $\xbj$, $\xb$ in element $e$,}\\\
        u^h(\xbj)-u^h(\xb), &\ \text{otherwise.}\\
\end{aligned}\right.   
\end{equation}

We assume that we have a subset of the quadrature points $\{\mathbf{x}_j\}$ that are in the same element as $\mathbf{x}$, $\{\mathbf{x}_{j_\text{in}}\}$, and another subset of the quadrature points that are not, $\{\mathbf{x}_{j_\text{out}}\}$. Membership of points in these subsets is linked to the chosen spacing for the quadrature points; we assume that this spacing is small enough relative to the element size so that $\{\mathbf{x}_{j_\text{in}}\}$ is nonempty. We also assume that the inner quadrature weights $\{\omega_j\}$ are positive. Then, Eq.~(\ref{Proof1_Dh}) can be recast as

\begin{equation}
\begin{aligned}
    D^h(u^h,u^h)=&\sum_{\Omega_e^h\in\mathcal{M}^h}\int_{\Omega_e^h}\left[\sum_{{j_\text{in}}=1}^{N_{j_\text{in}}}\gamma(\mathbf{x},\mathbf{x}_{j_\text{in}})\left[u^h(\mathbf{x}_{j_\text{in}})-u^h(\mathbf{x})\right]^2\omega_{j_\text{in}}\right.\\
    &\left.+\sum_{{j_\text{out}}=1}^{N_{j_\text{out}}}\gamma(\mathbf{x},\mathbf{x}_{j_\text{out}})\left[u^h(\mathbf{x}_{j_\text{out}})-u^h(\mathbf{x})\right]^2\omega_{j_\text{out}}\right] d\mathbf{x}\\
    \geq& \sum_{\Omega_e^h\in\mathcal{M}^h}\int_{\Omega_e^h}\sum_{{j_\text{in}}=1}^{N_{j_\text{in}}}\gamma(\mathbf{x},\mathbf{x}_{j_\text{in}})\left[u^h(\mathbf{x}_{j_\text{in}})-u^h(\mathbf{x})\right]^2\omega_{j_\text{in}}d\mathbf{x}\\
    =&\frac{\zeta}{\delta^{d+2}}\sum_{\Omega_e^h\in\mathcal{M}^h}\int_{\Omega_e^h}\sum_{{j_\text{in}}=1}^{N_{j_\text{in}}}\left[u^h(\mathbf{x}_{j_\text{in}})-u^h(\mathbf{x})\right]^2\omega_{j_\text{in}}d\mathbf{x}\\
    =&\frac{\zeta}{\delta^{d+2}}\sum_{\Omega_e^h\in\mathcal{M}^h}\int_{\Omega_e^h}\sum_{{j_\text{in}}=1}^{N_{j_\text{in}}}\left[\nabla u^h(\xb)\right|_e\cdot(\mathbf{x}_{j_\text{in}}-\xb)]^2\omega_{j_\text{in}}d\mathbf{x}.
\end{aligned}
    \label{Proof3_Dh_inout}
\end{equation}

For simplicity, we now restrict our discussion to the one-dimensional case.  In this setting, we assume that there exists a constant $C > 0$ independent of $h$ and $\delta$, such that $C\delta < \omega_\text{min} < \omega_{j_\text{in}}$ (which we verify by direct computation in \ref{sec:1d-weights}). Under these assumptions, we have the following coercivity result for the discrete bilinear form.
\begin{lem}\label{lem:coercive} There exists a constant $c>0$ independent of $h$ and $\delta$ such that $\forall u^h\in\mathcal{V}^h_0$ 
\begin{equation}\label{Proof5_Dh_inout2}
\begin{aligned}
D^h(u^h,u^h)\geq c\vert u^h\vert_{H^1}^2.
\end{aligned}
\end{equation}
\end{lem}
\begin{proof}
Restricting Eq.~(\ref{Proof3_Dh_inout}) to the one-dimensional setting, using our assumed lower bound on $\omega_{j_\text{in}}$, abbreviating restriction of $u^h$ to element $e$ as $u^h_e$, and allowing the symbol $C$ to be a generic constant independent of $h$ and $\delta$ (possibly with different numerical values in different places), we get
\newpage
\begin{equation}\label{Proof4_Dh_inout2}
\begin{aligned}
D^h(u^h,u^h)&\geq\frac{\zeta}{\delta^{3}}\sum_{\Omega_e^h\in\mathcal{M}^h}\int_{\Omega_e^h}\sum_{{j_\text{in}}=1}^{N_{j_\text{in}}}\left[u^h({x}_{j_\text{in}})-u^h({x})\right]^2\omega_{j_\text{in}}d{x}\\
&=\frac{\zeta}{\delta^{3}}\sum_{\Omega_e^h\in\mathcal{M}^h}\int_{\Omega_e^h}\left\lbrack\sum_{{j_\text{in}}=1}^{N_{j_\text{in}}}\left\{\frac{du^h_e}{dx}\cdot(x_{j_\text{in}}-x)\right\}^2\omega_{j_\text{in}}\right\rbrack\,d{x}\\
&=\frac{\zeta}{\delta^{3}}\sum_{\Omega_e^h\in\mathcal{M}^h}\left(\frac{du^h_e}{dx}\right)^2\int_{\Omega_e^h}\left\lbrack\sum_{{j_\text{in}}=1}^{N_{j_\text{in}}}(x_{j_\text{in}}-x)^2\omega_{j_\text{in}}\right\rbrack\,d{x}\\
&=\frac{\zeta}{\delta^{3}}\sum_{\Omega_e^h\in\mathcal{M}^h}\left(\frac{du^h_e}{dx}\right)^2\sum_{{j_\text{in}}=1}^{N_{j_\text{in}}}\left\{\int_{\Omega_e^h}\sum_{{j_\text{in}}=1}^{N_{j_\text{in}}}(x_{j_\text{in}}-x)^2\omega_{j_\text{in}}d{x}\right\}\\
&\geq\frac{\zeta}{\delta^{3}}\sum_{\Omega_e^h\in\mathcal{M}^h}\left(\frac{du^h_e}{dx}\right)^2\sum_{{j_\text{in}}=1}^{N_{j_\text{in}}}\left\{\int_{\Omega_e^h\setminus (x_{j_\text{in}}-h/4,x_{j_\text{in}}+h/4)}(x_{j_\text{in}}-x)^2\omega_{j_\text{in}}d{x}\right\}\\
&\geq\frac{\zeta}{\delta^{3}}\sum_{\Omega_e^h\in\mathcal{M}^h}\left(\frac{du^h_e}{dx}\right)^2\sum_{{j_\text{in}}=1}^{N_{j_\text{in}}}\left\{\int_{\Omega_e^h\setminus (x_{j_\text{in}}-h/4,x_{j_\text{in}}+h/4)}\left(\frac{h}{4}\right)^2\omega_{j_\text{in}}d{x}\right\}\\
&\geq\frac{\zeta}{\delta^{3}}\sum_{\Omega_e^h\in\mathcal{M}^h}\left(\frac{du^h_e}{dx}\right)^2\sum_{{j_\text{in}}=1}^{N_{j_\text{in}}}\left\{\left(\frac{h}{2}\right)\left(\frac{h}{4}\right)^2\omega_{j_\text{in}}\right\}\\
&\geq\frac{\zeta}{\delta^{3}}\sum_{\Omega_e^h\in\mathcal{M}^h}\left(\frac{du^h_e}{dx}\right)^2\left\{Ch^3\omega_\text{min}\right\}\\
&\geq\frac{\zeta}{\delta^{3}}\sum_{\Omega_e^h\in\mathcal{M}^h}\left(\frac{du^h_e}{dx}\right)^2\left\{Ch^3\delta\right\}\\
&\geq\frac{C}{\delta^{3}}\left(\sum_{\Omega_e^h\in\mathcal{M}^h}\left(\frac{du^h_e}{dx}\right)^2h\right)h^2\delta\\
&\geq\frac{Ch^2}{\delta^2}\vert u^h\vert_{H^1}^2,
\end{aligned}
\end{equation}
which gives the desired result when $h\sim\delta$.  
\end{proof}
%

\subsubsection{Preliminary convergence estimate}
We prove the convergence of numerical solutions for a simple one-dimensional case with a uniform grid and $\delta=h$. To do that, we first prove a lemma that holds in more general situations. 
We assume that for any $v^h\in \mathcal{V}^h_0$, $\int b(\xb) v^h(\xb)d\xb$ is exactly integrated. {By using the same arguments as in Strang's first lemma,} we have the following result.

\begin{lem}
\label{lem:strang}
There exists $C>0$ independent of $\delta$ such that 
\[
\| u-u^h\|_{\mathcal{V}} \leq C  \inf_{v^h\in \mathcal{V}_0^h} \left(\| u-v^h\|_{\mathcal{V}} +
\sup_{w^h\in \mathcal{V}_0^h}\frac{|D(v^h, w^h)-D^h( v^h, w^h )|}{\| w^h\|_{\mathcal{V}}}  \right). 
\]
In addition, if $u\in H^1$, then
\[
\| u-u^h\|_{H^1} \leq C \inf_{v^h\in \mathcal{V}_0^h} \left(\| u-v^h\|_{H^1} +
\sup_{w^h\in \mathcal{V}_0^h}\frac{|D(v^h, w^h)-D^h( v^h, w^h )|}{\| w^h\|_{\mathcal{V}}} \right).
\]
\end{lem}
\begin{proof}

By Lemma \ref{lem:coercive}, for any $v^h \in \mathcal{V}_0^h$,  we have 
\[
\begin{split}
c |u^h - v^h|_{H^1}^2  \leq&  D^h(u^h- v^h, u^h- v^h )  \\
 =& D(u-v^h, u^h -v^h) +(D(v^h, u^h -v^h)
 - D^h( v^h, u^h- v^h )) \\
 &+ (D^h(u^h, u^h- v^h) - D(u,u^h -v^h )) \\
 =&  D(u-v^h, u^h -v^h) + (D(v^h, u^h -v^h)-D^h( v^h, u^h- v^h )).
\end{split}
\]
By the boundedness of the bilinear form, i.e., $|D(u-v^h, u^h -v^h)|\leq  C \| u-v^h\|_{\mathcal{V}} \| u^h -v^h\|_{\mathcal{V}} $, where we write $\| w\|_{\mathcal{V}} = \sqrt{D(w,w)}$ for $w\in \mathcal{V}$, we have

\begin{equation}
\label{eq:stranglemma_1}
\begin{split}
c \frac{|u^h - v^h|_{H^1}^2}{\| u^h -v^h\|_{\mathcal{V}}}  &\leq  C \| u-v^h\|_{\mathcal{V}} + \frac{|D(v^h, u^h -v^h)-D^h( v^h, u^h- v^h )|}{\| u^h -v^h\|_{\mathcal{V}}} \\
& \leq  C \| u-v^h\|_{\mathcal{V}} +
\sup_{w^h\in \mathcal{V}_0^h}\frac{|D(v^h, w^h)-D^h( v^h, w^h )|}{\| w^h\|_{\mathcal{V}}}. 
\end{split}
\end{equation}
Notice that $H^1$ is continuously embedded in $\mathcal{V}$, i.e., 

\[
\| v\|_{\mathcal{V}}\leq C \| v\|_{H^1} \quad \forall v\in H^1,
\]
where the constant $C$ is independent of $\delta$ (see e.g., \cite{BBM01}). From \eqref{eq:stranglemma_1}, for any $v^h \in \mathcal{V}_0^h$

\[
\| u^h - v^h \|_{\mathcal{V}}\leq C \left( \| u-v^h\|_{\mathcal{V}} +
\sup_{w^h\in \mathcal{V}_0^h}\frac{|D(v^h, w^h)-D^h( v^h, w^h )|}{\| w^h\|_{\mathcal{V}}} \right)
\]
and, if in addition $u\in H^1$, 

\[
\| u^h - v^h \|_{H^1}\leq C \left( \| u-v^h\|_{H^1} +
\sup_{w^h\in \mathcal{V}_0^h}\frac{|D(v^h, w^h)-D^h( v^h, w^h )|}{\| w^h\|_{\mathcal{V}}} \right).
\]
Therefore from the triangle inequality 

\[
\| u - u^h\|\leq \| u - v^h\| + \| u^h - v^h \|
\]
with either the $\mathcal{V}$-norm or the $H^1$-norm, we can get the desired result. 
\end{proof}

{The following theorem provides a convergence result for} the simple one-dimensional case with a uniform grid, $\delta=h$ and a kernel function $\gamma(x,y) = \frac{1}{\delta^3} 1_{\{ |y-x|<\delta \}}$, where, for simplicity, we removed the constant $\zeta$.
\begin{thm}\label{thm:conv}
Assume we have a uniform grid in one-dimension and $\delta=h$.
In addition, assume that $u\in H^2$. Then,
\[
\| u - u^h\|_{H^1} \leq C h ,
\]
{where $C>0$ is a constant that depends on $\| u\|_{H^2}$, but is independent of $h$ and $\delta$.} 
\end{thm}
\begin{proof}
By Lemma \ref{lem:strang} and $u\in H^2$, we have 
\[
\| u-u^h\|_{H^1} \leq C  \inf_{v^h\in \mathcal{V}_0^h} \left(\| u-v^h\|_{H^1} +
\sup_{w^h\in \mathcal{V}_0^h}\frac{|D(v^h, w^h)-D^h( v^h, w^h )|}{\| w^h\|_{\mathcal{V}}}  \right)
\]
where, from Eq. (\ref{Vh0}), $\mathcal{V}^h_0$ is the space of continuous piecewise linear functions that satisfy zero Dirichlet boundary conditions. Taking $v^h:= I_h u$,  the piecewise linear interpolation of $u$, then it is well-known in finite element analysis that 

\[
\| u-I_h u \|_{H^1}\leq C h \| u\|_{H^2}. 
\]

Let $\Omega\cup\mathscr{B}\Omega = \cup_{i=1}^N \Omega^h_i$ and extend functions by zero outside $\Omega\cup\mathscr{B}\Omega$ when necessary (e.g., on $\mathscr{B}\Omega^{t_e}$). Then, by assuming $\delta=h$, we have 

\begin{equation}
\begin{split}
&D(v^h, w^h)=\\ &= \sum_{i=1}^N \int_{\Omega^h_i } \int_{\mathscr{H}{(x,\delta)}} \gamma(x,y) (v^h(y) - v^h (x))  (w^h(y) - w^h (x))   dy dx  \\
&=  \sum_{i=1}^N \int_{\Omega^h_i } \int_{\Omega^h_{i-1} \cup \Omega^h_i \cup \Omega^h_{i+1}} \gamma(x,y) (v^h(y) - v^h (x))  (w^h(y) - w^h (x))   dy dx \,.
\end{split}
\end{equation} 
Notice that in the above equation, $\Omega^h_0$ and $\Omega^h_{N+1}$ are outside $\Omega\cup\mathscr{B}\Omega$. Since the functions are always zero on $\Omega^h_0, \Omega^h_1, \Omega^h_N, \Omega^h_{N+1}$, we see that $\int_{\Omega^h_1} \int_{\Omega^h_0}  \cdots$ and $\int_{\Omega^h_N} \int_{\Omega^h_{N+1}}  \cdots$ are zero. 
Assume that on each $\Omega^h_i$, $v^h(x) $ is a linear function of slope $a_i \in \mathbb{R}$, and $w^h(x) $ is a linear function of slope $b_i \in \mathbb{R}$, then we have 

\begin{equation}
\begin{aligned}
&\int_{\Omega^h_i } \int_{\Omega^h_i} \gamma(x,y) (v^h(y) - v^h (x))  (w^h(y) - w^h (x))   dy dx =\\&= \int_{\Omega^h_i } \int_{\Omega^h_i} \gamma(x,y) (y-x)^2 a_i b_i dy dx . 
\end{aligned}
\end{equation}
Now we calculate $\int_{\Omega^h_i } \int_{\Omega^h_{i+1}}  \cdots dy dx$.  In this case, we have $y \in \Omega^h_{i+1}$ and $x \in \Omega^h_i$. Let $s_i$ denote the point that connects $\Omega^h_i$ and $\Omega^h_{i+1}$, then we can write 

\begin{equation}
\begin{split}
&v^h(y) = v^h(s_i) + (y-s_i) a_{i+1}, \quad w^h(y) = w^h(s_i) + (y-s_i) b_{i+1} ,  \\
&v^h(x) = v^h(s_i) + (x-s_i) a_{i}, \quad w^h(x) = w^h(s_i) + (x-s_i) b_{i} ,
\end{split}
\end{equation}
for all $y\in \Omega^h_{i+1}$ and $x \in \Omega^h_i$. Therefore, 

\begin{equation}
v^h(y) - v^h(x) = (y-s_i) a_{i+1} + (s_i-x) a_i = (y-s_i) (a_{i+1} -a_i) + (y-x) a_i 
\end{equation}
and similarly for $w^h(y) - w^h(x) $. We then have 

\begin{equation}
\begin{aligned}
&\int_{\Omega^h_i } \int_{\Omega^h_{i+1}} \gamma(x,y) (v^h(y) - v^h (x))  (w^h(y) - w^h (x))   dy dx \\
= &\int_{\Omega^h_i } \int_{\Omega^h_{i+1}} \gamma(x,y)\left[(y-s_i) (a_{i+1} -a_i) + (y-x) a_i \right]  \\ &\phantom{ \int_{\Omega^h_i } \int_{\Omega^h_{i+1}} \gamma(x,y)}\left[(y-s_i) (b_{i+1} -b_i) + (y-x) b_i \right]   dy dx .    \\
\end{aligned}
\end{equation}
Notice that, in the above, there is a term $\int_{\Omega^h_i } \int_{\Omega^h_{i+1}}  \gamma(x,y)  (y-x)^2 a_i b_i dydx $ which can be combined with $\int_{\Omega^h_i } \int_{\Omega^h_{i}} \gamma(x,y)  (y-x)^2 a_i b_i  dy dx$. 
The rest of the terms can be written as 

\begin{equation}\label{eq:I_i_right}
\begin{aligned}
&\int_{\Omega^h_i } \int_{\Omega^h_{i+1}} \gamma(x,y) \left((y-s_i)^2 (a_{i+1} -a_i)(b_{i+1} - b_i)\right. \\ &\left.+ (y-s_i)(y-x) \left[ (a_{i+1}-a_i) b_i  + (b_{i+1} - b_i) a_i \right]  \right)   dy dx  \\
= & \int_{\Omega^h_i } \int_{\Omega^h_{i+1}} \gamma(x,y) \left((y-s_i)^2 (a_{i+1} b_{i+1} - a_i b_i) \right.\\ &\left.+ (y-s_i)(s_i-x) \left[ (a_{i+1}-a_i) b_i  + (b_{i+1} - b_i) a_i \right]  \right)   dy dx  \\
= & (a_{i+1} b_{i+1} - a_i b_i)   \int_{\Omega^h_i} \int_{s_i}^{x+\delta} \frac{(y-s_i)^2}{\delta^3}  dy dx \\ & +  \left[ (a_{i+1}-a_i) b_i  + (b_{i+1} - b_i) a_i \right]  \int_{\Omega^h_i} \int_{s_i}^{x+\delta} \frac{(y-s_i)(s_i-x)}{\delta^3}    dy dx . 
\end{aligned}
\end{equation}
We can similarly calculate $\int_{\Omega^h_i} \int_{\Omega^h_{i-1}} \cdots dydx$ and get  $\int_{\Omega^h_i } \int_{\Omega^h_{i-1}}  \gamma(x,y)  (y-x)^2 a_i b_i dydx  $ (which is to be combined with $\int_{\Omega^h_i } \int_{\Omega^h_{i}} \gamma(x,y)  (y-x)^2 a_i b_i  dy dx$) and 

\begin{equation}
\label{eq:I_i_left}
\begin{aligned}
&(a_{i-1} b_{i-1} - a_i b_i)   \int_{\Omega^h_i} \int_{x-\delta}^{s_{i-1}} \frac{(y-s_{i-1})^2}{\delta^3}  dy dx  + \\&  \left[ (a_{i-1}-a_i) b_i  + (b_{i-1} - b_i) a_i \right]  \int_{\Omega^h_i} \int_{x-\delta}^{s_{i-1}} \frac{(y-s_{i-1})(s_{i-1}-x)}{\delta^3}    dy dx .
\end{aligned}
 \end{equation}
Replacing $i-1$ with $i$ in \eqref{eq:I_i_left}, we get the contribution from $\int_{\Omega^h_{i+1}} \int_{\Omega^h_i} \cdots dydx$: 

 \begin{equation}
 \label{eq:I_i+1_left}
 \begin{aligned}
&(a_{i} b_{i} - a_{i+1} b_{i+1})   \int_{\Omega^h_{i+1}} \int_{x-\delta}^{s_{i}} \frac{(y-s_{i})^2}{\delta^3}  dy dx  \\&+  \left[ (a_{i}-a_{i+1}) b_{i+1}  + (b_{i} - b_{i+1}) a_{i+1} \right]  \int_{\Omega^h_{i+1}} \int_{x-\delta}^{s_{i}} \frac{(y-s_{i})(s_{i}-x)}{\delta^3}    dy dx 
\end{aligned}
 \end{equation}
Now by adding \eqref{eq:I_i_right} with \eqref{eq:I_i+1_left} and noticing, from symmetry, that 
  
\begin{equation}
\label{eq:symmetery_continuous}
\begin{split}
& \int_{\Omega^h_i} \int_{s_i}^{x+\delta} \frac{(y-s_i)^2}{\delta^3}  dy dx  = \int_{\Omega^h_{i+1}} \int_{x-\delta}^{s_{i}} \frac{(y-s_{i})^2}{\delta^3}  dy dx  \\
 &\int_{\Omega^h_i} \int_{s_i}^{x+\delta} \frac{(y-s_i)(s_i-x)}{\delta^3}    dy dx = \int_{\Omega^h_{i+1}} \int_{x-\delta}^{s_{i}} \frac{(y-s_{i})(s_{i}-x)}{\delta^3}    dy dx , 
\end{split}
\end{equation}
we get 

\begin{equation}
\begin{aligned}
|  \eqref{eq:I_i_right} + \eqref{eq:I_i+1_left} | &\leq 2 |a_{i+1} -a_{i}| |b_{i+1} - b_i|   \int_{\Omega^h_i} \int_{s_i}^{x+\delta} \frac{|y-s_i||s_i-x|}{\delta^3}    dy dx  \\& \leq C h |a_{i+1} -a_{i}| |b_{i+1} - b_i|. 
\end{aligned}
 \end{equation}
Combining the above results, we have 

\begin{equation}\label{eq:continuous_form}
\begin{aligned}
    D(v^h, w^h) &= \sum_{i=1}^N \int_{\Omega^h_i}  \int_{\mathscr{H}{(x,\delta)}} \gamma(x,y)  (y-x)^2 a_i b_i dydx \\&\phantom{= }+ \sum_{i=0}^N |a_{i+1} -a_{i}| |b_{i+1} - b_i| O(h) . 
\end{aligned}
\end{equation}
 
Now to estimate $ D^h(v^h, w^h) $, we follow the exact procedure for $D(v^h, w^h)$, but with the inner integral replaced by GMLS quadrature. 
In particular, if we have symmetry of the quadrature points, then 

\begin{equation}
\label{eq:symmetery_discrete}
\begin{split}
    & \int_{\Omega^h_i} \sum_{ s_i < y_j < x+\delta} \frac{(y_j -s_i)^2}{\delta^3} \omega_j   dx  = \int_{\Omega^h_{i+1}} \sum_{x-\delta< y_j < s_{i}} \frac{(y_j-s_{i})^2}{\delta^3} \omega_j dx  \\
    &\int_{\Omega^h_i} \sum_{ s_i < y_j < x+\delta}  \frac{(y_j-s_i)(s_i-x)}{\delta^3}  \omega_j   dx = \int_{\Omega^h_{i+1}}  \sum_{x-\delta< y_j < s_{i}} \frac{(y_j-s_{i})(s_{i}-x)}{\delta^3} \omega_j dx .
\end{split}
\end{equation}
Then we can show that 

\begin{equation}\label{eq:discrete_form}
\begin{aligned}
D^h(v^h, w^h) &= \sum_{i=1}^N \int_{\Omega^h_i}  \sum_{j=1}^{NP} \gamma(x,y_j)  (y_j-x)^2 a_i b_i \omega_j dx\\&\phantom{= } +   \sum_{i=0}^N |a_{i+1} -a_{i}| |b_{i+1} - b_i| O(h) .  
\end{aligned}
\end{equation}
Comparing \eqref{eq:continuous_form} with \eqref{eq:discrete_form}, we notice that $\int_{\Omega^h_i}\int_{\mathscr{H}{(x,\delta)}} \gamma(x,y)  (y-x)^2 dy dx= \int_{\Omega^h_i} \sum_{j=1}^{NP} \gamma(x,y_j)  (y_j-x)^2  \omega_j dx$.  We therefore only need an estimate of

\begin{equation}
\label{eq:strang_estimate}
 \sup_{w^h \in \mathcal{V}^h_0} \frac{ \sum_{i=0}^N |a_{i+1} -a_{i}| |b_{i+1} - b_i| O(h) }{\| w^h\|_{\mathcal{V}}}
\end{equation}
with $v^h=I_h u$. 
Notice that $\| w^h\|_{\mathcal{V}} = \sqrt{D(w^h, w^h)}$, so we can write it out by the same procedure above and get 

\begin{equation}
\begin{aligned}
\| w^h\|_{\mathcal{V}}^2 &= \sum_{i=1}^N \int_{\Omega^h_i}  \int_{\mathscr{H}{(x,\delta)}} \gamma(x,y)  (y-x)^2 b^2_i dydx \\ &\phantom{= } -2 \sum_{i=0}^N (b_{i+1} - b_i)^2 \int_{\Omega^h_i} \int_{s_i}^{x+\delta} \frac{(y-s_i)(s_i-x)}{\delta^3}    dy dx.  
\end{aligned}
\end{equation}
By letting $\gamma(x,y) = \frac{1}{\delta^3} 1_{\{ |y-x|<\delta \}}$ and a direct calculation of the above integrals, we get

\begin{equation}
\| w^h\|_{\mathcal{V}}^2  =  \frac{2h}{3}  \sum_{i=1}^N  b_i^2  -\frac{h }{12} \sum_{i=0}^N (b_{i+1} - b_i)^2   = h \left(\frac{2}{3 } \sum_{i=1}^{N}  b_i^2 -  \frac{1}{12} \sum_{i=1}^{N-1} (b_{i+1} - b_i)^2 \right) ,
\end{equation}
where the last equality is a result of $b_0=b_1=b_{N}=b_{N+1} = 0$. Therefore,

\begin{equation}
\begin{split}
& \phantom{= }\;\; \frac{ \sum_{i=0}^N |a_{i+1} -a_{i}| |b_{i+1} - b_i| O(h) }{\| w^h\|_{\mathcal{V}^h_0}} \\ &= \frac{ \sum_{i=1}^{N-1} |a_{i+1} -a_{i}| |b_{i+1} - b_i| O(h) }{  \sqrt{h \left(\frac{2}{3 } \sum_{i=1}^{N}  b_i^2 -  \frac{1}{12} \sum_{i=1}^{N-1} (b_{i+1} - b_i)^2\right)} }  \\
&\leq O(h) \frac{  \sqrt{ \sum_{i=1}^{N-1}|a_{i+1} -a_{i}|^2 } \sqrt{\sum_{i=1}^{N-1} |b_{i+1} - b_i|^2} }{  \sqrt{h \left(\frac{2}{3 } \sum_{i=1}^{N}  b_i^2 -  \frac{1}{12} \sum_{i=1}^{N-1} (b_{i+1} - b_i)^2\right)} }  \\
& = O(h) \frac{  \sqrt{ \sum_{i=1}^{N-1}|a_{i+1} -a_{i}|^2 }}{  \sqrt{h \left( \frac{2}{3 }  (\sum_{i=1}^{N}  b_i^2)/ (\sum_{i=1}^{N-1} (b_{i+1} - b_i)^2) -  \frac{1}{12} \right)} }  . 
\end{split}
\end{equation}
Notice that $\sum_{i=1}^{N-1} (b_{i+1} - b_i)^2 = 2 \sum_{i=1}^N b_i^2 - 2\sum_{i=1}^{N-1} b_{i+1}b_{i} \leq 4 \sum_{i=1}^N b_i^2 $, therefore,

\begin{equation}
\begin{aligned}
\frac{ \sum_{i=0}^N |a_{i+1} -a_{i}| |b_{i+1} - b_i| O(h) }{\| w^h\|_{\mathcal{V}}} & \leq  O(h) \frac{  \sqrt{ \sum_{i=1}^{N-1}|a_{i+1} -a_{i}|^2 }}{  \sqrt{h \left( \frac{8}{3 } -  \frac{1}{12} \right)} } \\&  = O(h) \sqrt{\frac{  \sum_{i=1}^{N-1}|a_{i+1} -a_{i}|^2 }{ h}}. 
\end{aligned}
\end{equation}
  Since $v^h $ is the piecewise linear interpolation of $u$, then $a_i = u^\prime(x_i)$ for some $x_i \in \Omega^h_i$, so

\begin{equation}
     |a_{i+1} - a_{i}| = |u^\prime(x_{i+1}) - u^\prime(x_i)|= \left|\int_{x_i}^{x_{i+1}} u^{\prime\prime}(s)ds\right| \leq  h \int_{x_i}^{x_{i+1}} |u^{\prime\prime}(s)|ds, 
\end{equation}
where the last inequality comes from Cauchy-Schwartz inequality. 
Therefore we have $\sqrt{\frac{  \sum_{i=1}^{N-1}|a_{i+1} -a_{i}|^2 }{ h} }\leq \| u\|_{H^2}$.  
All together, we have shown 

\begin{equation}
    \sup_{w^h \in \mathcal{V}^h_0} \frac{|D(v^h, w^h) - D^h(v^h, w^h)|}{\| w^h\|_{\mathcal{V}}} \leq C h \| u\|_{H^2}  
\end{equation}
for $v^h = I_h u$, 
and therefore the desired result. 
\end{proof}

\begin{rem}\label{remark_rate}
The proof of Theorem \ref{thm:conv} utilizes the structure of the uniform grid. In particular, \eqref{eq:symmetery_continuous} and \eqref{eq:symmetery_discrete} hold only if we have a uniform grid. For quasi-uniform grids, i.e., non-uniform grids with bounded ratio between the maximum mesh size $h_{\max}$ and the miniumum mesh size $h_{\min}$, we can follow the similar arguments so that \eqref{eq:strang_estimate} is then replaced with \[
 \sup_{w^h \in \mathcal{V}^h_0} \frac{  \left( \sqrt{\sum_{i=0}^N |a_{i}|^2}\sqrt{\sum_{i=0}^N |b_i|^2}  \right) O(h_{\max}) }{\| w^h\|_{\mathcal{V}}}
\]
from where one can proceed to show an $O(1)$ estimate of $\| u- u^h \|_{H^1}$.
This estimate will also be numerically verified later. 
\end{rem}

\section{Numerical examples}\label{sec:numerics}

In this section, we present numerical convergence results obtained by employing the proposed quadrature scheme. We consider one-dimensional and two-dimensional problems discretized on uniform and non-uniform grids.

To evaluate the accuracy of the numerical solutions and test the convergence properties of the proposed method, we employ the $L^2$ and $H^1$ norms of the difference between the nonlocal numerical solution, $u^h$, and the analytical solution, $u_0$, to a local Poisson problem, i.e.,

\begin{equation}
    \lVert u^h(\mathbf{x})-u_0(\mathbf{x})\rVert_{L^2}=\left[\int_{\Omega}\left(u^h(\mathbf{x})-u_0(\mathbf{x})\right)^2d\mathbf{x}\right]^{\frac{1}{2}},
\end{equation}
and

\begin{equation}
    \lVert u^h(\mathbf{x})-u_0(\mathbf{x})\rVert_{H^1}=\left[\int_{\Omega}\left(u^h(\mathbf{x})-u_0(\mathbf{x})\right)^2+\left(\nabla u^h(\mathbf{x})-\nabla u_0(\mathbf{x})\right)^2d\mathbf{x}\right]^{\frac{1}{2}}.
\end{equation}
These norms are computed numerically with Gauss quadrature over the mesh elements, i.e.

\begin{equation}
    \lVert u^h(\mathbf{x})-u_0(\mathbf{x})\rVert_{L^2}\approx\left[\sum_{\Omega^h_e\in\mathcal{M}^h_\Omega}\sum_{\mathbf{x}^e_{gs}\in\Omega^h_e}\left(u^h(\mathbf{x}_{gs})-u_0(\mathbf{x}_{gs})\right)^2\omega_{gs}\right]^{\frac{1}{2}},
\end{equation}
and

\begin{equation}
    \begin{aligned}
     \lVert u^h(\mathbf{x})-u_0(\mathbf{x})\rVert_{H^1}\approx&\Bigg\{\sum_{\Omega^h_e\in\mathcal{M}^h_\Omega}\sum_{\mathbf{x}^e_{gs}\in\Omega^h_e}\left[\left(u^h(\mathbf{x}_{gs})-u_0(\mathbf{x}_{gs})\right)^2\right.\\&+\left.\left(\nabla u^h(\mathbf{x}_{gs})-\nabla u_0(\mathbf{x}_{gs})\right)^2\right]\omega_{gs}\Bigg\}^{\frac{1}{2}},
    \end{aligned}
\end{equation}
where $\left\{\mathbf{x}_{gs}\right\}_{gs=1}^{N_{gs}}$ and $\left\{\omega_{gs}\right\}_{gs=1}^{N_{gs}}$, ${N_{gs}}\in\mathbb{N}$, are the element Gauss quadrature points and weights, respectively. In this work, we take $N_{gs}=8^d$, where $d$ is the dimension of the problem. Also, in all our numerical examples, we employ fixed ratios $m=\delta/h\in\mathbb{N}$.

\subsection{One-dimensional test cases}\label{numerical_1D}

\begin{figure} [H]
\centering
\vspace{0pt}  
\includegraphics[trim = 0mm 0mm 0mm -10mm, clip=true,width=0.65\textwidth]{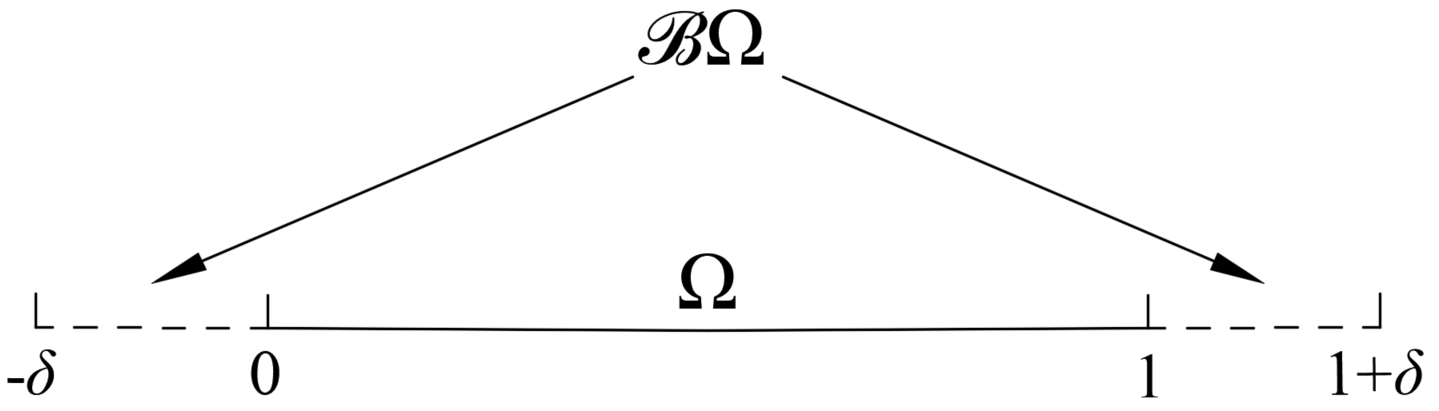}
\caption{One-dimensional domain $\Omega$, with associated boundary layer $\mathscr{B}\Omega$.}
\label{PD_1Ddomain_figure}
\end{figure}

We consider a one-dimensional domain $\Omega=(0,1)$. For a given horizon $\delta$, its associated interaction domain is $\mathscr{B}\Omega=[-\delta,0]\cup[1,\delta]$. Note that the inner domain, where the function $u$ is unknown (see Eq. (\ref{strong form_NLD.})), is considered constant in size, while the boundary layer varies with the value of $\delta$. Thus, during our convergence studies, the inner solution domain $\Omega$ remains consistent during the refinement ($\delta\rightarrow 0$) so that the $L^2$ error norms associated with each considered value of $\delta$ are comparable. We consider the following kernel functions: the constant kernel 

\begin{equation}
    \gamma_{1,c}({x},{y}) = \left\{\begin{aligned}
         \ \frac{3}{2\delta^3} \quad\ &\rm{for}\ |{y}-{x}|\leq\delta,\\\
        \ 0 \ \ \quad\ &\rm{for}\ |{y}-{x}|>\delta,\\
\end{aligned}\right.   
\label{kernel_form1_1D}
\end{equation}
and the rational kernel

\begin{equation}
    \gamma_{1,r}({x},{y}) =\left\{\begin{aligned}
         \ \frac{1}{\delta^2 |{y}-{x}|} \quad \ &\rm{for}\ |{y}-{x}|\leq\delta,\\\
        \ 0 \ \ \quad\ &\rm{for}\ |{y}-{x}|>\delta,\\
\end{aligned}\right.  
\label{kernel_form2_1D}
\end{equation}
which correspond to the expressions in Eqs. (\ref{kernel_form1}) and (\ref{kernel_form2}) for $\zeta=3/2$ and $\zeta=1$, respectively. These values of $\zeta$ are such that

\begin{equation}
\lim_{\delta\to0}\mathcal{L}_{\delta}u({x})=\Delta u(x),
\label{limit_for1D_operator}
\end{equation}
where $\Delta$ is the local Laplace operator. 
To illustrate the numerical convergence of the proposed method, we consider manufactured solutions, i.e. we choose analytical solutions, $u_0(x)$, to the local Poisson equation and compute the corresponding forcing term $b(x)$ and Dirichlet volume constraint $g(x)$. These are then used for the nonlocal Poisson problem (\ref{strong form_NLD.}). Specifically, we consider two cases: a sinusoidal and a linear solution (with the purpose of performing the so-called patch test). Therefore, for the first case we set $u_0(x)=\sin(2\pi x)$, for which $g(x)=\sin(2\pi x)$ and 

\begin{equation}
\begin{aligned}
    b(x) &= -\Delta u_0(x)= -\Delta \sin{(2\pi x)}= 4\pi^2 \sin{(2\pi x)}.
\end{aligned}
\end{equation}
For the second case, instead, we have $u_0(x)=x$, $g(x)=x$ and 

\begin{equation}
\begin{aligned}
    b(x) &= 0.
\end{aligned}
\end{equation}

\subsubsection{Uniform discretizations}\label{numerical_1D_uniform}
We investigate the convergence behavior for uniform discretizations. The finite element mesh has a uniform discretization size, $h$, over $[-\delta,1+\delta]=\left([-\delta,0]\cup[1,\delta]\right)\cup(0,1)$. Recall that we consider cases for which $m=\delta/h\in\mathbb{N}$, meaning that elements of size $h$ subdivide $(0,1)$ and $[-\delta,1+\delta]$ exactly. The same applies when the domain extension $[-\delta-t_e,1+\delta+t_e]$, with $t_e=\delta$, is employed. For the outer quadrature, we consider $N_q=40$ Gauss points, while for inner quadrature we use $\overline{N}_{qp}=10$.

For the sinusoidal solution we use $\gamma_{1,r}$, $h=0.01$, and $m=2$, meaning that $\Omega$ is discretized using 100 elements and $\delta=0.02$. For the construction of the inner quadrature weights, we consider two cases: one without domain extension, i.e., $t_e=0$, and one with domain extension $t_e=\delta$. The obtained numerical solutions are reported in Figure \ref{d_m2_IF2_bc2bc3}, while Figure \ref{absError_m2_IF2_bc2bc3} shows the absolute error obtained for the two considered cases. We observe that, when $t_e=0$, the error concentrates near the boundary of the domain, whereas this does not occur for $t_e=\delta$. Next, we perform an $L^2$ norm convergence study by varying $h$ and $\delta$, with fixed ratio $m=2$. As shown in Figure \ref{1D_CG_U_sine_extension_noextension_convergence}, for $t_e=0$, we observe a linear convergence, whereas, for $t_e=\delta$, the rate is quadratic. This suggests that the concentration of error near the boundary observed for $t_e=0$ reduces the overall convergence rate. Therefore, from now on, we only employ $t_e=\delta$ in the construction of the inner quadrature weights.

\begin{figure}[H] 
\begin{center}
\subfigure[Numerical and exact solutions]{\includegraphics[trim = 5mm 60mm 15mm 60mm, clip=true,width=0.49\textwidth]{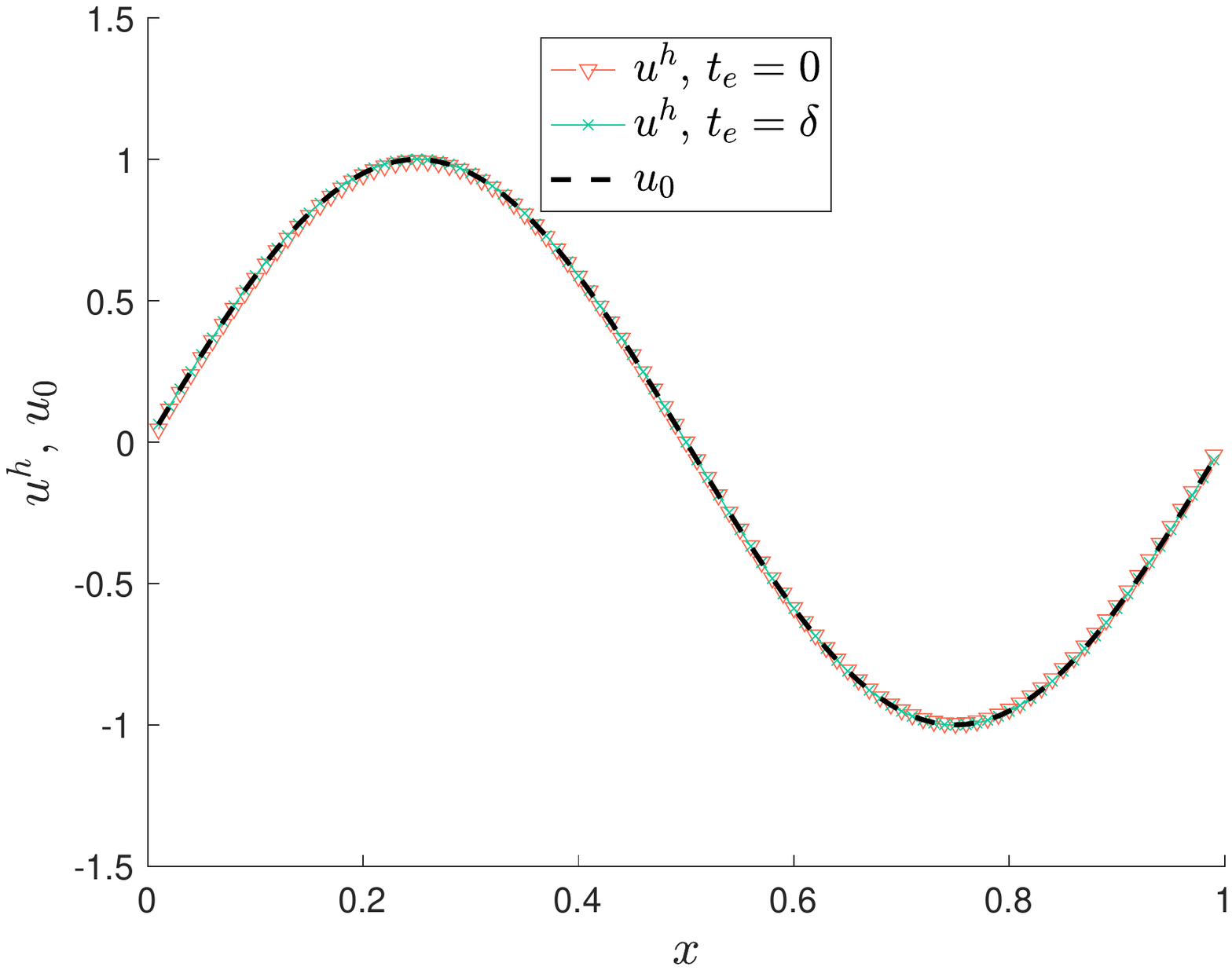}\label{d_m2_IF2_bc2bc3}} 
\subfigure[Absolute error]{\includegraphics[trim = 5mm 60mm 15mm 60mm, clip=true,width=0.49\textwidth]{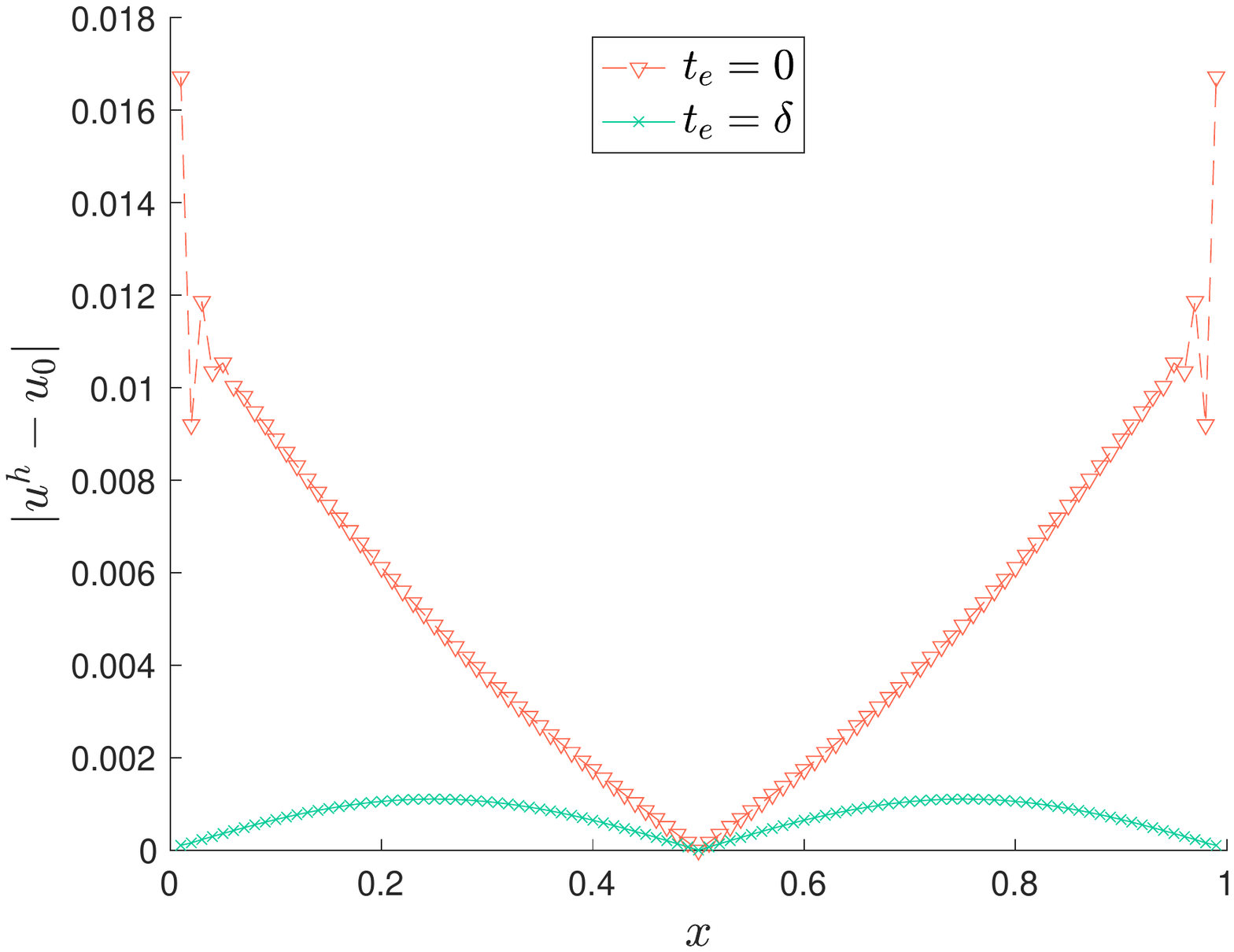}\label{absError_m2_IF2_bc2bc3}} 
\caption{Numerical solution and associated absolute error for the one-dimensional problem with sinusoidal solution for $\gamma_{1,r}$, $m=2$, ${N}_{q}=40$, and $\overline{N}_{qp}=10$. A uniform element size $h=0.01$, corresponding to 100 elements for the discretization of $\Omega$ is employed.}
\label{1D_CG_U_sine_extension_noextension_solutionanderror}
\end{center}
\end{figure}

\begin{figure}[H] 
\begin{center}
\subfigure[$\gamma_{1,c}$]{\includegraphics[trim = 15mm 60mm 15mm 60mm, clip=true,width=0.49\textwidth]{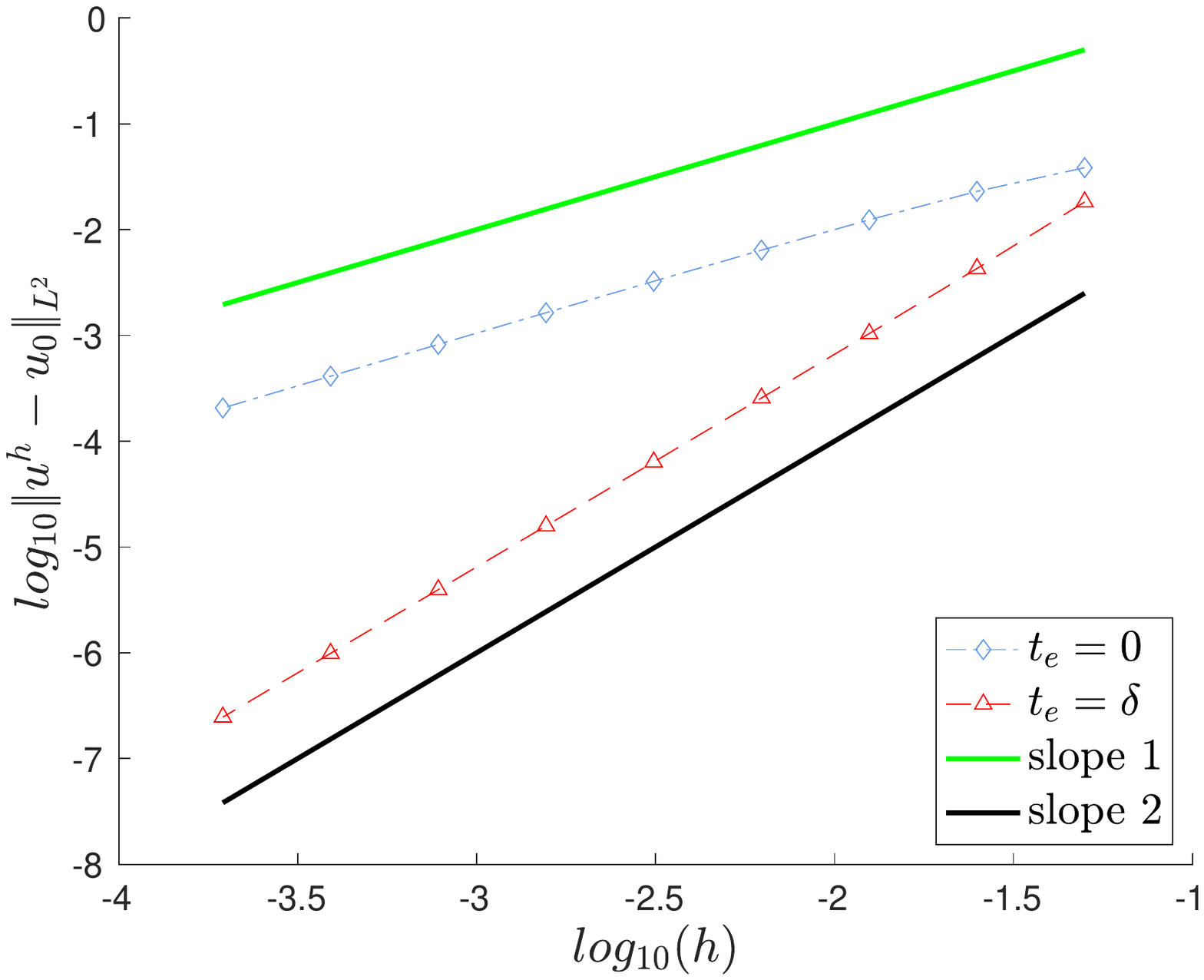}\label{1D_CG_U_OuterGauss_InnerMethod2_sine_f3_bx1_1del_2del_m2_IF3_No40_Ni10}} 
\subfigure[$\gamma_{1,r}$]{\includegraphics[trim = 15mm 60mm 15mm 60mm, clip=true,width=0.49\textwidth]{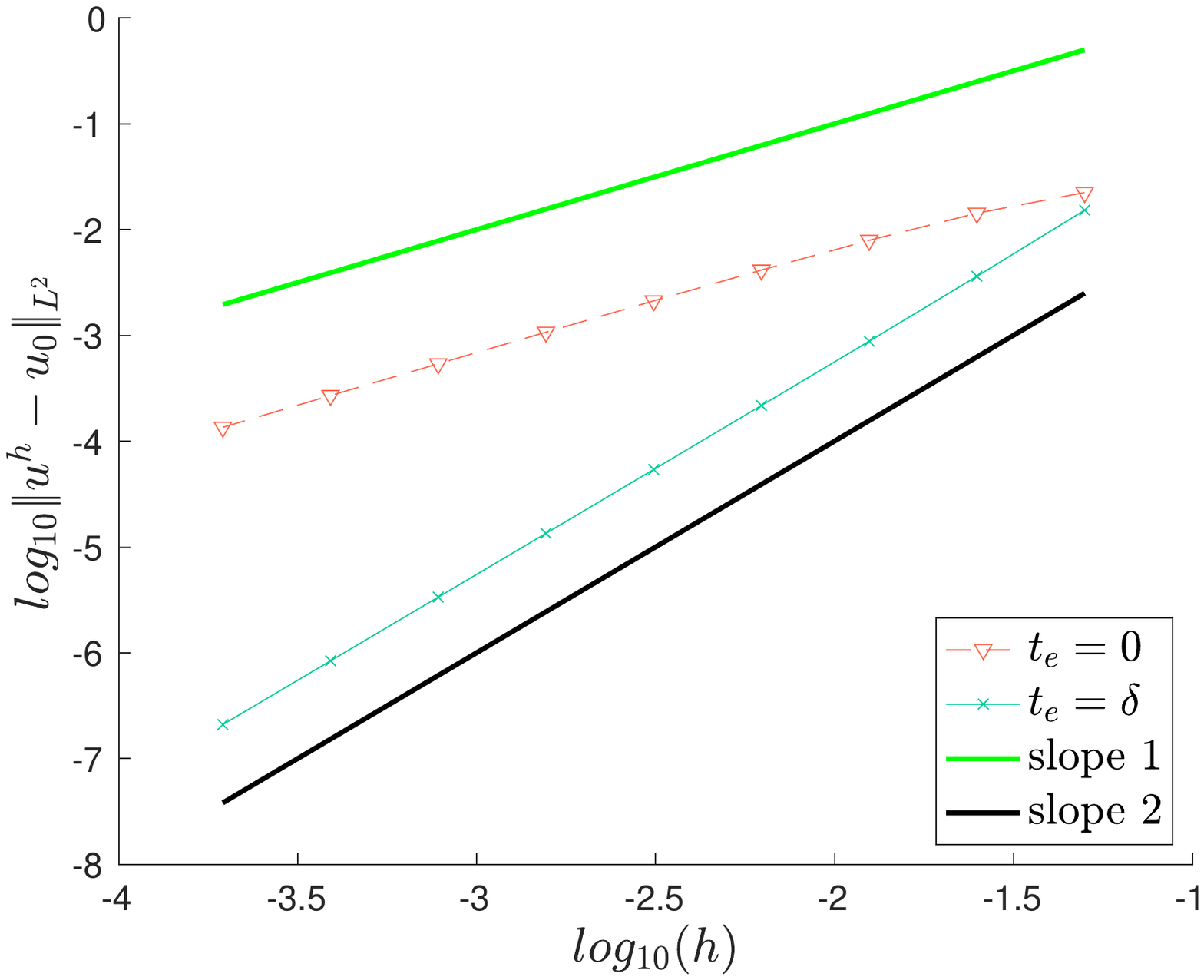}\label{1D_CG_U_OuterGauss_InnerMethod2_sine_f3_bx1_1del_2del_m2_IF2_No40_Ni10}} 
\caption{$L^2$ norm convergence behaviors of the one-dimensional numerical solutions  for the case with sinusoidal solution. $m=2$, uniform discretization, and $t_e=0$ and $t_e=\delta$. ${N}_{q}=40$ and $\overline{N}_{qp}=10$ are employed. $\gamma_{1,c}$ and $\gamma_{1,r}$ are both considered.}
\label{1D_CG_U_sine_extension_noextension_convergence}
\end{center}
\end{figure}

\noindent
Figures \ref{1D_CG_U_oGauss_iGMLS_sine_f3_bx1_2del_m123_IF3_No40_Ni10} and \ref{1D_CG_U_oGauss_iGMLS_sine_f3_bx1_2del_m123_IF2_No40_Ni10} show the $L^2$ norm convergence behavior for $\gamma_{1,c}$ and $\gamma_{1,r}$, respectively. For both cases, we employ $t_e=\delta$. ${N}_{q}=40$, and $\overline{N}_{qp}=10$. For all considered values of $m$ (i.e., $m=1,2,3$), we observe a second-order convergence rate in the $L^2$ norm. The convergence behavior in the $H^1$ norm is presented in Figure \ref{1D_CG_U_sine_m123_H1}. For all of the considered cases, a first-order convergence is obtained, which is consistent with the theoretical prediction from Section \ref{sec:convergence}. It can also be noted that convergence in the $H^1$ norm is one rate lower than in $L^2$.

\begin{figure}[H] 
\begin{center}
\subfigure[$\gamma_{1,c}$]{\includegraphics[trim = 15mm 60mm 15mm 60mm, clip=true,width=0.49\textwidth]{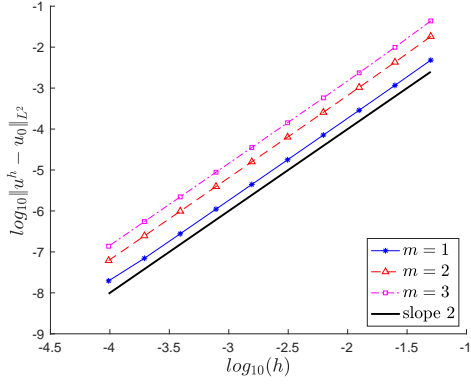}\label{1D_CG_U_oGauss_iGMLS_sine_f3_bx1_2del_m123_IF3_No40_Ni10}} 
\subfigure[$\gamma_{1,r}$]{\includegraphics[trim = 15mm 60mm 15mm 60mm, clip=true,width=0.49\textwidth]{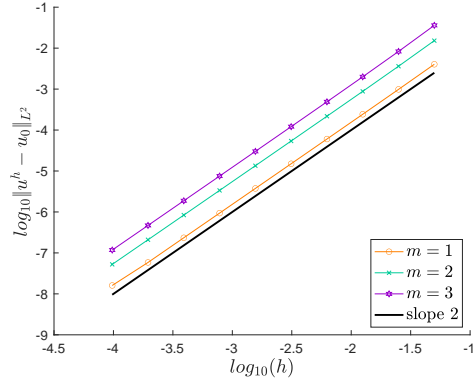}\label{1D_CG_U_oGauss_iGMLS_sine_f3_bx1_2del_m123_IF2_No40_Ni10}} 
\caption{$L^2$ norm convergence behaviors of the one-dimensional numerical solutions for the case with sinusoidal solution. $m=1,2,3$, uniform discretization, and $t_e=\delta$. ${N}_{q}=40$ and $\overline{N}_{qp}=10$ are employed. $\gamma_{1,c}$ and $\gamma_{1,r}$ are both considered.}
\label{1D_CG_U_sine_m123}
\end{center}
\end{figure}

\begin{figure}[H] 
\begin{center}
\subfigure[$\gamma_{1,c}$]{\includegraphics[trim = 15mm 60mm 15mm 60mm, clip=true,width=0.49\textwidth]{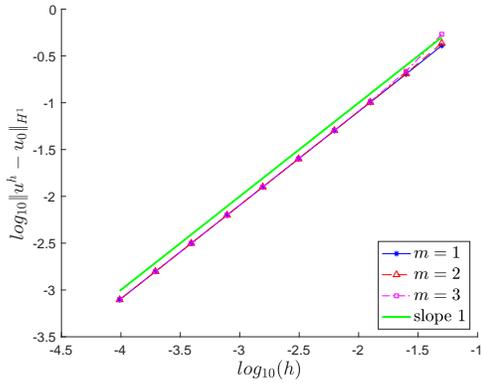}\label{1D_CG_U_oGauss_iGMLS_sine_f3_bx1_2del_m123_IF3_No40_Ni10_H1}} 
\subfigure[$\gamma_{1,r}$]{\includegraphics[trim = 15mm 60mm 15mm 60mm, clip=true,width=0.49\textwidth]{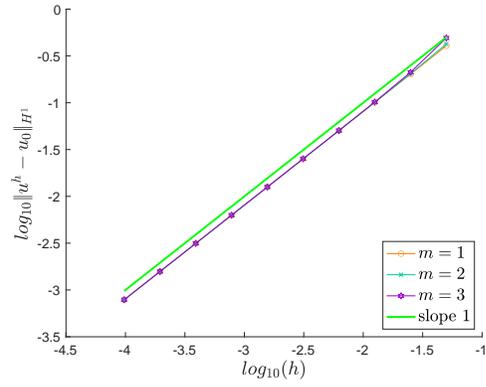}\label{1D_CG_U_oGauss_iGMLS_sine_f3_bx1_2del_m123_IF2_No40_Ni10_H1}} 
\caption{$H^1$ norm convergence behaviors of the one-dimensional numerical solutions for the case with sinusoidal solution. $m=1,2,3$, uniform discretization, and $t_e=\delta$. ${N}_{q}=40$ and $\overline{N}_{qp}=10$ are employed. $\gamma_{1,c}$ and $\gamma_{1,r}$ are both considered.}
\label{1D_CG_U_sine_m123_H1}
\end{center}
\end{figure}

Next, we consider the case with a linear solution. For the reasons illustrated above, we consider $t_e=\delta$; we set $h=0.01$ and $m=2$, meaning that $\Omega$ is discretized using 100 elements and $\delta=0.02$. The $L^2$ norms of the error for the cases with $\gamma_{1,r}$ and $\gamma_{1,c}$ are $1.59$E$-13$ and $6.96$E$-14$, respectively. This fact implies that the proposed approach passes the patch test for uniform discretizations, i.e. the numerical solution is accurate up to machine precision for linear solutions. This is expected since the exact, local solution belongs to the discretization space $\mathcal V^h$.

\subsubsection{Nonuniform discretizations}\label{numerical_1D_nonuniform}

Next, we investigate the performance of the proposed method for non-uniform discretizations. The non-uniform discretizations are constructed by perturbing uniform discretizations of size $h$. This is achieved by moving each finite element node in $(0,1)$ and $(-\delta,0)\cup(1,1+\delta)$ from their original position $x^{u}$ to a new randomly selected position $x^{nu}=x^u+\epsilon h R_a$, where $\epsilon$ is a chosen perturbation factor and $R_a$ is a random number in $\left[-1,1\right]$.

As for uniform discretizations, we first consider the sinusoidal solution. We select $t_e=\delta$, ${N}_{q}=40$, and $\overline{N}_{qp}=10$. Figures \ref{1D_CG_NU_sine_m23} and \ref{1D_CG_NU_sine_m23_H1} show the convergence behavior for $m=2,3$ for both $\gamma_{1,c}$ and $\gamma_{1,r}$, in the $L^2$ and $H^1$ norms, respectively. We observe an apparent second-order convergence rate in the $L^2$, and first-order for the $H^1$ norm, i.e., one rate lower. However, it should be noted that Figure \ref{1D_CG_NU_sine_m23_H1} shows a reduction in the $H^1$ convergence rate for the finer cases, suggesting that, asymptotically, the convergence rate may reach a zeroth-order convergence, as discussed in Remark \ref{remark_rate}.

\begin{figure}[H] 
\begin{center}
\subfigure[$\gamma_{1,c}$]{\includegraphics[trim = 15mm 60mm 15mm 60mm, clip=true,width=0.49\textwidth]{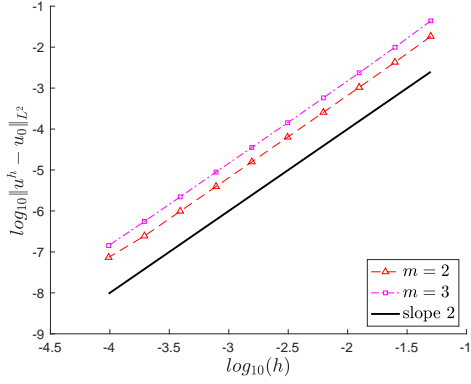}\label{1D_CG_NU_OuterGauss_InnerMethod2_sine_f3_bx1_2del_m2_m3_IF3_No40_Ni10_unih}} 
\subfigure[$\gamma_{1,r}$]{\includegraphics[trim = 15mm 60mm 15mm 60mm, clip=true,width=0.49\textwidth]{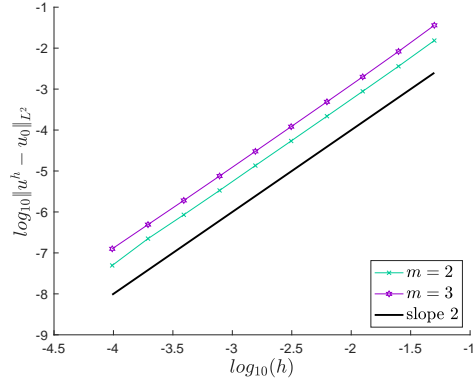}\label{1D_CG_NU_OuterGauss_InnerMethod2_sine_f3_bx1_2del_m2_m3_IF2_No40_Ni10_unih}} 
\caption{$L^2$ convergence behaviors of the one-dimensional numerical solutions for the case with sinusoidal solution. $m=2,3$, non-uniform discretization with $\epsilon=0.1$, and $t_e=\delta$. ${N}_{q}=40$ and $\overline{N}_{qp}=10$ are employed. $\gamma_{1,c}$ and $\gamma_{1,r}$ are both considered.}
\label{1D_CG_NU_sine_m23}
\end{center}
\end{figure}

\begin{figure}[H] 
\begin{center}
\subfigure[$\gamma_{1,c}$]{\includegraphics[trim = 15mm 60mm 15mm 60mm, clip=true,width=0.49\textwidth]{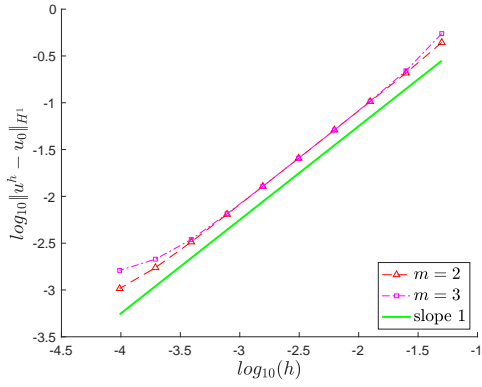}\label{1D_CG_NU_OuterGauss_InnerMethod2_sine_f3_bx1_2del_m2_m3_IF3_No40_Ni10_unih_H1}} 
\subfigure[$\gamma_{1,r}$]{\includegraphics[trim = 15mm 60mm 15mm 60mm, clip=true,width=0.49\textwidth]{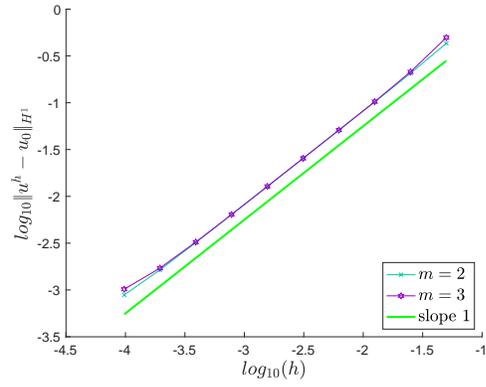}\label{1D_CG_NU_OuterGauss_InnerMethod2_sine_f3_bx1_2del_m2_m3_IF2_No40_Ni10_unih_H1}} 
\caption{$H^1$ norm convergence behaviors of the one-dimensional numerical solutions for the case with sinusoidal solution. $m=2,3$, non-uniform discretization with $\epsilon=0.1$, and $t_e=\delta$. ${N}_{q}=40$ and $\overline{N}_{qp}=10$ are employed. $\gamma_{1,c}$ and $\gamma_{1,r}$ are both considered.}
\label{1D_CG_NU_sine_m23_H1}
\end{center}
\end{figure}

We then consider the linear solution. As before, we take $t_e=\delta$, ${N}_{q}=40$, and $\overline{N}_{qp}=10$.
In contrast to the uniform case, for non-uniform discretizations, the proposed method does not pass the patch test. As shown in Figures \ref{1D_CG_NU_linear_m23} and \ref{1D_CG_NU_linear_m23_H1}, which report convergence behavior in the $L^2$ and $H^1$ norms, respectively, for $m=2,3$ for $\gamma_{1,c}$ and $\gamma_{1,r}$, the method shows a first-order $L^2$ norm convergence and a zeroth-order $H^1$ norm convergence (see Remark \ref{remark_rate}). By comparing Figure \ref{1D_CG_NU_sine_m23} and Figure \ref{1D_CG_NU_linear_m23}, it can be noted that the magnitude of the $L^2$ norm errors obtained for the case with a linear solution is much smaller compared with the magnitude obtained for the problem with a sinusoidal solution. 
This confirms our conjecture that the method has a first-order asymptotic convergence in the $L^2$ norm, and that second-order convergence is observed in the pre-asymptotic regime.  

A natural question is whether further refinement of the sinusoidal case would show a reduction in convergence rate in the $L^2$ norm; we note that attempting to refine the sinusoidal case further, the error becomes dominated by floating point arithmetic.  
Nonetheless, since, as shown in Figure \ref{1D_CG_NU_sine_m23_H1}, the convergence rate in the $H^1$ norm starts to reduce for the finer refinements, and we have observed one-order lower convergence in the $L^2$ norm, it is reasonable to expect that the convergence rate in the $L^2$ norm would reduce with further refinement.

\begin{figure}[H] 
\begin{center}
\subfigure[$\gamma_{1,c}$]{\includegraphics[trim = 15mm 60mm 15mm 60mm, clip=true,width=0.49\textwidth]{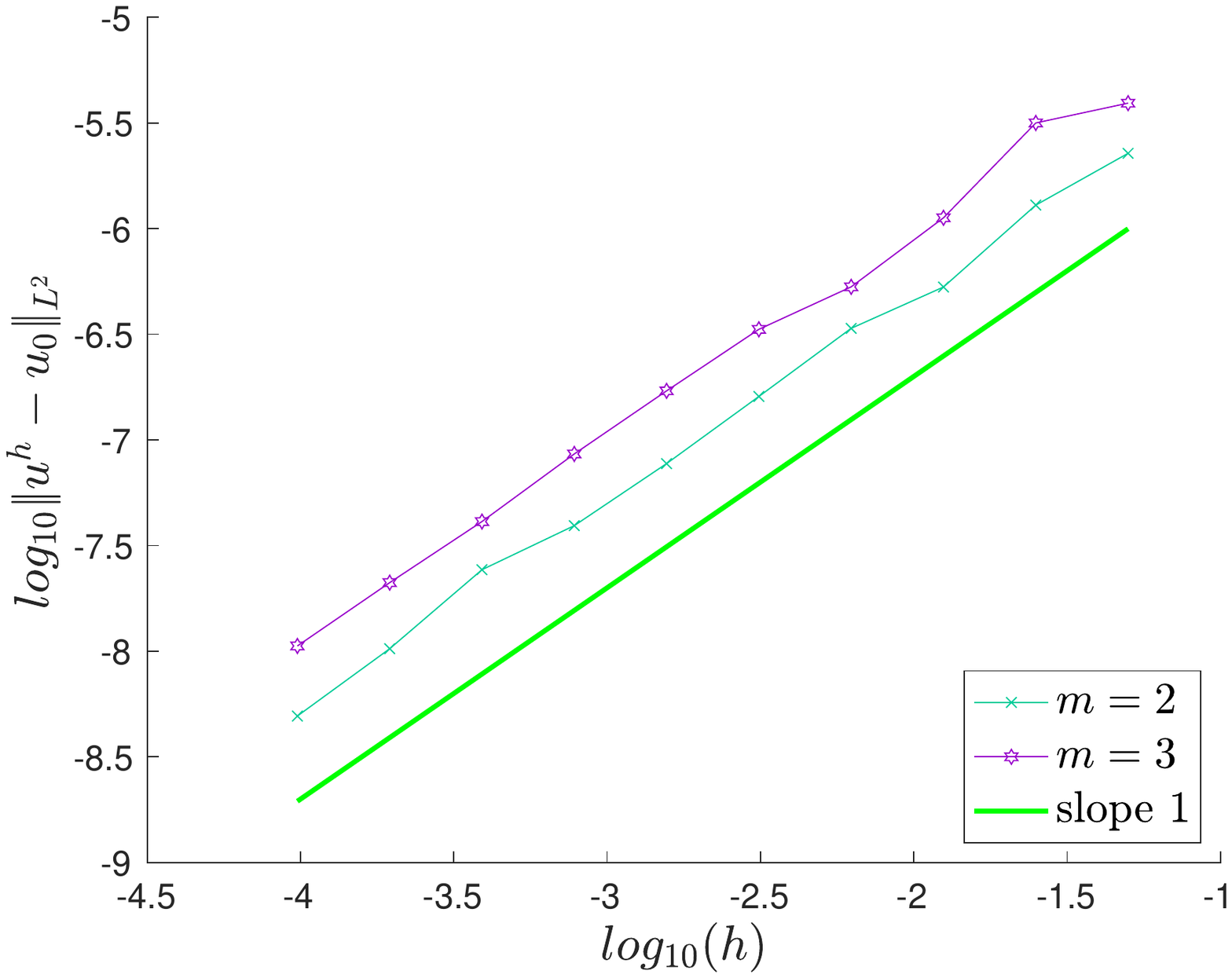}\label{1D_CG_NU_oGauss_iGMLS_linear_f3_bx1_2del_m2_m3_IF3_No40_Ni10}} 
\subfigure[$\gamma_{1,r}$]{\includegraphics[trim = 15mm 60mm 15mm 60mm, clip=true,width=0.49\textwidth]{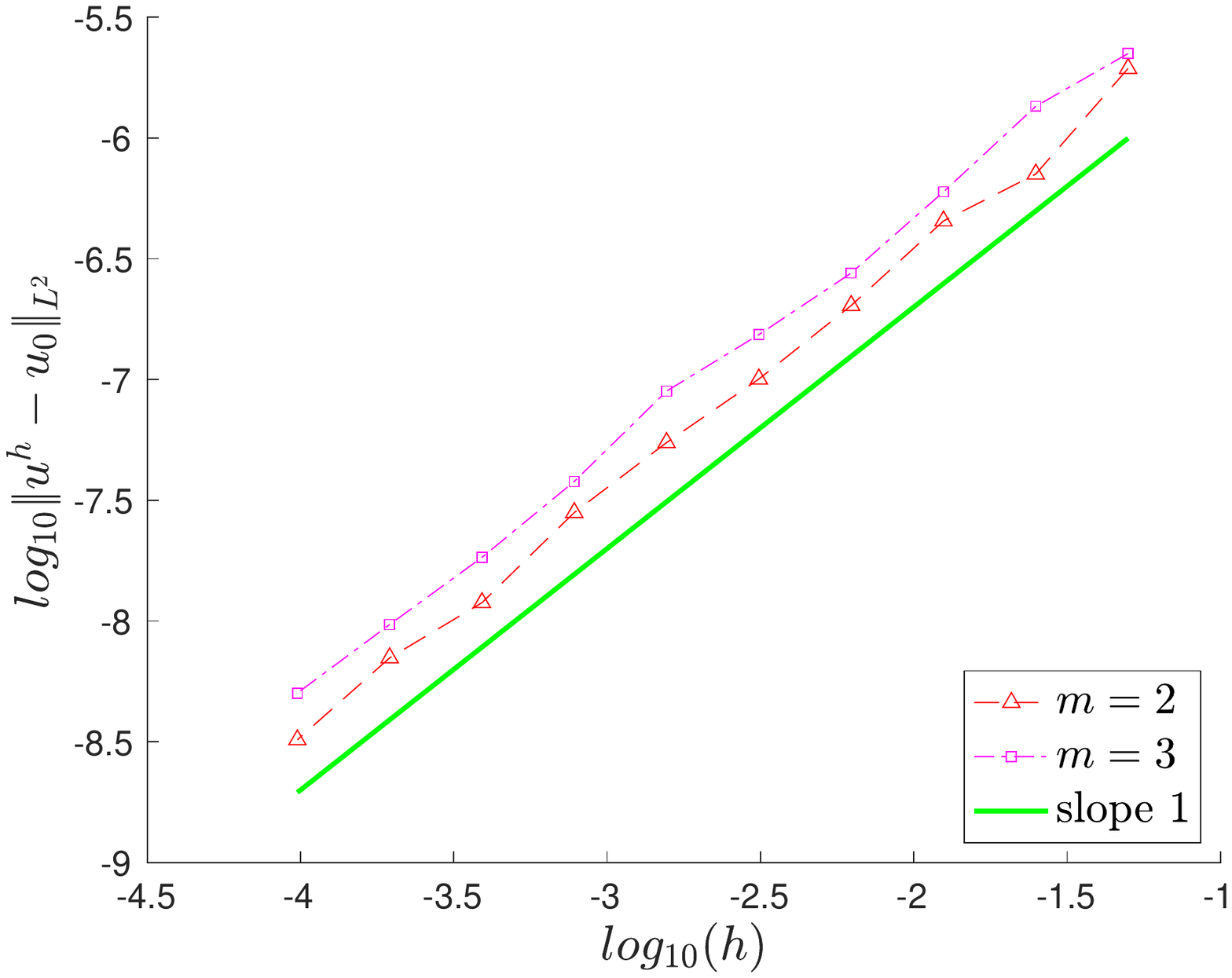}\label{1D_CG_NU_oGauss_iGMLS_linear_f3_bx1_2del_m2_m3_IF2_No40_Ni10}} 
\caption{$L^2$ norm convergence behaviors of the one-dimensional numerical solutions for the case with linear solution. $m=2,3$, non-uniform discretization with $\epsilon=0.1$, and $t_e=\delta$. ${N}_{q}=40$ and $\overline{N}_{qp}=10$ are employed. $\gamma_{1,c}$ and $\gamma_{1,r}$ are both considered.}
\label{1D_CG_NU_linear_m23}
\end{center}
\end{figure}

\begin{figure}[H] 
\begin{center}
\subfigure[$\gamma_{1,c}$]{\includegraphics[trim = 15mm 60mm 15mm 60mm, clip=true,width=0.49\textwidth]{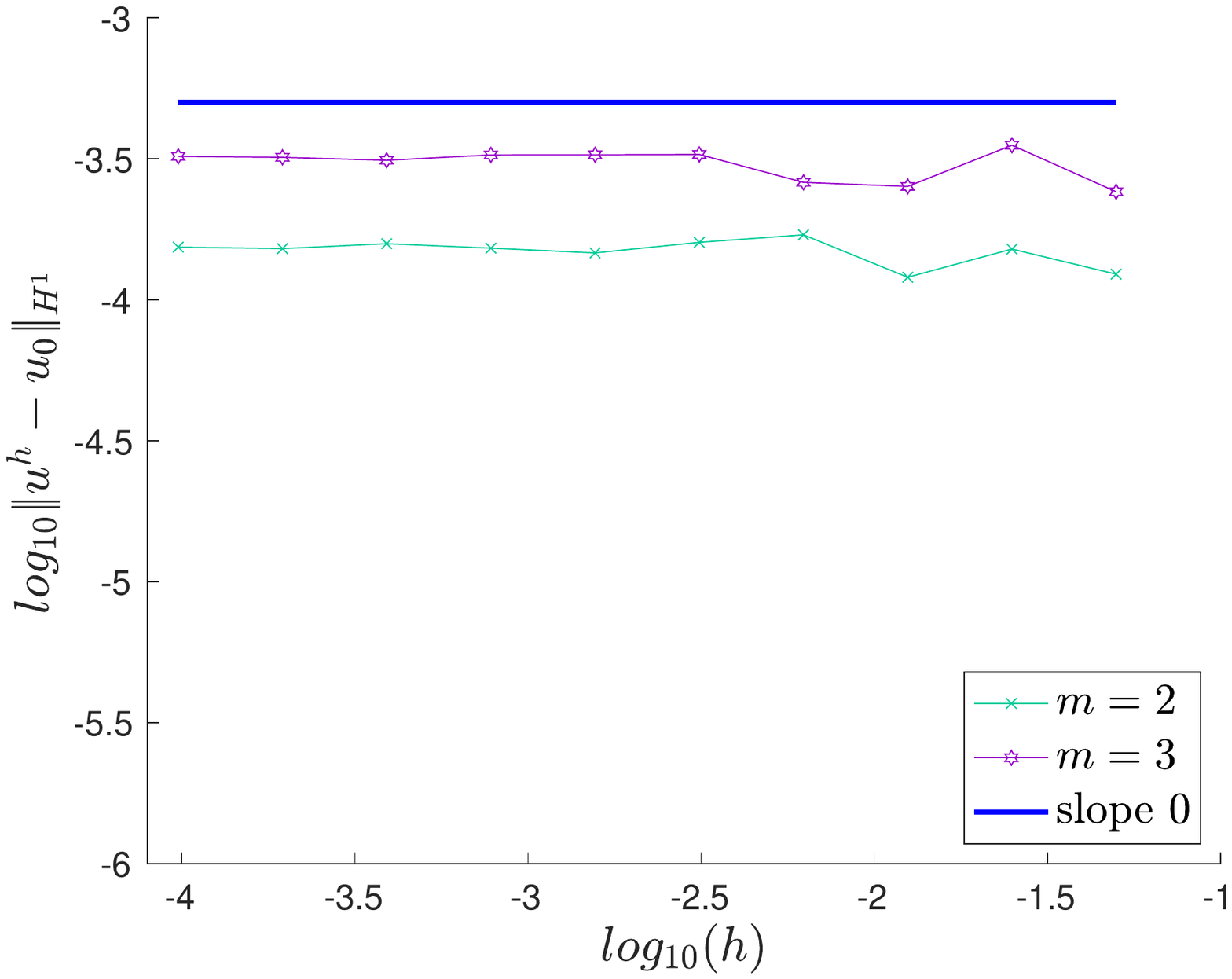}\label{1D_CG_NU_oGauss_iGMLS_linear_f3_bx1_2del_m2_m3_IF3_No40_Ni10_H1}} 
\subfigure[$\gamma_{1,r}$]{\includegraphics[trim = 15mm 60mm 15mm 60mm, clip=true,width=0.49\textwidth]{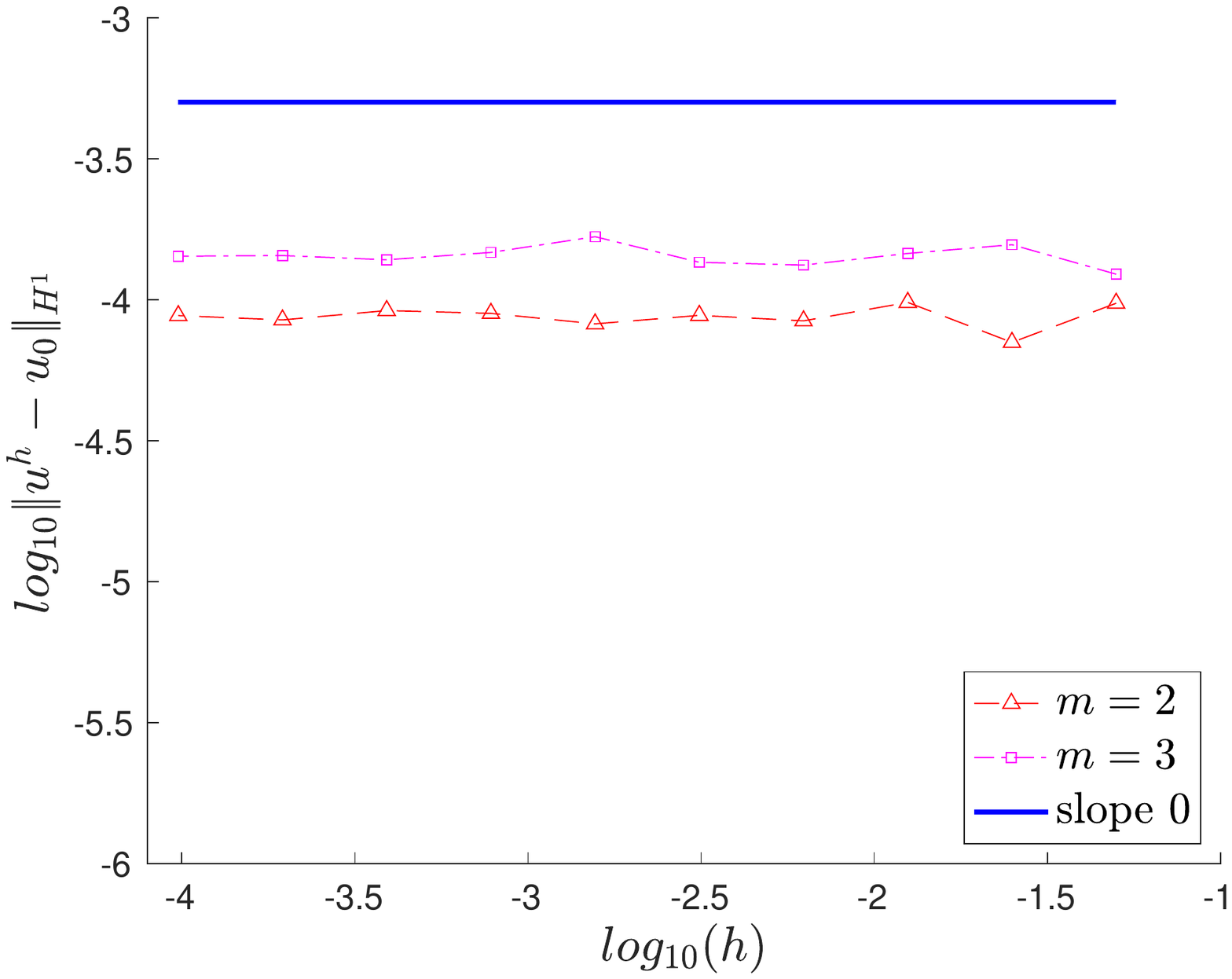}\label{1D_CG_NU_oGauss_iGMLS_linear_f3_bx1_2del_m2_m3_IF2_No40_Ni10_H1}} 
\caption{$H^1$ norm convergence behaviors of the one-dimensional numerical solutions for the case with linear solution. $m=2,3$, non-uniform discretization with $\epsilon=0.1$, and $t_e=\delta$. ${N}_{q}=40$ and $\overline{N}_{qp}=10$ are employed. $\gamma_{1,c}$ and $\gamma_{1,r}$ are both considered.}
\label{1D_CG_NU_linear_m23_H1}
\end{center}
\end{figure}

\subsection{Two-dimensional test cases}

\begin{figure} [H]
\centering
\vspace{0pt}  
\includegraphics[trim = 75mm 95mm 75mm 60mm, clip=true,width=0.9\textwidth]{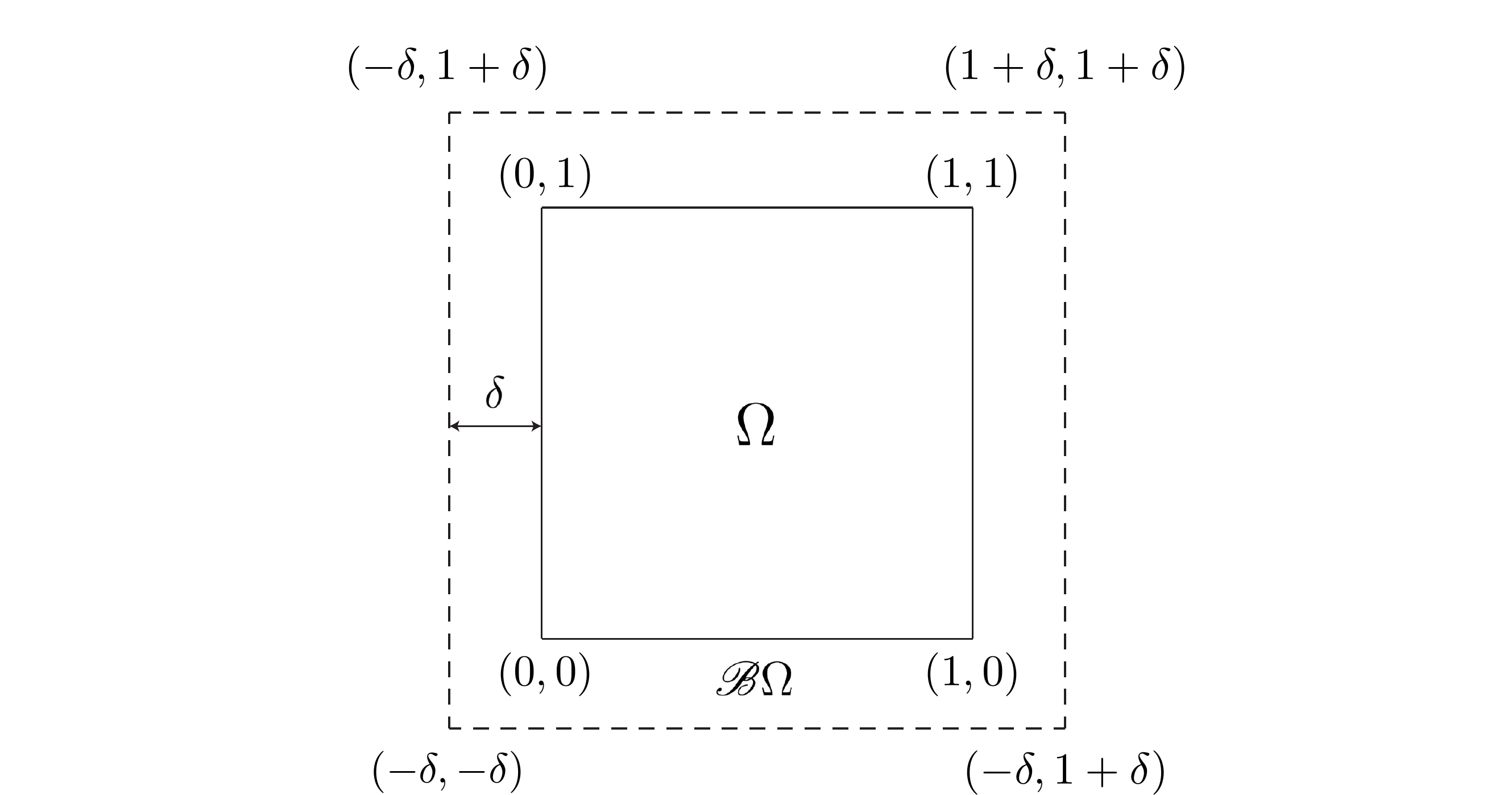}
\caption{Two-dimensional domain $\Omega$, with associated boundary layer $\mathscr{B}\Omega$}
\label{PD_2Ddomain_figure}
\end{figure}

We consider the two-dimensional domain $\Omega=(0,1)\times(0,1)$ with associated interaction domain $\mathscr{B}\Omega=\left([-\delta,1+\delta]\times[-\delta,1+\delta]\right)\setminus\Omega$. 
As in the previous section, this guarantees that in the convergence studies the inner solution domain $\Omega$ remains consistent during the refinement ($\delta\rightarrow 0$), so that the $L^2$ error norms are comparable for all $\delta$. We consider two kernel functions: a constant influence function

\begin{equation}
    \gamma_{2,c}(\mathbf{x},\mathbf{y}) = \left\{\begin{aligned}
         \ \frac{4}{\pi\delta^4} \quad\ &\rm{for}\ \lVert\mathbf{y}-\mathbf{x}\rVert\leq\delta,\\\
        \ 0 \ \ \quad\ &\rm{for}\ \lVert\mathbf{y}-\mathbf{x}\rVert>\delta,\\
\end{aligned}\right.   
\label{kernel_form1_2D}
\end{equation}
and a rational one

\begin{equation}
    \gamma_{2,r}(\mathbf{x},\mathbf{y}) =\left\{\begin{aligned}
         \ \frac{3}{\pi\delta^3 \lVert\mathbf{y}-\mathbf{x}\rVert} \quad \ &\rm{for}\ \lVert\mathbf{y}-\mathbf{x}\rVert\leq\delta,\\\
        \ 0 \ \ \quad\ &\rm{for}\ \lVert\mathbf{y}-\mathbf{x}\rVert>\delta,\\
\end{aligned}\right.  
\label{kernel_form2_2D}
\end{equation}
which correspond to the expressions in Eqs. (\ref{kernel_form1}) and (\ref{kernel_form2}) for $\zeta=4/\pi$ and $\zeta=3/\pi$, respectively. As for the one-dimensional case, these values of $\zeta$ are such that

\begin{equation}
\lim_{\delta\to0}\mathcal{L}_{\delta}u(\mathbf{x})=\Delta u(\mathbf{x}).
\label{limit_for2D_operator}
\end{equation}

As discussed in 
Section \ref{sec:Nonlocal_diff_model}, the supports of the kernels presented in (\ref{kernel_form1_2D}) and (\ref{kernel_form2_2D}) correspond to circular Euclidean $\ell^2$ balls. However, kernels associated with $\ell^\infty$ balls (i.e., square supports) were also investigated and similar results as the ones presented in this section for Euclidean balls were obtained.
As before, we employ the method of manufactured solutions.
We select $u_0(\mathbf{x})=\sin(2\pi x_1)\sin(2\pi x_2)$, where $\mathbf{x}=(x_1,x_2)$, which corresponds to $g(\mathbf{x})=\sin(2\pi x_1)\sin(2\pi x_2)$ and to the following source term:

\begin{equation}
\begin{aligned}
    b(\mathbf{x}) &= -\Delta u_0(\mathbf{x})\\
    &= -\Delta\left( \sin(2\pi x_1)\sin(2\pi x_2)\right)\\
    &= 8\pi^2\sin(2\pi x_1)\sin(2\pi x_2).
\end{aligned}
\end{equation}

\subsubsection{Uniform discretizations}
As for the one-dimensional case, we first investigate the convergence behavior for uniform discretizations. The two-dimensional uniform mesh is constructed as a tensor product $\mathcal{M}^{h,u}_{2}=\mathcal{M}^{h,u}_{x_2}\times\mathcal{M}^{h,u}_{x_1}$, where $\mathcal{M}^{h,u}_{x_1}$ and $\mathcal{M}^{h,u}_{x_2}$ are one-dimensional uniform meshes of size $h$ over $\left[-\delta,0\right]\cup\left(0,1\right)\cup\left[1,1+\delta\right]$. For the convergence study, we set $t_e=\delta$, ${\overline{N}_{qp}=64}$, and use a four by four Gauss quadrature rule for the outer integral (${N}_{q}=16$). Figures \ref{2D_CG_U_OutGauss_IGMLS_sine_f3_bx1_2del_m2_IF5_No4_Ni8} and \ref{2D_CG_U_OutGauss_IGMLS_sine_f3_bx1_2del_m2_IF4_No4_Ni8} show the obtained results for $\gamma_{2,c}$ and $\gamma_{2,r}$, respectively. Up to the considered level of refinement, we observe a second-order convergence rate in the $L^2$ norm for both kernels.

\begin{figure}[H] 
\begin{center}
\subfigure[$\gamma_{2,c}$]{\includegraphics[trim = 15mm 75mm 15mm 80mm, clip=true,width=0.49\textwidth]{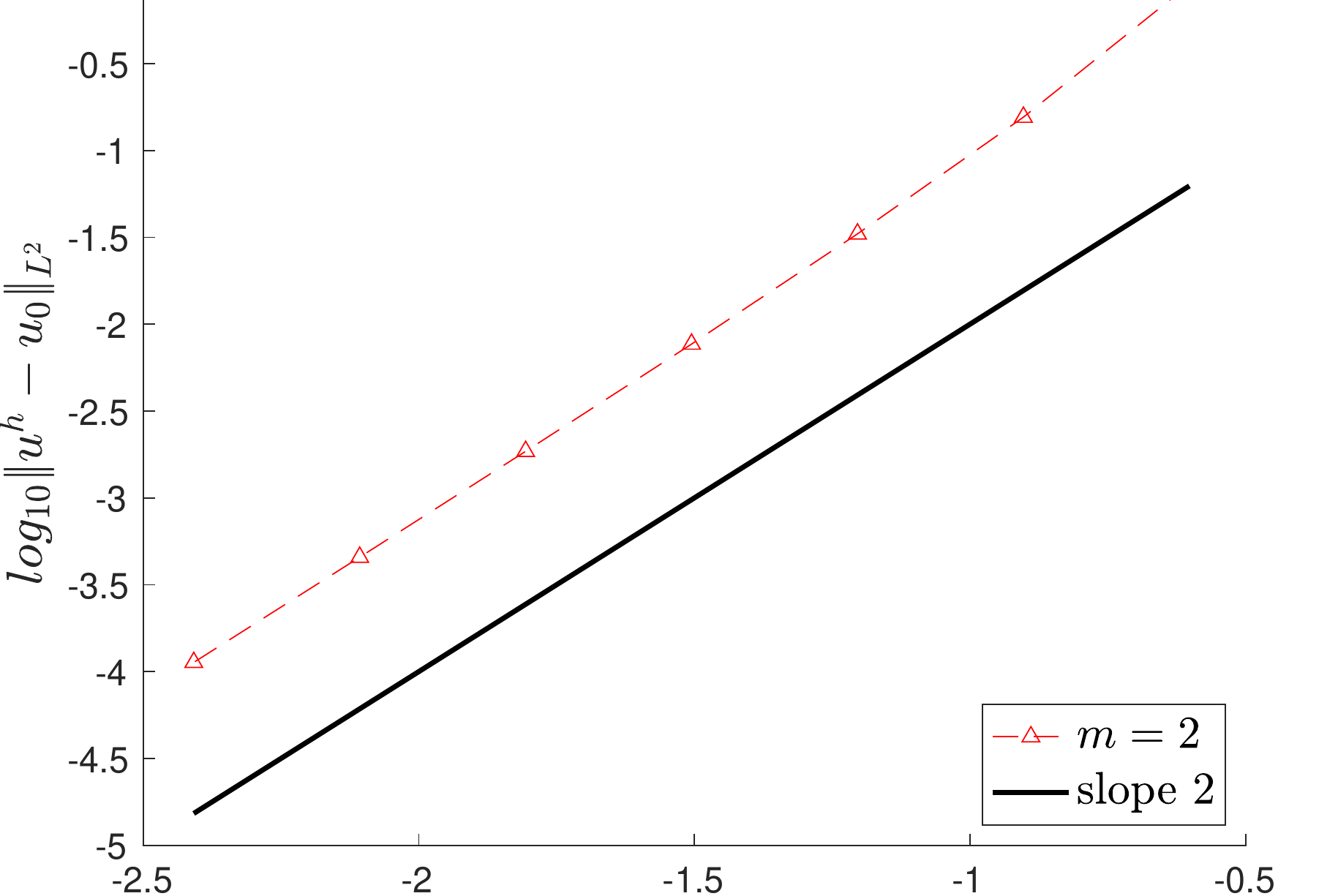}\label{2D_CG_U_OutGauss_IGMLS_sine_f3_bx1_2del_m2_IF5_No4_Ni8}} 
\subfigure[$\gamma_{2,r}$]{\includegraphics[trim = 15mm 75mm 15mm 80mm, clip=true,width=0.49\textwidth]{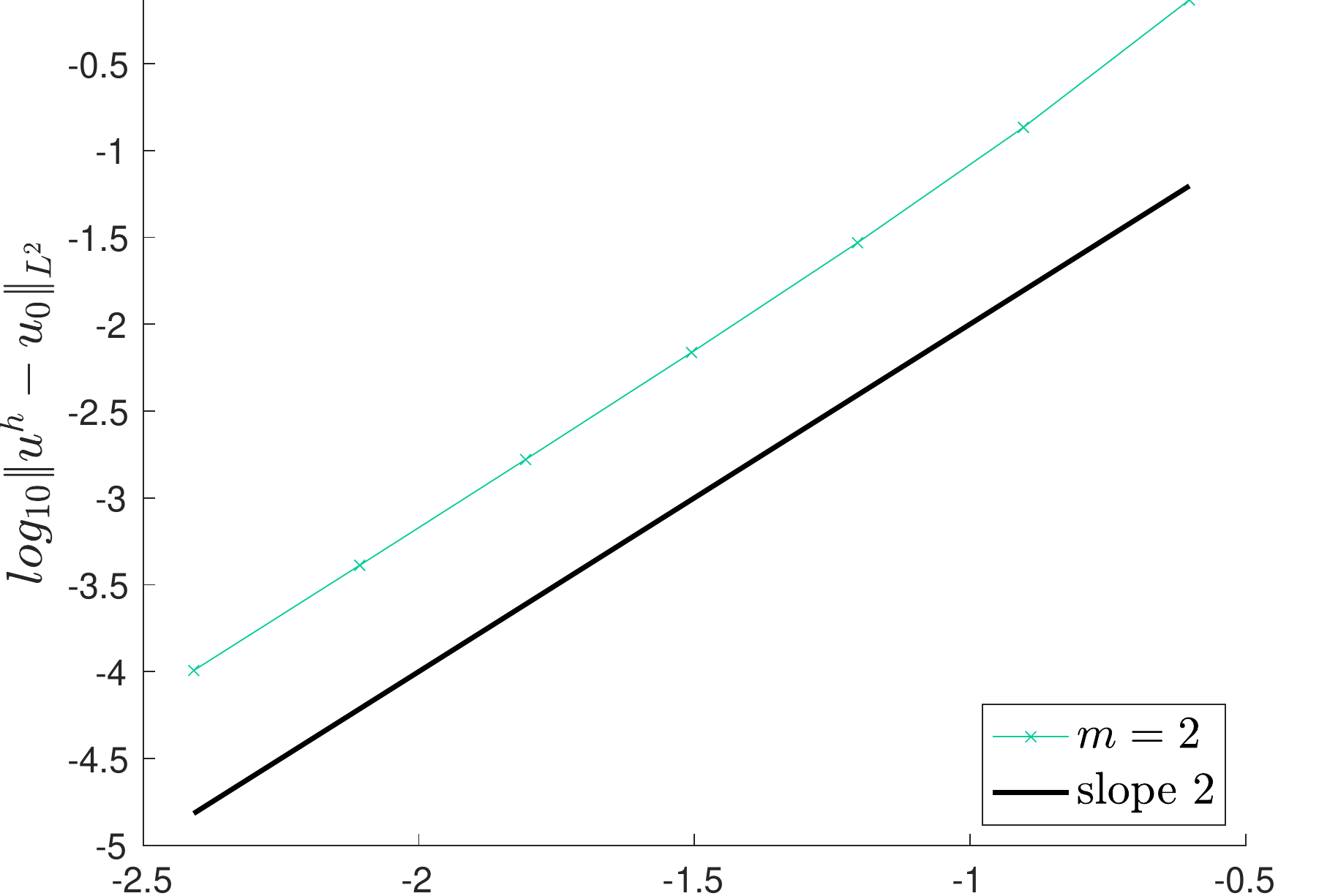}\label{2D_CG_U_OutGauss_IGMLS_sine_f3_bx1_2del_m2_IF4_No4_Ni8}} 
\caption{$L^2$ convergence behaviors of the two-dimensional numerical solutions for $m=2$, uniform discretization, and $t_e=\delta$. ${N}_{q}=16$ (as $4\times4$) and ${\overline{N}_{qp}=64}$. $\gamma_{2,c}$ and $\gamma_{2,r}$ are both considered.}
\label{2D_CG_U_OuterGauss_InnerGMLS_sine}
\end{center}
\end{figure}

\subsubsection{Nonuniform discretizations}
In this section, we investigate the performance of the proposed quadrature approach for two-dimensional non-uniform discretizations. We construct the two-dimensional non-uniform mesh as a tensor product of one-dimensional non-uniform discretization, i.e., $\mathcal{M}^{h,nu}_{2}=\mathcal{M}^{h,nu}_{x_1}\times\mathcal{M}^{h,nu}_{x_2}$, where $\mathcal{M}^{h,nu}_{x_1}$ and and $\mathcal{M}^{h,nu}_{x_2}$ are obtained by perturbing $\mathcal{M}^{h,u}_{x_1}$ and and $\mathcal{M}^{h,u}_{x_2}$, which are uniform meshes with spacing $h$ over $\left[-\delta,0\right]\cup\left(0,1\right)\cup\left[1,1+\delta\right]$. Similarly to the one-dimensional non-uniform case, the perturbation is achieved by moving the finite element nodes in $(0,1)$ and $(-\delta,0)\cup(1,1+\delta)$ from their original positions $x_1^u$ and $x_2^u$ to new randomly selected positions $x_1^{nu}=x_1^{u}+\epsilon h R_a$ and $x_2^{nu}=x_2^{u}+\epsilon h R_a$, where $\epsilon$ is a chosen perturbation factor and $R_a$ is a random number in $\left[-1,1\right]$.
For a visual example of $\mathcal{M}^{h,nu}_{2}$, see Figure \ref{Discretizations_schemes_2D}.

\begin{figure} [H]
\centering
\vspace{0pt}  
\includegraphics[trim = 550mm 265mm 75mm 250mm, clip=true,width=0.5\textwidth]{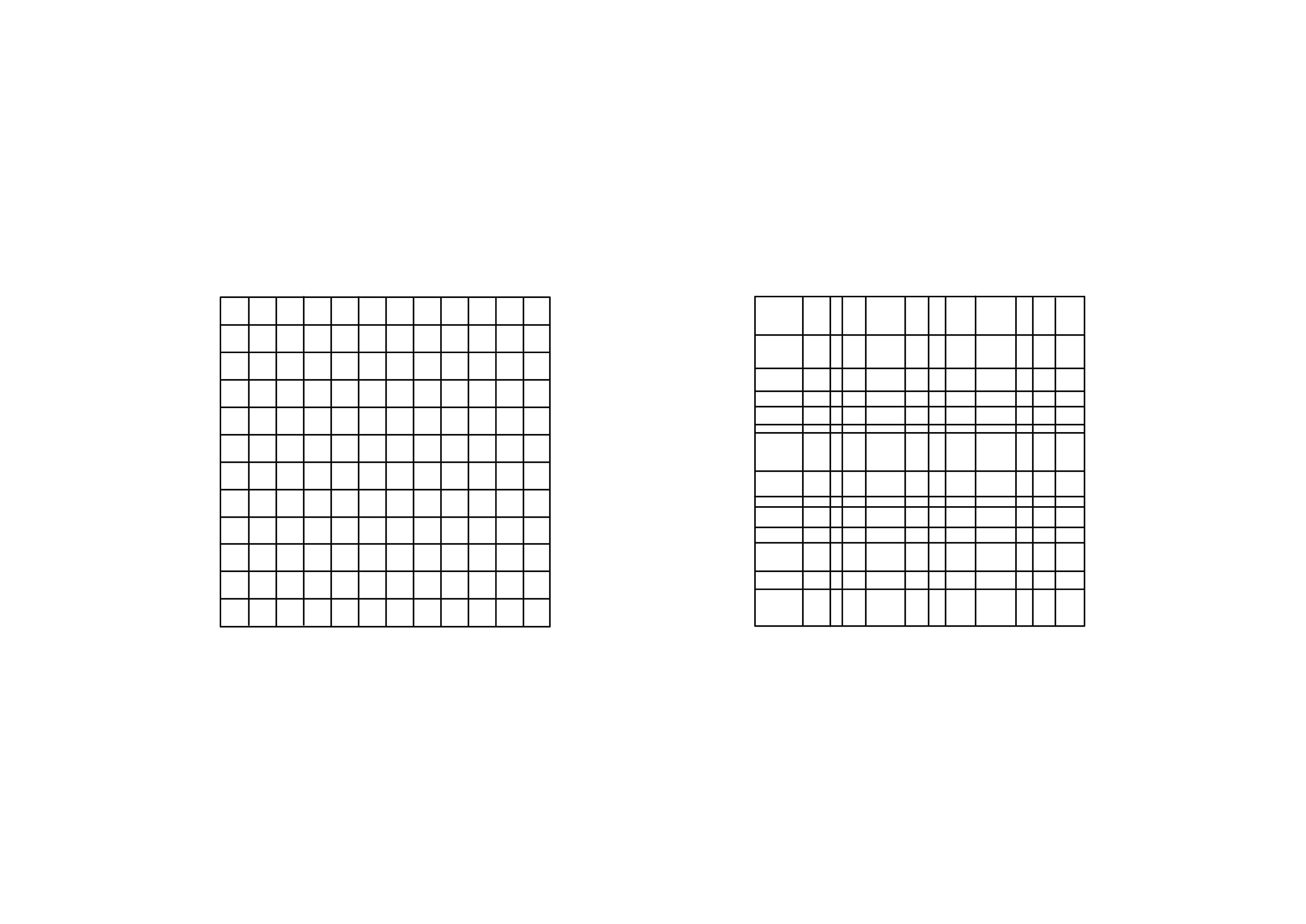}
\caption{Example of two-dimensional non-uniform mesh obtained as a tensor product of  perturbed one-dimensional meshes}
\label{Discretizations_schemes_2D}
\end{figure}

For the convergence studies we use $t_e=\delta$, ${\overline{N}_{qp}=64}$, and ${N}_{q}=16$ (four by four Gauss quadrature), with $\epsilon=0.1$. Figures \ref{2D_CG_NU_OutGauss_IGMLS_sine_f3_bx1_2del_m2_IF5_No4_Ni8} and \ref{2D_CG_NU_OutGauss_IGMLS_sine_f3_bx1_2del_m2_IF4_No4_Ni8} show the results for $\gamma_{2,c}$ and $\gamma_{2,r}$, respectively. Up to the considered level of refinement, we observe a second-order convergence rate in the $L^2$ norm for both kernels also for the non-uniform case.  As discussed in more detail for the one-dimensional nonuniform case in Section \ref{numerical_1D_nonuniform}, we conjecture that this rate is pre-asymptotic, and that the first-order asymptotic regime is difficult to observe in practice.

\begin{figure}[H] 
\begin{center}
\subfigure[$\gamma_{2,c}$]{\includegraphics[trim = 15mm 75mm 15mm 80mm, clip=true,width=0.49\textwidth]{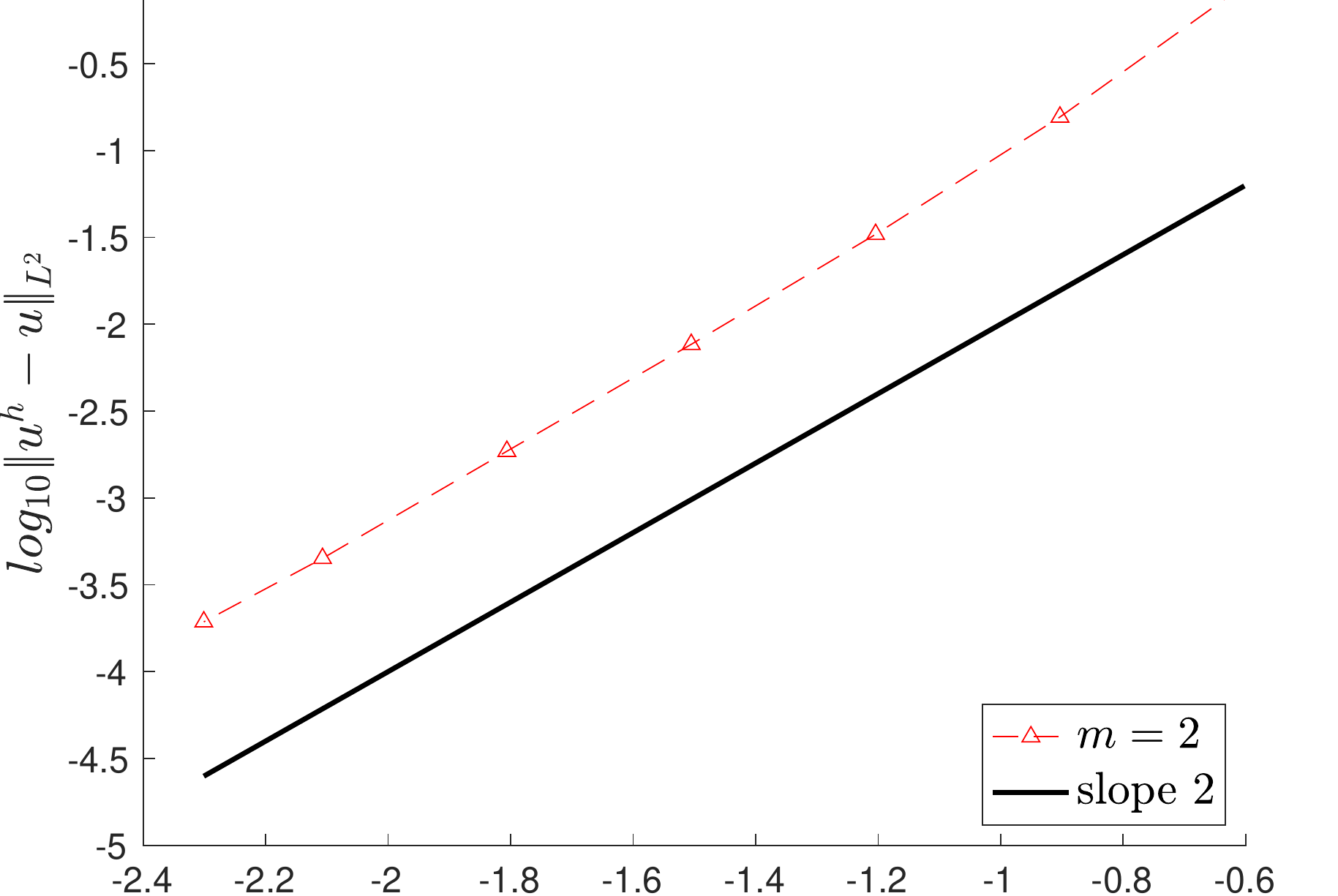}\label{2D_CG_NU_OutGauss_IGMLS_sine_f3_bx1_2del_m2_IF5_No4_Ni8}} 
\subfigure[$\gamma_{2,r}$]{\includegraphics[trim = 15mm 75mm 15mm 80mm, clip=true,width=0.49\textwidth]{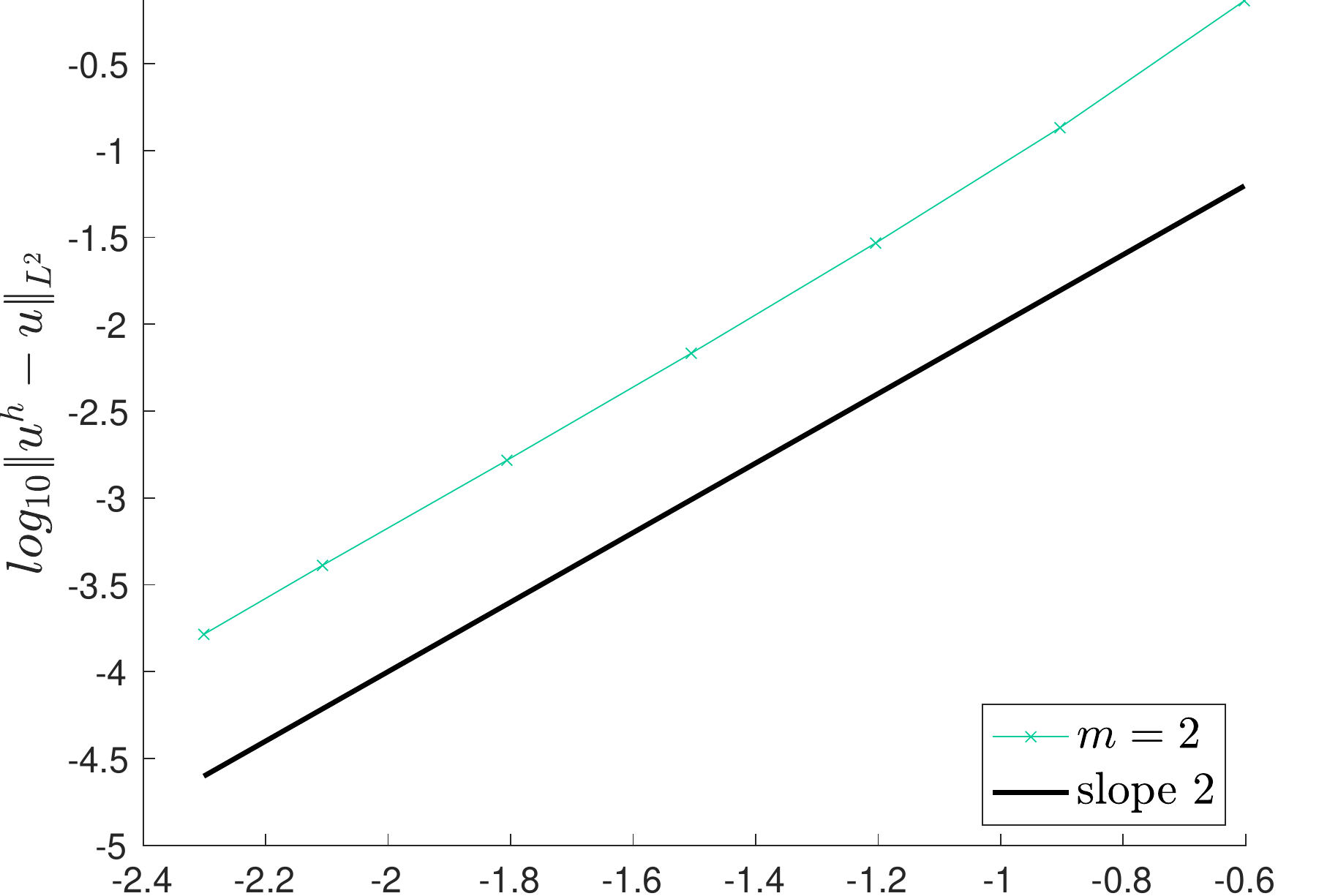}\label{2D_CG_NU_OutGauss_IGMLS_sine_f3_bx1_2del_m2_IF4_No4_Ni8}} 
\caption{$L^2$ convergence behaviors of the two-dimensional numerical solutions for $m=2$, non-uniform discretization with $\epsilon=0.1$, and $t_e=\delta$. ${N}_{q}=16$ (as $4\times4$) and ${\overline{N}_{qp}=64}$. $\gamma_{2,c}$ and $\gamma_{2,r}$ are both considered.}
\label{2D_CG_NU_OuterGauss_InnerGMLS_sine}
\end{center}
\end{figure}


\section{Conclusions}
\label{sec:conclusion}
We proposed a novel quadrature rule for the computation of integrals that arise in the matrix assembly of finite-element discretizations of nonlocal problems. In contrast to all previously employed methods, our technique does not require element-by-element integration, but relies on global integration over the nonlocal neighborhood. Specifically, we consider quadrature rules based on the generalized moving least squares method where the (global) quadrature weights are obtained by solving an equality-constrained optimization problem. The major advantage of this technique is the fact that the computation of element--ball intersections, a nontrivial and time consuming task, is avoided. Additionally, this technique requires minimal implementation effort, as it can be implemented in an existing finite element code. For this reason, we expect the proposed approach to become a building block of agile engineering codes. Our numerical experiments show that, when boundary conditions are treated carefully and the outer integral is computed accurately, our method is asymptotically compatible in the limit of $h\sim\delta\to 0$, featuring at least first-order convergence in $L^2$ for all dimensions and for both uniform and nonuniform grids. For piecewise linear finite-element implementations, in the case of uniform grids, our method features an optimal, second-order convergence rate in $L^2$ and passes the patch test.  For nonuniform grids, we see effective second-order convergence over a long pre-asyptotic regime, whereas the asymptotic first-order convergence is only evident in deviations from the patch test, which are very small relative to errors in more complicated solutions.  Convergence rates in $H^1$ are consistently one order lower than the $L^2$ rates.

We also carry out a preliminary numerical analysis of the method, but using the $H^1$ norm and restricted to the case of $h = \delta$ in one spatial dimension.  This analysis is consistent with the convergence rates observed in numerical experiments, but it does not account for the increase in convergence rate when measuring the $L^2$ norm instead of $H^1$.  As such, we believe that an interesting future direction for numerical analysis of this quadrature scheme would be to obtain sharp $L^2$ error estimates.  


\section*{Acknowledgements}
Kamensky and Pasetto were supported by start-up funding from the University of California San Diego. The work of Tian was partially supported by the National Science
Foundation grant DMS-2111608.
D'Elia and Trask are supported by the Sandia National Laboratories (SNL) Laboratory-directed Research and Development program and by the U.S. Department of Energy, Office of Advanced Scientific Computing Research under the Collaboratory on Mathematics and Physics-Informed Learning Machines for Multiscale and Multiphysics Problems (PhILMs) project. SNL is a multimission laboratory managed and operated by National Technology and Engineering Solutions of Sandia, LLC., a wholly owned subsidiary of Honeywell International, Inc., for the U.S. Department of Energy's National Nuclear Security Administration under contract {DE-NA0003525}. This paper, SAND2022-0834, describes objective technical results and analysis. Any subjective views or opinions that might be expressed in this paper do not necessarily represent the views of the U.S. Department of Energy or the United States Government.

\appendix\section{One-dimensional inner quadrature weights}\label{sec:1d-weights}

In this appendix, in order to verify the assumptions on the optimization-based quadrature weights employed in Section \ref{proof_coercivity}, we derive explicit expressions for the optimization-based inner quadrature weights in a one-dimensional setting and for the constant kernel function $\gamma(x,y)=\frac{\zeta}{\delta^3}$ defined in \eqref{kernel_form1}.
Recall from Section \ref{quadrature_discretization} that the inner quadrature points are positioned in $\mathscr{H}(x,\delta)$ according to

\begin{equation}\label{1d_quadpoints_x}
    x_{qp}=x_q+(2k-\sgn(k))\frac{\overline{h}}{2}, \;\;\;\; -\overline{N}_{qp,\delta}\leq k\leq \overline{N}_{qp,\delta}
\end{equation}
with $k\in\mathbb{Z}\setminus\left\{0\right\}$ and 

\begin{equation}\label{1d_quadpoints_hbar}
    \overline{h}=\frac{\delta}{\overline{N}_{qp,\delta}}=\frac{m}{\overline{N}_{qp,\delta}}h=\upsilon h,
\end{equation}
where we defined $\upsilon=m/\overline{N}_{qp,\delta}$. Following the procedure outlined in Section \ref{sec:GMLS_construction}, we have

\begin{equation}\label{1d_quadpoints_B}
\begin{aligned}
    \mathbf{B}&=\frac{\zeta}{\delta^3}\begin{bmatrix}
\left(x_q-x_{-\overline{N}_{qp,\delta}}\right)^2 & \ldots & \left(x_q-x_{i}\right)^2 & \ldots & \left(x_q-x_{\overline{N}_{qp,\delta}}\right)^2
\end{bmatrix}\\
&=\frac{\zeta \overline{h}^2}{4\delta^3}\left[\begin{matrix}
\left(-2N_{qp,\delta}-\sgn(-N_{qp,\delta})\right)^2 & \ldots & \left(2i-\sgn(i)\right)^2\end{matrix}\right. \\ &\left.\begin{matrix} & \ldots & \left(2N_{qp,\delta}-\sgn(N_{qp,\delta})\right)^2\end{matrix}\right]
\end{aligned}
\end{equation}
and

\begin{equation}\label{1d_quadpoints_S}
\begin{aligned}
    S&=\mathbf{B}\mathbf{B}^T=\frac{\zeta^2 \overline{h}^4}{16\delta^6}\sum_{\substack{k=-\overline{N}_{qp,\delta}\\
    k\neq 0}}^{\overline{N}_{qp,\delta}}\left[2(-k)-\sgn(k)\right]^4\\&=\frac{\zeta^2 \overline{h}^4}{16\delta^6}\left[\frac{2}{15}\left(7\overline{N}_{qp,\delta}-40\overline{N}_{qp,\delta}^3+48\overline{N}_{qp,\delta}^5\right)\right],
\end{aligned}
\end{equation}
which leads to 

\begin{equation}\label{1d_quadpoints_S_inverse}
\begin{aligned}
    S^{-1}&=\frac{16\delta^6}{\zeta^2 \overline{h}^4}\frac{1}{\left[\frac{2}{15}\left(7\overline{N}_{qp,\delta}-40\overline{N}_{qp,\delta}^3+48\overline{N}_{qp,\delta}^5\right)\right]}.
\end{aligned}
\end{equation}
For the choice of $\mathbf{V}_h$ defined in (\ref{Vh_exact_constraints}),

\begin{equation}\label{1d_quadpoints_g}
\begin{aligned}
    g=\frac{2}{3}\zeta.
\end{aligned}
\end{equation}
Therefore, from Eqs. (\ref{gmls_final_equation}),  (\ref{1d_quadpoints_x}), (\ref{1d_quadpoints_B}), (\ref{1d_quadpoints_S_inverse}), and (\ref{1d_quadpoints_g})

\begin{equation}\label{1d_quadpoints_final_weights}
    \begin{aligned}
      \boldsymbol{\omega}&=\mathbf{B}^\mathrm{T}{S}^{-1}{g}\\
      &=\frac{\zeta \overline{h}^2}{4\delta^3}\begin{bmatrix}
\left(-2\overline{N}_{qp,\delta}-\sgn(-\overline{N}_{qp,\delta})\right)^2 \\ \ldots \\ \left(2i-\sgn(i)\right)^2 \\ \ldots \\ \left(2\overline{N}_{qp,\delta}-\sgn(\overline{N}_{qp,\delta})\right)^2
\end{bmatrix}\\&\phantom{=}\frac{16\delta^6}{\zeta^2 \overline{h}^4}\frac{1}{\left[\frac{2}{15}\left(7\overline{N}_{qp,\delta}-40\overline{N}_{qp,\delta}^3+48\overline{N}_{qp,\delta}^5\right)\right]}\frac{2}{3}\zeta\\
&=\frac{8\delta^3}{3\overline{h}^2}\frac{1}{\left[\frac{2}{15}\left(7\overline{N}_{qp,\delta}-40\overline{N}_{qp,\delta}^3+48\overline{N}_{qp,\delta}^5\right)\right]}\begin{bmatrix}
\left(-2\overline{N}_{qp,\delta}-\sgn(-\overline{N}_{qp,\delta})\right)^2 \\ \ldots \\ \left(2i-\sgn(i)\right)^2 \\ \ldots \\ \left(2\overline{N}_{qp,\delta}-\sgn(\overline{N}_{qp,\delta})\right)^2
\end{bmatrix}\\
&=\frac{8\delta\overline{N}_{qp,\delta}}{\left[\frac{2}{15}\left(7\overline{N}_{qp,\delta}-40\overline{N}_{qp,\delta}^3+48\overline{N}_{qp,\delta}^5\right)\right]}\begin{bmatrix}
\left(-2\overline{N}_{qp,\delta}-\sgn(-\overline{N}_{qp,\delta})\right)^2 \\ \ldots \\ \left(2i-\sgn(i)\right)^2 \\ \ldots \\ \left(2\overline{N}_{qp,\delta}-\sgn(\overline{N}_{qp,\delta})\right)^2
\end{bmatrix}\\
&=\frac{8\delta}{\left[\frac{2}{15}\left(7-40\overline{N}_{qp,\delta}^2+48\overline{N}_{qp,\delta}^4\right)\right]}\begin{bmatrix}
\left(-2\overline{N}_{qp,\delta}-\sgn(-\overline{N}_{qp,\delta})\right)^2 \\ \ldots \\ \left(2i-\sgn(i)\right)^2 \\ \ldots \\ \left(2\overline{N}_{qp,\delta}-\sgn(\overline{N}_{qp,\delta})\right)^2
\end{bmatrix}.
    \end{aligned}
\end{equation}
It has to be noted that, for $\overline{N}_{qp,\delta}\in\mathbb{N}$, $\omega_k>0$, $\forall k$. Therefore, in this case, all the quadrature weights are positive and there exists a generic constant $C > 0$ independent of $h$ and $\delta$, such that $C\delta < \omega_\text{min}$, with $\omega_\text{min}=\min\limits_{k}\{\omega_k\}$.





\bibliographystyle{model1-num-names}
\bibliography{sample.bib}







\end{document}